\documentclass[a4paper,11pt]{amsart}

\usepackage[utf8]{inputenc}
\usepackage[T1]{fontenc}
\usepackage{amsfonts}
\usepackage{amsthm}
\usepackage{amsmath}
\usepackage[english]{babel}
\usepackage[all]{xy}

\usepackage{graphicx}
\usepackage{amscd}
\usepackage{latexsym}
\usepackage{hyperref}
\usepackage{mathrsfs}

\usepackage{color}

\usepackage{calc}

\theoremstyle{definition}
 \newtheorem{defi}{Definition}[section]

\theoremstyle{remark}
 \newtheorem{remark}[defi]{Remark}
 \newtheorem{exam}[defi]{Example}

\theoremstyle{plain}
\newtheorem{theo}[defi]{Theorem}
\newtheorem{prop}[defi]{Proposition}
\newtheorem{cor}[defi]{Corollary}
\newtheorem{lemma}[defi]{Lemma}

\newtheorem*{acknow}{Acknowledgements}

\newcommand{\zz}{\mathbb{Z}}
\newcommand{\qq}{\mathbb{Q}}
\newcommand{\rr}{\mathbb{R}}
\newcommand{\cc}{\mathbb{C}}

\newcommand{\Z}[1]{\zz_{#1}}

\newcommand{\abs}[1]{\left\vert #1 \right\vert}

\newcommand{\red}{/\!\!/}

\newcommand{\A}{\mathscr{A}}

\newcommand{\Mg}{\mathscr{M}^\mathfrak{g}}

\newcommand{\G}{\mathscr{G}}
\newcommand{\Gc}{\mathscr{G}^c}

\newcommand{\N}{\mathscr{N}}
\newcommand{\Nc}{\mathscr{N}^c}

\newcommand{\I}{\mathcal{I}}

\DeclareMathOperator{\supp}{supp}

\newcommand{\Symp}{\textbf{Symp}}
\newcommand{\Cob}{\textbf{Cob}}  
\newcommand{\Cobelem}{\textbf{Cobelem}}

\newcommand{\arnaque}{\begin{flushright}
$\Box$
\end{flushright}}

\title[HSI: twisting, connected sums, and Dehn surgery]{Symplectic Instanton Homology: twisting, connected sums, and Dehn surgery}
\author[Guillem Cazassus]{Guillem Cazassus}

\begin{document}

\begin{abstract}We define a twisted version of Manolescu and Woodward's Symplectic Instanton homology, prove that this invariant fits into the framework of Wehrheim and Woodward's Floer Field theory, and describe its behaviour for connected sum and Dehn surgery.
\end{abstract}

\maketitle
\tableofcontents

\section{Introduction}\label{sec:intro}

Instanton homology groups are graded abelian groups associated to  integral homology 3-spheres   introduced by Floer in \cite{Floerinst}. They form natural receptacles for  analogs of Donaldson invariants for 4-manifolds with boundary, and allow one  to compute Donaldson invariants by cutting a closed 4-manifold along 3-manifolds. Indeed,  they satisfy axioms  of a (3+1)-topological quantum field theory in the sense of Atiyah (with the exception that the groups are not defined for all 3-manifolds).

Likewise, Atiyah suggested in \cite{Atiyahfloer} the following alternative procedure for defining  these invariants by cutting 3-manifolds along surfaces, called the Atiyah-Floer conjecture. By stretching a 3-manifold  $Y$ along a Heegaard surface $\Sigma$ dividing $Y$ into two handlebodies $H_0$ and $H_1$, he suggested that these invariants can be computed by counting holomorphic discs inside a symplectic manifold $M(\Sigma)$ associated to the surface $\Sigma$, with boundary conditions in two Lagrangian submanifolds $L(H_0), L(H_1) \subset M(\Sigma)$ associated to the handlebodies. The symplectic manifold suggested by this procedure is the moduli space of flat connections on the trivial $SU(2)$-principal bundle over $\Sigma$ introduced in \cite{AtiyahBott}, and the Lagrangian submanifolds correspond to connections that extend flatly to $H_0$, resp. $H_1$.

Although the symplectic manifold $M(\Sigma)$ is not smooth (which is one of the main difficulties in defining Floer homology inside it), Huebschmann and Jeffrey proved that it can be realized as the symplectic quotient of a smooth finite dimensional $SU(2)$-Hamiltonian  manifold, \cite{jeffrey}, \cite{huebschmann},  called the extended moduli space. Manolescu and Woodward then managed to define homology groups,  replacing the symplectic manifold appearing in Atiyah's suggestion by an open subset of this moduli space. These groups, called Symplectic Instanton homology (HSI) \cite{MW}, are defined for every closed oriented three-manifold, and their isomorphism type is a topological invariant.

As a variant of the Atiyah-Floer conjecture, Manolescu and Woodward conjecture that these groups (with $\qq$-coefficients) are isomorphic to a variation of Instanton homology defined by Donaldson, \cite[Conjecture 1.2]{MW}, and ask whether or not the same holds with the hat version of Heegaard Floer homology.

At the same time, Wehrheim and Woodward developed a general framework, called Floer field theory, for dealing with symplectic constructions similar with Atiyah's procedure, and provided a general criterion for these to give topological invariants: the assignment 
\begin{align*}
\Sigma &\mapsto M(\Sigma), \\ H_i &\mapsto L(H_i)
\end{align*}
should be functorial in a certain sense, see \cite{WWfft,Wehrheimphilo}.

The goal of the present paper is to investigate the effect of connected sums and Dehn surgery on symplectic instanton homology. Regarding the second operation, in similar theories, one usually has long exact sequences relating the invariants associated to a ``surgery triad'' (see Definition~\ref{def:triad}), for example the $0,1, \infty$ -surgeries along a framed knot: see \cite{Floertri} for instanton homology,  \cite{OSzholodisk2} for Heegaard Floer theory, and \cite{KMOSzchir} for monopole homology. 

However, the existence of such an exact sequence in instanton homology must face the following problem: three manifolds forming a triad can never be simultaneously integral homology spheres. Floer overcame this issue by using a nontrivial $SO(3)$-bundle over the $S^2\times S^1$-homology manifold appearing in the triad. Even though HSI homology is well-defined for every closed oriented 3-manifold, we will see that the same phenomena also appears in this context. Following a suggestion by Chris Woodward, we  introduce a ``twisted'' version $HSI(Y,c)$ of this invariant, associated to a 3-manifold $Y$ endowed with a class $c$ in $H_1(Y,\Z{2})$, or equivalently an isomorphism class of $SO(3)$-bundles over $Y$.

\subsection{Statement of the main results} 
Concerning connected sum, we obtain the following Künneth formula: 
\begin{theo}(Künneth formula for connected sum)\label{sommecnx} Let  $Y$ and $Y'$ be two closed oriented 3-manifolds, and $c, c'$ two classes in $H_1(Y;\Z{2})$ and $H_1(Y';\Z{2})$ respectively. Then,
\begin{align*}
HSI(Y \# Y' , c+ c') \simeq & HSI(Y, c) \otimes HSI(Y' , c' )  \\
 & \oplus \mathrm{Tor}(HSI(Y, c) , HSI(Y' , c' ))[-1].
\end{align*}
\end{theo} 

Concerning Dehn surgery, recall first the definition of a surgery triad: 
\begin{defi}\label{def:triad} A \emph{surgery triad} is a triple of 3-manifolds $Y_\alpha$, $Y_\beta$ and $Y_\gamma$ obtained from a compact oriented 3-manifold $Y$ with genus one boundary by gluing a solid torus along the  boundary, identifying the meridian with respectively  three simple curves  $\alpha$, $\beta$ and $\gamma$, such that $\alpha. \beta = \beta. \gamma =\gamma. \alpha = -1$.
\end{defi}
Our exact sequence can then be stated as:
\begin{theo}[Surgery exact sequence]\label{suiteintro}\label{trianglechir} Let $(Y_\alpha , Y_\beta, Y_\gamma)$ be a surgery triad obtained from  $Y$ as in the previous definition, $c\in H_1(Y;\Z{2})$, and for $\delta \in \lbrace \alpha,\beta,\gamma\rbrace$, $c_\delta \in H_1(Y_\delta;\Z{2})$ the  class induced from  $c$ by the inclusions. Let also  $k_\alpha \in H_1(Y_\alpha;\Z{2})$ be the  class corresponding to the core of the solid torus. Then, there exists a long exact sequence:
\[ \cdots\rightarrow HSI(Y_\alpha ,c_\alpha+ k_\alpha) \rightarrow HSI(Y_\beta,c_\beta) \rightarrow HSI(Y_\gamma,c_\gamma)\rightarrow \cdots . \]
\end{theo}

\begin{remark}As in Heegaard Floer theory, one can define maps associated to four-dimensional cobordisms, and show that two of the three mophisms appearing in this sequence are associated to the canonical cobordisms, with appropriate cohomology classes. This has been done (in French) in the author's PhD \cite[Chapter 6]{these}, and will be the object of a forthcoming paper \cite{cobordisms}, which should also address the naturality issue.
\end{remark}

\subsection{Applications}

The exact sequence of Theorem~\ref{suiteintro}, together with the knowledge of the Euler characteristic of symplectic instanton homology and an observation from Ozsv{\'a}th and Szab{\'o}, allow one to compute the HSI groups of several manifolds. We will present some of them in section~\ref{sec:applic}, including plumbings, surgeries on knots, and branched double covers.

\subsection{Organisation of the paper}

In section~\ref{sec:fft}, we review Wehrheim and Woodward's Floer Field theory and adapt it to the framework of HSI. In section~\ref{sec:hsi} we will build the twisted symplectic instanton homology groups $HSI(Y,c)$ within this framework. In section~\ref{sec:firstprop} we give their first properties: we prove in section~\ref{sec:compusplitting} that the definition agrees with the one we  outline in the next paragraph, and prove Theorem~\ref{sommecnx} in section~\ref{sec:connectedsum}. Section~\ref{sec:surgery} is devoted to Dehn surgery, we prove Theorem~\ref{trianglechir} in section~\ref{sec:proofsurgeryexactseq}, and provide applications in section~\ref{sec:applic}.

\subsection{Outline of the construction of $HSI(Y,c)$, and sketch of the proofs}\label{sec:outlineproofs}

We first outline the construction of the twisted version from a Heegaard splitting, without using Wehrheim and Woodward's theory. Let $Y = H_0 \cup_\Sigma H_1 $ be a Heegaard splitting, and $C_0$, $C_1$ two knots inside $H_0$ and $H_1$ respectively, such that the class of their union in $H_1(Y;\Z{2})$ equals $c$. Let $\Sigma'$ be the surface with boundary obtained by removing a disc to $\Sigma$, and $*\in \partial\Sigma'$ a base point. The open part of Huebschmann and Jeffrey's extended moduli space   $\N(\Sigma')$ that will be involved in the construction admits the following description: 
\[\N(\Sigma') = \lbrace  \rho \in \mathrm{Hom}(\pi_1(\Sigma',*) , SU(2) ):  \rho(\partial \Sigma')   \neq -I \rbrace. \]
By choosing a base of the free group $\pi_1(\Sigma',*)$, this space can be realized as an open subset of $SU(2)^{2g}$, where $g$ is the genus of $\Sigma$. It also admits a natural symplectic structure, for which the following sub-manifolds  are  Lagrangian:
\begin{align*}
 L_0 &= \lbrace  \rho \circ i_{0,*}: \rho  \in \mathrm{Hom}(\pi_1(H_0\setminus C_0,*), SU(2) ) ,  \rho(\mu_0)  = -I \rbrace  \\
 L_1 &= \lbrace  \rho \circ i_{1,*}: \rho  \in \mathrm{Hom}(\pi_1(H_1\setminus C_1,*), SU(2) ) , \rho(\mu_1)  = -I \rbrace ,
\end{align*}
with $i_{0,*}$ and $i_{1,*}$ induced by the inclusions, and $\mu_0$, $\mu_1$  meridians of $C_0$ and $C_1$ respectively.

The group  $HSI(Y,c)$ can then be defined as the  Lagrangian Floer homology $HF(L_0,L_1)$. 

However, since $\N(\Sigma')$ is noncompact, the fact that Lagrangian Floer homology can be defined inside it is not immediate at all. Manolescu and Woodward manage to do so after a quite involved construction. They first compactify $\N(\Sigma')$ by a symplectic cutting, and obtain a compact moduli space $\Nc(\Sigma') = \N(\Sigma')\cup R$, with a symplectic hypersurface $R$ and a monotone 2-form which degenerates at $R$. By performing another symplectic cutting, they  find a symplectic form on $\Nc(\Sigma')$ which is not monotone anymore, and after some interplay with both forms, they define Floer homology relatively to $R$, which corresponds to Floer homology inside $\N(\Sigma')$.

In order to be able to use Wehrheim and Woodward's theory of quilts, we will define this invariant directly in the Floer field theory framework, and then prove in Proposition~\ref{calculscindement} that its definition agrees with the one we just outlined. This will have two advantages: first, from this definition, the proof of Theorem~\ref{sommecnx} is a simple application of K\"unneth formula, although it is unclear (at least for the author) how to prove it from the Heegaard splitting definition. Second, when proving the surgery exact sequence, it will be convenient to work directly inside the moduli space of the punctured torus, and deal with generalized Dehn twists rather than fibered ones as in \cite{WWtriangle}.

The proof of Theorem~\ref{trianglechir} will follow from a generalization of Seidel's long exact sequence, and from the fact that a Dehn twist on a punctured torus induces a generalized Dehn twist on its extended moduli space.

\begin{acknow} I would like to warmly thank my advisors Paolo Ghiggini and Michel Boileau for suggesting me this problem and for their constant support. I also would like to thank Chris Woodward for suggesting the definition of the twisted version, Christian Blanchet for suggesting the application to alternating links, and Jean-François Barraud, Frédéric Bourgeois, Laurent Charles, Julien Marché and Raphael Zentner, for their interest, and helpful conversations.
\end{acknow}

\section{Floer Field Theory}\label{sec:fft}

\subsection{Outline}\label{sec:outline}

According to Wehrheim and Woodward \cite{WWfft}, a \emph{Floer field theory} is a functor from a category of (2+1)-cobordisms, to a variation of Weinstein's symplectic category, whose objects are symplectic manifolds, and morphisms (equivalence classes of) sequences of Lagrangian correspondences (see Definition~\ref{def:lagcorresp}).

Such a functor being given, one can associate a topological invariant to a closed oriented 3-manifold $Y$ in the following way: after removing two 3-balls to $Y$, one obtains a cobordism from the 2-sphere to itself, the functor applied to this cobordism gives rise to a sequence of Lagrangian correspondences having the same source and target symplectic manifolds, and to such a sequence (with extra technical assumptions) one can associate a homology group called quilted Floer homology, which is a generalization of Lagrangian Floer homology.  

We will see that HSI groups fit into this framework, as already suggested in Manolescu and Woodward's proof of their stabilisation invariance. Yet, some slight modifications will be necessary, mainly for three reasons:

\begin{itemize}

\item The moduli spaces are associated to surfaces with one boundary component: a  disc will have to be removed to a closed surface, and similarly, a tube connecting two such discs will have to be removed to a cobordism. Since the Lagrangian correspondence obtained will depend on the choice of this tube (as in Example~\ref{exemchgtchemin}), these should be incorporated in the cobordism category. Rather than being seen as an inconvenient, this phenomena might be used to define invariants for knots and sutured manifolds. Moreover, Lagrangian corresponces will also depend on a parametrization of the tube (see Example~\ref{exemreparam}), of which the category should also keep track.

\item The twisting homology  class in $H_1(Y;\Z{2})$ should also be incorporated in the cobordism. Although this class corresponds to a second Stiefel-Whitney class, we prefer working with homology classes to avoid using relative cohomology classes.

\item The target symplectic category will also have to be more complicated in order to be able to define quilted Floer homology, for the reasons described in section~\ref{sec:outlineproofs} : symplectic manifolds will come equipped with a hypersurface, and two 2-forms, which should satisfy extra technical assumptions as in \cite[Assumption 2.5]{MW}.

\end{itemize}

\subsection{Quilts, the symplectic category, and Floer homology}
\label{sec:quiltsetc}

\begin{defi}\label{def:lagcorresp}
A \emph{Lagrangian correspondence} between two symplectic ma\-nifolds  $M$ and $M'$ is a Lagrangian submanifold $L\subset M^- \times M'$, where $M^-$ denotes the manifold $M$ endowed with the opposite of its symplectic form.
\end{defi}

This kind of correspondences, sometimes called canonical relations, appears frequently  in symplectic geometry: a diffeomorphism between two symplectic manifolds is a symplectomorphism if and only if its graph is a Lagrangian correspondence. Moreover, given a Hamiltonian group action of a group $G$ on a symplectic manifold $M$ which is free on the zero level on the moment map, the zero level of the moment map induces a Lagrangian correspondence from $M$ to $M\red G$.

Recall the following definition from Wehrheim and Woodward:
\begin{defi}
A \emph{generalized Lagrangian correspondence} between two symplectic manifolds $M$ and $M'$ consists of intermediate symplectic manifolds $M_1$, $M_2$, ..., $M_{k-1}$, and a sequence of Lagrangian correspondences, for $0\leq i \leq k-1$,  $L_{i(i+1)} \subset M_i ^- \times M_{i+1}$, with $M_0 = M$ and $M_k = M'$. Such a sequence will be denoted $\underline{L}$:
\[\underline{L} = \left(  \xymatrix{   M_0 \ar[r]^{L_{01}} &  M_1 \ar[r]^{L_{12}} &  M_2 \ar[r]^{L_{23}} & \cdots \ar[r]^{L_{(k-1)k}} & M_k } \right), \]
The integer $k$ will be referred to as the \emph{length of $\underline{L}$}. If $L$ (resp. $\underline{L}$) denotes a (resp. generalized) Lagrangian correspondence  from $M$  to $M'$, we will denote $L^T$ (resp. $\underline{L}^T$) the correspondence from $M'$ to $M$, obtained by reversing the arrows.
\end{defi}

We will denote $pt$ the symplectic manifold consisting of one point. A  Lagrangian correspondence from $pt$ to $M$ is then  simply a Lagrangian submanifold of $M$.

If $\underline{L}$ is a generalized Lagrangian correspondence from  $pt$ to $pt$, the quilted Floer homology of $\underline{L}$ can be defined (whenever this is possible) as the Lagrangian Floer homology
\[ HF(\underline{L}) = HF( L_{01}\times L_{23} \times \cdots , L_{12}\times L_{34} \times \cdots ), \]
where the ambient symplectic manifold is the product of all the manifolds $M_0^- \times M_1 \times M_2^- \times \cdots $. When $L_{01}\times L_{23} \times \cdots $ and $L_{12}\times L_{34} \times \cdots $ intersect transversely, the corresponding chain complex is generated by the set of  \emph{generalized intersection points}
\[ \I(\underline{L}) =\lbrace (x_0, \cdots , x_k) ~|~ \forall i, (x_i,x_{i+1}) \in L_{i(i+1)}   \rbrace, \]
and its differential counts index 1 Floer trajectories, which can alternatively be viewed as ``quilted strips''. We now recall  the definition of quilted surface and pseudo-holomorphic quilt, which are the key objects involved in Wehrheim and Woodward's theory:

\begin{defi}A \emph{quilted surface} $\underline{S}$ consists of: 
\begin{enumerate}
\item[$(i)$] a collection of Riemann surfaces $\underline{S} = (S_k)_{k=1\cdots m}$, called patches, and endowed with complex structures $j_k$. The boundary components of $S_k$ will be indexed by a set $\mathcal{B}(S_k)$: $\partial S_k = \bigcup_{b\in\mathcal{B}(S_k) }{I_{k,b}}$.
\item[$(ii)$] a collection $\mathcal{S}$ of \emph{seams}, consisting of pairwise disjoints two-element subsets:
$ \sigma \subset \bigcup_{k=1}^{m}\bigcup_{b\in\mathcal{B}(S_k)} I_{k,b}$, and for each $\sigma = \lbrace  I_{k,b}, I_{k',b'}\rbrace$, a real analytic diffeomorphism $\varphi_\sigma \colon I_{k,b} \rightarrow I_{k',b'}$.
\end{enumerate}\end{defi}

\begin{defi} Let $\underline{S}$ be a quilted surface as before, and $\underline{M} = (M_k)_{k=1\cdots m}$ a collection of symplectic manifolds,  one for each patch, and  \[\underline{L} = \left(  L_\sigma \subset M_k ^- \times M_{k'} , L_{k,b} \subset M_k\right)\] a collection of Lagrangian correspondences, one associated to each seam $\sigma = \lbrace  I_{k,b}, I_{k',b'}\rbrace$, and Lagrangian submanifolds, one for each $I_{k,b}$ which is not contained in any seam. A \emph{(pseudo-holomorphic) quilt} $  \underline{u}\colon   \underline{S} \rightarrow \left(\underline{M}, \underline{L} \right) $ is a collection of (pseudo-holomorphic, provided the $M_k$ are endowed with almost-complex structures) maps   $u_i\colon S_i\rightarrow M_i$ satisfying the following seam and boundary conditions:
 \[ (u_k(x), u_k'(\varphi_\sigma (x)) ) \in L_\sigma,  x \in I_{k,b}  ,\]
 \[ u_k(x) \in L_{k,b}, x \in I_{k,b} .\]
\end{defi}
These data can be summarized in a diagram as in figure~\ref{cylinder}:

\begin{figure}[!h]
    \centering
    \def\svgwidth{.65\textwidth}
    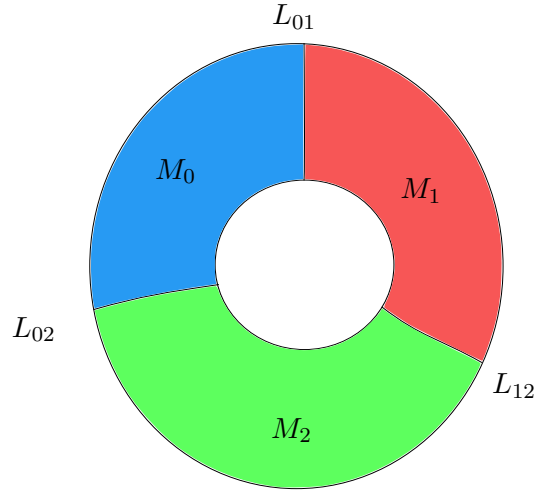
      \caption{A quilted cylinder.}
      \label{cylinder}
\end{figure}

\begin{remark}
\begin{enumerate}
\item We will refer to $I_{k,b}$ as a boundary of $\underline{S}$ if it is not contained in any seam.
\item We will sometimes identify $\underline{S}$ and the  surface obtained by gluing all the patches together along the seams.
\end{enumerate}
\end{remark}

 Floer trajectories involved in defining the differential can then be seen as quilted strips as in figure~\ref{bande}.

\begin{figure}[!h]
    \centering
    \def\svgwidth{.65\textwidth}
   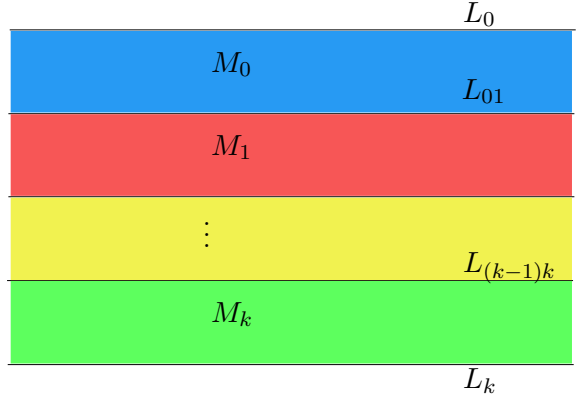
      \caption{ A quilted strip involved in the differential of quilted Floer homology.}
      \label{bande}
\end{figure}

Weinstein suggested in \cite{weinsteincat} that Lagrangian correspondences  should be seen as morphisms of  a category, for which the composition should be defined in the following way:

\begin{defi}[Geometric composition]

Let $M_0$, $M_1$, $M_2$ be three symplectic manifolds, and $L_{01} \subset M_0\times M_1$, $L_{12} \subset M_1\times M_2$ be Lagrangian correspondences. The \emph{geometric composition} of $L_{01}$ with $L_{12}$ is the subset:
\[L_{01} \circ L_{12} = \pi_{02}(L_{01}\times M_2 \cap M_0\times L_{12}),\]
where $\pi_{02}$ denotes the projection
\[\pi_{02}\colon M_0\times M_1\times M_2 \rightarrow M_0\times  M_2 .\]
\end{defi}

\begin{remark}This composition generalises the composition of symplectomorphisms, and the diagonal $\Delta_{M} = \lbrace (x,x)~|~ x\in M \rbrace$ plays the role of the identity.
\end{remark}

Unfortunately, the resulting correspondence may not be smooth, therefore this composition is not always a Lagrangian correspondence. To overcome this difficulty,  Wehrheim and Woodward only allow this operation whenever the following criterion is satisfied:

\begin{defi}[Embedded geometric composition]
A geometric composition  $L_{01} \circ L_{12}$ is called \emph{embedded} when:

\begin{itemize}
\item $L_{01} \times M_2$ and $M_0 \times L_{12}$ intersect transversally,

\item $\pi_{02}$ induces an embedding of $L_{01}\times M_2 \cap M_0\times L_{12}$ in $M_0\times  M_2$.
\end{itemize}

\end{defi}

When this criterion is fulfilled,  not only  $L_{01} \circ L_{12}$ is a Lagrangian correspondence, but quilted Floer homology also behaves well in this situation: we will see in Theorem~\ref{compogeom} that, assuming $L_{01} \circ L_{12}$ is embedded, and some extra hypothesis,
\[ HF( \cdots,  L_i, L_{i+1}, \cdots ) \simeq HF( \cdots,  L_i \circ L_{i+1}, \cdots ). \]

Wehrheim and Woodward then define the extended symplectic category  $ \mathrm{Symp}_\tau^ \#$ (see \cite[Def. 3.1.7]{WWfft}) as the category whose objects are symplectic manifolds (with some extra monotonicity  conditions), and  morphisms equivalence classes of generalized Lagrangian correspondences   (also satisfying extra conditions), where the equivalence relation is generated by:
\[ ( \cdots,  L_i, L_{i+1}, \cdots ) \sim ( \cdots,  L_i \circ L_{i+1}, \cdots ),\]
provided $L_i \circ L_{i+1}$ is embedded, and some extra  conditions ensuring Floer homology is well-defined.

\begin{remark}It will follow from Cerf theory that the  generalized Lagrangian correspondences we will construct are equivalent to length 2 generalized Lagrangian correspondences. This facts holds for every generalized Lagrangian correspondence, as observed by  Weinstein in \cite{Weinsteinremark}. 
\end{remark}

In order to fit in the framework of \cite[Assumption 2.5]{MW} required  for defining symplectic instanton homology, we will adapt the definition of  $ \mathrm{Symp}_\tau^ \#$ for defining \Symp\ (c.f. Definition~\ref{defcat}). We first give here some definitions required in order to do so. 

As the symplectic manifolds we will consider will be equipped with symplectic hypersurfaces, the Lagrangian correspondences shall satisfy the following compatibility condition (\cite[Def. 6.2]{MW}):

\begin{defi}[Lagrangian correspondences  compatible with  a pair of hypersurfaces] \label{correspcompat}
 Let  $M_0$, $M_1$ be two  symplectic manifolds, and $R_{0}\subset M_0$, $R_{1} \subset M_1 $ two  symplectic hypersurfaces. A Lagrangian correspondence $L_{01} \subset M_0\times M_1$ is \emph{compatible with  the pair $(R_0,R_1)$} if $L_{01}$ intersects $R_0 \times M_1$ and $M_0 \times R_1$ transversally, and these two intersections coincide (and are equal to $(R_0 \times R_1 ) \cap L_{01}$).
\end{defi}

In these conditions, one can choose tubular  neighborhoods  $\tau_0$ and $\tau_1$ of $R_0$ and $R_1$ respectively:
\[ \tau_0\colon N_{R_0} \rightarrow M_0,\   \tau_1\colon N_{R_1} \rightarrow M_1, \]
 such that the preimage $ (\tau_0 \times \tau_1)^{-1}(L_{01}) \subset N_{R_0} \times N_{R_1} $ is the graph of a bundle isomorphism 
\[ \varphi \colon (\widetilde{N}_{R_0})|_{(R_0 \times R_1 ) \cap L_{01}} \rightarrow (\widetilde{N}_{R_1})|_{(R_0 \times R_1 ) \cap L_{01}}, \]
where $\widetilde{N}_{R_0}$ is the normal bundle  of $R_0 \times M_1 \subset M_0 \times M_1$, and $\widetilde{N}_{R_1}$ is the  normal bundle of $M_0 \times R_1 \subset M_0 \times M_1$.

Notice that if one of the symplectic manifolds  is a point, a Lagrangian $L$ is compatible with an hypersurface $R$ if and only if $L$ and $R$ are disjoint.

 Let    $\underline{S}$ be a  quilted surface,  $\underline{u}\colon \underline{S}\rightarrow (\underline{M},\underline{L})$ a quilt, and $\underline{R}\subset \underline{M}$ a family of hypersurfaces such that each Lagrangian correspondence associated to a seam is  compatible with  the corresponding hypersurfaces. Recall that $\underline{u}$ and $\underline{R}$ have a well-defined intersection number $\underline{u}. \underline{R}$:

Let $\underline{U}\subset \underline{S}$ be an open neighborhood   of $\underline{u}^{-1}(\underline{R})$ such that the image of each patch $S_i$ is contained in the tubular neighborhoods  $\tau_i$ of $R_i$. Each map $u_i$ can be seen as a section of the complex line bundle $u_i^* N_{R_i}$. All these bundles can be glued together into a bundle over $\underline{U}$ by using the isomorphisms $\varphi$, and the sections $u_i$ glue together to a global section  of this bundle, which is nonzero over the boundary  $\partial \underline{U}$.  The bundle is then trivial over this boundary and extends to a bundle over $\underline{S}$, and the sections extend to global sections, which are non-zero outside $\underline{U}$. The intersection number $\underline{u}\cdot \underline{R}$ is then defined as the Euler number  of this bundle.

The following lemma  is proven in \cite[Lemma 6.4]{MW} when the quilted surface consists of several  parallel strips, its proof adapts to any  quilted surface. 

\begin{lemma}(\cite[Lemma 6.4]{MW}) The intersection number $\underline{u}\cdot \underline{R}$ doesn't change when one perturbs $\underline{u}$ by a homotopy preserving the boundary and seam conditions. Moreover, if $\underline{u}$ is pseudo-holomorphic, $\underline{R}$ almost complex and $\underline{u}$ and $\underline{R}$ intersect transversely, then this number is given by:
\[ \underline{u} \cdot \underline{R} = 
\sum_{j=0}^k 
\# \{ z_j \in \operatorname{int}(S_j) | u_j(z_j)
\in R_j \} + \frac{1}{2} \# \{ z_j \in \partial S_j | u_j(z_j) \in R_j \} .\]
\end{lemma}

\begin{defi}\label{defcat}
We will call $\Symp$ the following category:  

$\bullet$ Its objects are tuples $(M, \omega , \tilde{\omega}  , R, \tilde{J})$ satisfying conditions $(i)$, $(ii)$, $(iii)$, $(iv)$, $(v)$, $(x)$, $(xi)$ and $(xii)$ of  \cite[Assumption 2.5]{MW}, namely:
\begin{enumerate}
\item[$(i)$] $(M,\omega)$ is a compact symplectic manifold.

\item[$(ii)$] $\tilde{\omega}$ is a closed  2-form on $M$.

\item[$(iii)$] The degeneracy locus $R\subset M$ of $\tilde{\omega}$ is a symplectic hypersurface  for $\omega$.

\item[$(iv)$] $\tilde{\omega}$ is $\frac{1}{4}$-monotone, that is $\left[ \tilde{\omega} \right] = \frac{1}{4} c_1(TM) \in H^2(M; \rr)$.

\item[$(v)$] The restrictions of $\tilde{\omega}$ and $\omega$ to $M\setminus R$ define the same cohomology class in  $H^2(M\setminus R; \rr)$.

\item[$(x)$] The minimal Chern number $N_{M\setminus R}$ with respect to $\omega$ is a positive multiple  of 4, so that the minimal Maslov number   $N = 2N_{M\setminus R}$ is a positive  multiple of $8$.

\item[$(xi)$] $\tilde{J}$ is an $\omega$-compa\-tible almost complex structure  on $M$, $\tilde{\omega}$-compa\-tible on $M\setminus R$, and such that $R$ is an  almost complex hypersurface for $\tilde{J}$.

\item[$(xii)$] Every index zero $\tilde{J}$-holomorphic sphere in $M$, necessarily contained in $R$ by monotonicity, has an intersection number with  $R$ equal to a negative multiple  of 2.

\end{enumerate}

$\bullet$ The set of morphisms between two objects consists of strings of elementary morphisms $\underline{L} = (L_{01}, L_{12}, \cdots ) $, modulo an equivalence relation:
\begin{itemize}
\item The elementary morphisms are  correspondences $L_{i(i+1)} \subset  M_i^- \times M_{i+1}$ which are Lagrangian for the monotone forms $\tilde{\omega}_i$, simply con\-nected, $(R_i, R_{i+1})$-compatible in the sense of  Definition~\ref{correspcompat}, such that $L_{i(i+1)} \setminus \left( R_i \times R_{i+1} \right)$ is spin, and such that every pseudo-holomor\-phic disc of $M_i^- \times M_{i+1}$ with boundary in $L_{i(i+1)}$ and zero area has an intersection number with $(R_i, R_{i+1})$ equal to a positive multiple  of $-2$.
\item The equivalence relation on strings of morphisms is generated by the following identification:  $(L_{01},\cdots ,  L_{(i-1)i},L_{i(i+1)}, \cdots )$ is identified with $(L_{01},\cdots  L_{(i-1)i}\circ  L_{i(i+1)}, \cdots )$ whenever the composition of $L_{(i-1)i}$ and $L_{i(i+1)}$ is embedded, simply connected, $(R_{i-1} , R_{i+1} )$-compatible, spin outside $R_{i-1}\times R_{i+1}$, satisfies the above  hypothesis concerning  pseudo-holomorphic discs, and also the following one: every  quilted pseudo-holomorphic cylinder as in figure~\ref{cylinder} of zero area  and with seam conditions in $L_{(i-1)i}$,   $L_{i(i+1)}$ and $L_{(i-1)i}\circ  L_{i(i+1)}$ has an intersection number with $(R_{i-1} , R_i, R_{i+1})$ smaller than $-2$.
\end{itemize}
\end{defi}

\begin{remark}\label{monotonicity} Under these hypotheses, the  generalized Lagrangian correspondences are automatically monotone: if $\underline{x},\underline{y} \in  \I(\underline{L})$ are generalized intersection points, and   $\underline{u}$ denotes a quilted strip with seam  conditions in $\underline{L}$ and with limits $\underline{x},\underline{y}$, then the symplectic area of $\underline{u}$ is $A(\underline{u}) = \frac{1}{8} I(\underline{u}) + c(\underline{x},\underline{y})$, with  $c(\underline{x},\underline{y})$ a number only depending  on the points $\underline{x}$ and $\underline{y}$. This follows from monotonicity of the  symplectic manifolds and from simply connectedness of the correspondences, see \cite[Lemma 2.8]{MW}.
\end{remark}

\subsubsection{Definition of quilted Floer homology }\label{defhomolgiequilted}

Let $\underline{L} = (L_{i(i+1)})_{i = 0\cdots k} $ be a morphism of $\Symp$ of length $k+1$ from $pt$ to $pt$, with intermediate objects $\underline{M} = (M_i, \omega_i , \tilde{\omega}_i, R_i , \tilde{J}_i)_{i = 0\cdots k+1}$, and  $M_0 = M_{k+1} = pt$:

\[\underline{L} = \left(  \xymatrix{   pt\ar[r]^{L_{01}} &  M_1 \ar[r]^{L_{12}} &  M_2 \ar[r]^{L_{23}} & \cdots \ar[r]^{L_{k(k+1)}} & pt } \right). \]

Let $\mathcal{J}(M_i, \mathrm{int}\left\lbrace \omega_i = \tilde{\omega}_i \right\rbrace ,  \tilde{J}_i)$ be the set of $\omega_i$-compatible almost-complex structures on $M_i$ coinciding with  $\tilde{J}_i$ outside $\mathrm{int}\left\lbrace \omega_i = \tilde{\omega}_i \right\rbrace$. Let also \[\mathcal{J}_t(M_i, \mathrm{int}\left\lbrace \omega_i = \tilde{\omega}_i \right\rbrace ,  \tilde{J}_i) = \mathcal{C}^\infty([0,1], \mathcal{J}(M_i, \mathrm{int}\left\lbrace \omega_i = \tilde{\omega}_i \right\rbrace ,  \tilde{J}_i))\] be the set of time-dependent almost-complex structures.

In order to have transverse intersections, we introduce  Hamiltonian perturbations. Let  $\underline{H} = (H_i)_{i = 1\cdots k}$ be  Hamiltonians, $H_i: M_i\times \rr \rightarrow \rr$ with support inside $\mathrm{int}\left\lbrace \omega_i = \tilde{\omega}_i \right\rbrace$. Denote $\varphi_i$ the time 1  flow of $X_{H_i}$ and
\begin{align*}
& \widetilde{L}_{i(i+1)} = \left\lbrace (\varphi_i(x_i),x_{i+1}) ~|~ (x_i,x_{i+1}) \in L_{i(i+1)} \right\rbrace   \\
 & \widetilde{L}_{(0)} = \widetilde{L}_{0} \times \widetilde{L}_{12} \times \cdots \\
& \widetilde{L}_{(1)} = \widetilde{L}_{01} \times \widetilde{L}_{23} \times \cdots.
\end{align*}

Suppose that the generalized intersection points $\I(\underline{L})$ are contained inside the product of the $\mathrm{int}\left\lbrace \omega_i = \tilde{\omega}_i \right\rbrace$,  which will be the case when defining HSI homology. For a generic choice of Hamiltonians, the intersection $\widetilde{L}_{(0)} \cap \widetilde{L}_{(1)}$ is transverse. The finite set  $\widetilde{L}_{(0)} \cap \widetilde{L}_{(1)}$ is in one-to-one correspondence  with  the set of perturbed generalized intersection points  $\I_{\underline{H}}(\underline{L})$ consisting of $k$-tuples $p_i\colon [0,1] \rightarrow M_i$ such that $\frac{\mathrm{d}p_i}{\mathrm{d}t} = X_{H_i}$, and $(p_i(1), p_{i+1}(0) )\in L_{i(i+1)}$. Indeed,  given  $p_i(1) = \varphi_i(p_i(0))$, they correspond to points $(x_1, \cdots , x_k)$ of $\widetilde{L}_{(0)} \cap \widetilde{L}_{(1)},$ with $x_i = p_i(0)$.

Let $\widetilde{\mathcal{M}}(\underline{x},\underline{y})$ be the set of maps $u_i\colon \rr \times [0,1]\to M_i$ such that, with  $s\in\rr$ and $t\in [0,1]$:

\[ \begin{cases} 0 = \partial_s u_i + J_t ( \partial_t u_i - X_{H_i}) \\
\mathrm{lim}_{s\rightarrow - \infty}{u_i(s,t)} = \varphi_i^t (y_i)\\
\mathrm{lim}_{s\rightarrow + \infty}{u_i(s,t)} = \varphi_i^t (x_i)\\
(u_i(s,1), u_{i+1}(s,0) )\in L_{i(i+1)} \\
\underline{u} \cdot \underline{R} = 0 \\
I(\underline{u}) = 1.
\end{cases} \]
The space of quilted  Floer trajectories  is then the quotient \[\mathcal{M}(\underline{x},\underline{y})  = \widetilde{\mathcal{M}}(\underline{x},\underline{y}) /\rr\] by the $s$-reparametrization. For generic choices of  Hamiltonians $\underline{H}$ and almost complex structures $\underline{J}$, it is a finite set, see \cite[Prop. 5.2.1]{WWqfc}.

The Floer complex  is then defined as \[CF(\underline{L},\underline{H}, \underline{J}) = \bigoplus_{\underline{x}\in \I_{\underline{H}}(\underline{L})} \zz \underline{x},\] and endowed with the differential  defined by \[\partial \underline{x} = \sum_{\underline{y}} \# \mathcal{M}(\underline{x},\underline{y}) \underline{y},\] 
where $\# \mathcal{M}(\underline{x},\underline{y}) = \sum_{u\in \mathcal{M}(\underline{x},\underline{y}) }{o(u)}$, with  $o(\underline{u}) = \pm 1$ the orientation of the point $\underline{u}$ constructed in \cite{WWorient} with the unique relative spin structure on $\underline{L}$.
Recall the following result from Manolescu and Woodward that makes Floer homology well-defined for elements of $Hom_{Symp}(pt,pt)$:

\begin{theo}(\cite[Theorem 6.5]{MW}) Let  $\underline{L}$ and $\underline{H}$ be as above. There exists a comeagre subset
\[\mathcal{J}_t^{reg}(M_i, \mathrm{int}\left\lbrace \omega_i = \tilde{\omega}_i \right\rbrace ,  \tilde{J}_i) \subset \mathcal{J}_t(M_i, \mathrm{int}\left\lbrace \omega_i = \tilde{\omega}_i \right\rbrace ,  \tilde{J}_i)\] of regular almost complex structures such that the  differential is finite and satisfies $\partial ^2 = 0$. Therefore, the quilted Floer homology  $HF (\underline{L})$ is well-defined for generic almost complex structures and Hamiltonian perturbations, and is independent on these choices, except eventually on the reference almost complex structure.
\end{theo}

\begin{remark}For the manifolds we will work with, the choice of the reference almost complex structure will not affect the construction since it will be chosen in a contractible space, see \cite[Remark 4.13]{MW}.
\end{remark}

\paragraph{\bf Grading} The hypothesis on the minimal Maslov number allows one to define a relative $\Z{8}$-grading on the chain complex: if $\underline{x}$ and $\underline{y}$ are two generalized intersection points, and $u$, $v$ two quilted Floer trajectories (not necessarily pseudo-holomorphic) connecting them, $I(u) = I(v)$ modulo 8. We then denote $I(\underline{x},\underline{y}) \in \Z{8}$ the common quantity. The differential is then of degree $1$. It follows that $I(\underline{x},\underline{y})$ defines a  relative grading on $HF(\underline{L})$.

Recall then  the following result, which ensures  the invariance of quilted Floer homology under embedded geometric composition:

\begin{theo}(\cite[Theorem 6.7]{MW})\label{compogeom} Let $\underline{L}$ be a generalized Lagrangian correspondence as before. Assume moreover that the geometric composition $L_{i-1,i} \circ  L_{i, i + 1}$ is embedded, simply connected, $(R_{i-1} , R_{i+1})$-compa\-tible, and such that the intersection number of every pseudo-holomorphic quilted cylinder  with  $(R_{i-1} , R_i,  R_{i+1})$ is smaller than $-2$, then $HF (\underline{L})$ is canonically isomorphic to $HF (\cdots L_{i-1,i} \circ L_{i, i + 1} \cdots )$.
\end{theo}

\subsection{Cobordisms with vertical boundaries and connected Cerf theory}\sectionmark{Cerf theory}\label{sec:connectedcerf}

We start by defining the cobordism category that will be the source of the Floer field theory functor.

\begin{defi}[Category of cobordisms with vertical boundaries]\label{chemincob}
We will call \emph{category  of cobordisms with vertical boundaries, with degree one $\Z{2}$-homology class}, and will denote it $\Cob$, the category whose:

\begin{itemize}

\item  objects are pairs $(\Sigma,p)$, where $\Sigma$ is a compact connected oriented surface, with connected boundary, and $p\colon \rr/\zz \rightarrow \partial\Sigma$ is a diffeomorphism (parametrization).

\item  morphisms from $(\Sigma_0,p_0)$ to $(\Sigma_1,p_1)$ are diffeomorphism classes  of 5-tuples $(W, \pi_{\Sigma_0},  \pi_{\Sigma_1}, p,  c)$, where $W$ is a compact oriented 3-manifold with boundary, $\pi_{\Sigma_0}$,  $\pi_{\Sigma_1}$ and $p$ are embeddings of $\Sigma_0$, $\Sigma_1$ and $\rr/\zz \times [0,1]$ into $\partial W$, the first  reversing the orientation, the two others preserving it, and such that:
\begin{itemize}
\item $\partial W = \pi_{\Sigma_0}(\Sigma_0)\cup  \pi_{\Sigma_1}(\Sigma_1)\cup  p(\rr/\zz \times [0,1])$,
\item $\pi_{\Sigma_0}(\Sigma_0)$  and $\pi_{\Sigma_1}(\Sigma_1)$ are disjoint,
\item for $i=0,1$, $p(s,i) = \pi_{\Sigma_i}(p_i(s))$, and \[\pi_{\Sigma_i}(\Sigma_i)\cap p(\rr/\zz \times [0,1])  = \pi_{\Sigma_i}(p_i(\rr/\zz)) = p(\rr/\zz \times \lbrace i\rbrace),\]
\item $c \in H_1 (W, \Z{2})$.
\end{itemize}

We will refer to   $p(\rr/\zz \times [0,1]) $ as the vertical part of $\partial W$, and will denote it $\partial^{vert} W$.

Two such 5-tuples $(W, \pi_{\Sigma_0},  \pi_{\Sigma_1}, p,  c)$ and $(W', \pi'_{\Sigma_0},  \pi'_{\Sigma_1}, p',  c')$ will be identified if there exists a diffeomorphism $\varphi \colon W \rightarrow W'$ compatible with  the embeddings and preserving the classes. 

\item  composition of morphisms consists in gluing along the embeddings, and adding the homology classes.
\end{itemize}

\end{defi}

We will not assign a Lagrangian correspondence to any cobordism, but only to certain elementary ones, in the sense of Morse theory. We will see that an arbitrary  cobordism always decomposes into a finite number of such elementary cobordisms, hence giving rise to a sequence of Lagrangian correspondences, namely a  morphism of $\Symp$. The aim of this paragraph is to prepare the proof of the fact that this construction doesn't depend on the decomposition of the cobordism. We will essentially adapt the results of \cite{connectedcerf} to the framework of $\Cob$. Recall its principal one, which is false in dimension $1+1$:

\begin{theo}[\cite{connectedcerf}]\label{GWWtheo}
Let $n\geq 2$,
\begin{enumerate}

\item Every connected $(n+1)$-cobordism  between connected $n$-manifolds  admits a   decomposition into elementary cobordisms such that each  intermediate levels are  connected. Such a  decomposition will be called a \emph{Cerf decomposition}.

\item Given two such decompositions, it is always possible to go from one to the other by a finite number of \emph{Cerf moves}, namely a  diffeomorphism equivalence, a cylinder creation, a cylinder cancellation, a critical point creation, a critical point cancellation or a critical point switch (see \cite{connectedcerf} for the  definitions, or Definition~\ref{mvtcerf}).

\end{enumerate}
\end{theo}

We define an intermediate cobordism category:

\begin{defi}[Category  of elementary cobordisms with vertical boundaries]\label{chemincobelem}

Let $\Cobelem$ stand for the category  whose objects are the same as those of $\Cob$, and whose morphisms are strings of 6-tuples 
\[(W_k, \pi_{\Sigma_k},  \pi_{\Sigma_{k+1}}, p_k, f_k,  c_k),\]
where $(W_k, \pi_{\Sigma_k},  \pi_{\Sigma_{k+1}},p_k, c_k)$ is a cobordism from $(\Sigma_k,p_k)$ to $(\Sigma_{k+1},p_{k+1})$ as in Definition~\ref{chemincob}. The map  $f_k\colon W_k \rightarrow [0,1]$ is a Morse function such that, for $i=0,1$, $f_k^{-1}(i) = \pi_{\Sigma_{k+i}}(\Sigma_{k+i})$, admitting at most one critical point in the interior of $W_k$ and no critical point on $\partial W_k$, and $f(p(s,t)) = t \in [0,1]$. Finally $c_k \in H_1 (W_k, \Z{2}))$. We will  refer to such cobordisms as \emph{elementary}, and  denote a sequence of such cobordisms by:
\[\underline{W} = (W_0, \pi_{\Sigma_0},  \pi_{\Sigma_{1}}, p_0, f_0,  c_0) \odot  \cdots \odot (W_k, \pi_{\Sigma_k},  \pi_{\Sigma_{k+1}}, p_k, f_k,  c_k).\]
Their composition consist in concatenating, and will be denoted $\odot$.
\end{defi}

\begin{remark}\begin{enumerate}
\item When there will be no ambiguity, we will sometimes omit the embeddings and the Morse functions, and write simply  $(W,c)$ instead of $(W, \pi_{\Sigma_0},  \pi_{\Sigma_1}, f, p,  c)$.

\item There is a functor $\Cobelem\rightarrow\Cob$ that doesn't change the objects  and  consists in gluing altogether a string of cobordisms into one cobordism, adding classes, and forgetting Morse function.

\end{enumerate}
\end{remark}

In order to obtain in Proposition~\ref{cerfchemin} a similar result of Theorem~\ref{GWWtheo} for cobordisms with vertical boundaries endowed with a  degree 1 homology class with $\Z{2}$ coefficients, we define similar moves for such cobordisms.

\begin{defi}[Cerf moves] \label{mvtcerf}
If $\underline{W} = W_1 \odot \cdots \odot W_k$ and $\underline{W'} = W_1' \odot \cdots\odot W_l'$ are morphisms of  $\Cobelem$, we will call the replacement of $\underline{W}$ by $\underline{W'}$ a \emph{Cerf move} if one of the following modifications is done: 

\begin{enumerate}
\item[$(i)$] \emph{Diffeomorphism equivalence}: replacing a $W_i =(W, \pi_{\Sigma_0},  \pi_{\Sigma_1}, f, p,  c)$ by $(W', \pi'_{\Sigma_0},  \pi'_{\Sigma_1}$, $ f', p',  c')$, provided there exists  a diffeomorphism $\varphi $ from $ W$ to $ W'$ such that $f'\circ\varphi = f$, $\varphi \circ \pi_{\Sigma_i} = \pi'_{\Sigma_i}$, for $i=0,1$, $p' = \varphi \circ p$, and $c' = \varphi_* c$.
 
\item[$(ii)$] \emph{Cylinder creation, cylinder cancellation}: adding or removing a cobordism $(W, \pi_{\Sigma_0},  \pi_{\Sigma_1}, f, p,  c)$, with: 
\begin{itemize}
\item $W = \Sigma \times [0,1]$,
\item $\pi_{\Sigma_i} = id_{\Sigma} \times \lbrace i \rbrace$,
\item $f(s,t) = t$,
\item $p(s,t) = (p(s),t)$,
\item $c = 0$.
\end{itemize}

\item[$(iii)$] \emph{Critical point creation, critical point cancellation}:  A birth-death pair is a  string \[(W, \pi_{\Sigma_0},  \pi_{\Sigma_1}, f, z,  0) \odot(W', \pi_{\Sigma_1}',  \pi_{\Sigma_0}', f', z', 0),\] with $f$ and $f'$ having exactly one critical points, and such that the union $(W\cup_{\Sigma_1} W',f\cup f',0)$ is diffeomorphism equivalent to a cylinder.

\item[$(iv)$] \emph{Critical point switch}: Let  $s_1$ and $s_2$ be two disjoint attaching spheres in $\Sigma$, $h_1, h_2$ the corresponding handles, $W_1$ the cobordism corresponding to attaching $h_1$, $W_2$ the cobordism corresponding to attaching $h_2$ after having attached $h_1$, $W_2'$ the cobordism corresponding to attaching $h_2$, and $W_1'$ the cobordism corresponding to attaching $h_1$ after having attached $h_2$. The move consists in replacing $W_1 \odot W_2$ by $W_2' \odot W_1'$.

\item[$(v)$] \emph{homology class slide}:  If $c_i + c_{i+1} = d_i + d_{i+1}$ in $H_1(W_i \cup W_{i+1};\Z{2})$, and $W_i$ or $W_{i+1}$ is a cylinder, replacing $(W_i, c_i)\odot (W_{i+1}, c_{i+1})$ by $(W_i, d_i)\odot (W_{i+1}, d_{i+1})$ in $(\underline{W}, \underline{c})$.
\end{enumerate}
\end{defi}

\begin{prop}\label{cerfchemin}

\begin{enumerate}
\item Every cobordism with vertical boundary $(W,p,c)$ admits a  Cerf decomposition (i.e. $\Cobelem \rightarrow \Cob$ is surjective).
\item Once glued, two strings  of elementary cobordisms define  the same morphism in $\Cob$ if and only if one can pass from one to another by a finite sequence of the previous Cerf moves.
\end{enumerate}
\end{prop}

\begin{proof}

1. As in the case without vertical boundary  (\cite[Lemma 2.5]{connectedcerf}) such a decomposition is obtained from an ``excellent'' Morse function (that is, injective on the set of its critical points), without critical points of index $0$ and $3$,  and  such that if $p$ and $q$ are two critical points such that $ ind~p <ind~q $, then $f(p) < f(q)$.  Then, if $b_0 = min~f < b_1 <\cdots b_k = max~f$ is a sequence of regular values such that $[b_i, b_{i+1} ]$ contains at most one critical point, $W_i = f^{-1}([b_i, b_{i+1} ])$ is a connected cobordism  (provided there are no index $0$ and $3$ critical points) between connected surfaces. Indeed, the first disconnected surface would correspond to  a 2-handle attachment, and the next first connected one would correspond to  a 1-handle attachment, but we assumed that the  1-handles are attached before the 2-handles. 

We shall moreover assume that on the vertical boundary,  $f(p(s,t)) = K t$, for some constant $K >0$. We claim that it is possible to find such a function. Indeed, starting from a Morse function such that $f(p(s,t)) = K t$ in the neighborhood of the  vertical boundary, one can  rearrange the critical points and remove the  minimums and  maximums without modifying the function on the boundary: it suffices to take a  pseudo-gradient parallel to the vertical boundary, ensuring that the attaching spheres are confined on the interior of $W$, then the same proof as in the case without  vertical boundary applies.

Furthermore, the  class $c$ can be decomposed into classes $c_i \in H_1( W_i , \Z{2})$: to see it one can choose a 1-dimensional  representative $C\subset W$ which doesn't intersect the intermediate surfaces. A generic  representative  intersects a surface in an even number of points, that can be cancelled out in pairs.

2. Given two such Morse functions, it is possible to connect them by a path of functions having a finite number of birth-death degeneracies, or critical point switches, while keeping the same values on the vertical boundary. The end of the proof is completely analogous to the proof of \cite[Theorem 3.4]{connectedcerf}.

Notice that for the  moves $(iii), (iv)$, we assume that the homology class is zero. This can always be satisfied, up to adding  trivial cobordisms, and since one can isolate the homology classes outside a birth-death pair. 
\end{proof}

From Proposition~\ref{cerfchemin}, one obtains the following criterion, which  allows to factor a functor $\textbf{F}: \Cobelem \rightarrow  \Symp $ through the gluing functor $\Cobelem \rightarrow \Cob $.

\begin{cor}\label{factorisation} Let  $\textbf{F}: \Cobelem \rightarrow  \Symp $ be a functor satisfying:

\begin{enumerate}
\item[$(i)$] $\textbf{F}(W,c) = \textbf{F}(W',c)$ whenever $(W,c)$ and $(W',c)$ are diffeomorphism equivalent.

\item[$(ii)$] 

$\textbf{F}(W,0) = \Delta_{F(S)}$, if $W$ is a trivial cobordism. 

\item[$(iii)$] If $W\odot W'$ is a birth-death pair, the geometric composition $\textbf{F}(W,0) \circ \textbf{F}(W',0)$ satisfies the  hypotheses of Theorem~\ref{compogeom} and corresponds to the  diagonal $\Delta_{F(S)}$.

\item[$(iv)$] If $W_2' \odot W_1'$ is obtained from $W_1 \odot W_2$ by a critical point switch, then $\textbf{F}(W_1,0) \circ \textbf{F}(W_2,0) = \textbf{F}(W_2',0) \circ \textbf{F}(W_1',0)$ and these compositions satisfy the hypotheses of Theorem~\ref{compogeom}.

\item[$(v)$] If $c+c' =d+d'$, then $\textbf{F}(W,c) \circ \textbf{F}(W',c') = \textbf{F}(W,d) \circ \textbf{F}(W',d')$, and the left or right composition of $\textbf{F}(S\times I,c)$ with  any other morphism satisfies the  hypotheses of Theorem~\ref{compogeom}. 

\end{enumerate}

Then $\textbf{F}$ factors to a functor $\Cob \rightarrow \Symp $.

\end{cor}
\arnaque

\section{Construction of twisted symplectic instanton homology}\label{sec:hsi}

In order to  construct a functor $\Cob \rightarrow \Symp $, we start by building a functor $\Cobelem \rightarrow \Symp $ (first three paragraphs), and we check that it factors through the gluing functor (fourth paragraph). Finally, in the last paragraph  we will define twisted symplectic instanton  homology groups $HSI(Y,c)$ associated to a 3-manifold $Y$ together with a class $c\in H_1(Y;\Z{2})$.

\subsection{The extended moduli space}\label{sec:extmod}

All the moduli spaces appearing will always be  associated to the Lie group $SU(2)$. We will denote  $\mathfrak{su(2)}$ its Lie algebra, identified with trace-free antihermitian $2\times 2$ matrices, and endowed with the usual scalar product $\langle a,b \rangle = \mathrm{Tr} (a b^*) = -\mathrm{Tr} (ab)$. We will always identify $\mathfrak{su(2)}$ and $\mathfrak{su(2)}^*$ via this scalar product.

Let $(\Sigma,p)$ be a surface with parametrized boundary as in Definition~\ref{chemincob}, one can associate to it the extended moduli space  $ \Mg (\Sigma,  p)$ defined by Jeffrey in \cite{jeffrey} (and also, independently, by Huebschmann \cite{huebschmann}). Recall its definition:

\begin{defi}(Extended moduli space associated to a surface, \cite[Def. 2.1]{jeffrey}) 
Define the following space of flat connections:
\[ \mathscr{A}_F^\mathfrak{g}(\Sigma)  = \lbrace A \in \Omega^{1} (\Sigma )\otimes \mathfrak{su(2)}\ |\ F_A = 0,\  A_{|\nu \partial \Sigma} = \theta ds \rbrace, \]
where $\nu \partial \Sigma$ is a non-fixed tubular  neighborhood of $\partial \Sigma$, $s$  the parameter of $\rr/\zz$, and the group
 
 \[ \Gc (\Sigma ) = \left\lbrace u \colon \Sigma \rightarrow SU(2)\ |\ u_{|\nu \partial \Sigma} = I \right\rbrace \] acts by gauge transformations.
 
The extended moduli space is then defined as the quotient \[  \Mg (\Sigma, p) = \mathscr{A}_F^\mathfrak{g}(\Sigma) /\Gc (\Sigma ),\]
\end{defi}

The following proposition gives an explicit description of this space:
\begin{prop}\label{riemannhilbert} (\cite[Prop. 2.5]{jeffrey})
Let $* \in \partial \Sigma$ be a base point (we will usually take $* = p(0)$). The value $\theta ds$ of the connection  in the  neighborhood of the boundary and  the holonomy gives an identification of $\Mg (\Sigma, p)$ with  \[ \left\lbrace  (\rho, \theta)\in Hom(\pi_1(\Sigma, *), SU(2)) \times \mathfrak{su(2)}\ |\ e^\theta = \rho(p) \right\rbrace .\]
In particular, a presentation of the fundamental group
\[ \pi_1 (\Sigma, *) = \langle \alpha_1 , \beta_1 , \cdots \alpha_h , \beta_h\rangle,\]
such that  $p = \prod_{i=1}^{h}{[\alpha_i , \beta_i ]} $ induces a homeomorphism
\begin{align*}
 &\Mg (\Sigma, p) \simeq \\ & \left\lbrace   (\theta , A_1 , B_1 , \cdots, A_h , B_h ) \in \mathfrak{su(2)} \times  SU(2)^{2h} ~|~ e^{2\pi \theta}  =  \prod_{i=1}^{h}{[A_i , B_i ]} \right\rbrace . \end{align*}
The element $\theta$ is such that the connection equals $\theta ds$ in the neighborhood of the boundary, $A_i$ (resp. $B_i$)  is  the holonomy of $A$ along the curve $\alpha_i$ (resp. $\beta_i$).
\end{prop}

Define \[ \N(\Sigma,p) = \left\lbrace (\theta , A_1 , B_1 , \cdots A_h , B_h ) \in \Mg (\Sigma, p)  ~|~ |\theta | < \pi \sqrt{2} \right\rbrace.  \]

This subspace is identified with the open subset of $SU(2)^{2h}$ of elements $( A_1 , B_1 , \cdots , A_h , B_h ) $ such that $\prod_{i=1}^{h}{[A_i , B_i ]} \neq -I $, and is therefore smooth.

Recall that this space is endowed with Huebschmann-Jeffrey's 2-form, similar to the Atiyah-Bott form for a closed surface: if $A$ is a connection representing a smooth point of $\Mg (\Sigma, p)$, the tangent space  may be identified with the quotient:

\[ T_{[A]} \Mg (\Sigma, p) = \frac{ \left\lbrace  \alpha\in \Omega^1(\Sigma) \otimes \mathfrak{su(2)}~|~ \alpha_{|\nu \partial \Sigma} = \eta ds,~d_A\alpha = 0 \right\rbrace }{ \left\lbrace d_A f ~|~ f\in \Omega^0(\Sigma) \otimes \mathfrak{su(2)},~  f_{|\nu \partial \Sigma} = 0 \right\rbrace }. \]
If  $\alpha = \eta\otimes a$ and $\beta = \mu\otimes b$, with  $\eta,\mu \in \mathfrak{su(2)}$ and $a,b$ real-valued 1-forms, denote  $\langle\alpha \wedge \beta\rangle$ the real-valued 2-form defined by \[\langle\alpha \wedge \beta\rangle = \langle\eta, \mu\rangle  a\wedge b.\]
The Huebschmann-Jeffrey form $\omega$ is then defined by:
\[ \omega_{[A]}([\alpha],[\beta]) = \int_{\Sigma'} \langle \alpha\wedge\beta \rangle   .\]
This form is symplectic on the open set $\N(\Sigma,p)$, see \cite[Prop. 3.1]{jeffrey}.

\subsection{Compactification by symplectic cutting}\label{sec:cutting}
In this paragraph  we  briefly recall  how Manolescu and Woodward obtain an  object  of $\Symp$ from the space $\Mg(\Sigma,p)$: this object $\Nc(\Sigma,p)$ is a compactification of $\N(\Sigma,p)$ obtained by symplectic cutting. We refer to \cite[Parag. 4.5]{MW} for more details.

The map $[A] \mapsto \theta \in \mathfrak{su(2)}$ is the moment of an $SU(2)$-Hamil\-tonian action, so $[A] \mapsto |\theta| \in \rr$ is the moment of a circle action (in the complement of $ \lbrace \theta = 0 \rbrace $). It is then possible to consider Lerman's symplectic cutting  at some value $\lambda \in \rr$, namely the symplectic reduction  
\[ \Mg(\Sigma,p)_{\leq \lambda} = \left( \Mg(\Sigma,p) \times \cc \right)  \red U(1)  \]
of the circle action with moment $\Phi([A],z) = | \theta| +\frac{1}{2} |z|^2 -\lambda$.

\begin{remark}The circle action is not defined on $\lbrace \theta = 0 \rbrace$, Manolescu and Woodward consider in reality the reduction \[\left( ( \Mg(\Sigma,p) \setminus  \lbrace \theta = 0 \rbrace) \times \cc \right)  \red U(1),\] which contains $\N(\Sigma,p) \setminus  \lbrace \theta = 0 \rbrace$, and then glue back $\lbrace \theta = 0 \rbrace$ to it.
\end{remark}

For $\lambda = \pi \sqrt{2}$, 

\[ \Mg(\Sigma,p)_{\leq \pi \sqrt{2}} = \N(\Sigma,p) \cup R, \]

with $R =  \lbrace |\theta| = \pi \sqrt{2} \rbrace / U(1) $.

This space will be denoted $\Nc(\Sigma,p)$, and $\tilde{\omega}$ is the induced 2-form from the reduction. This form is monotone (\cite[Proposition 4.10]{MW}), but degenerated on $R$ (\cite[Lemma 4.11]{MW}).

Besides, if one cuts at $ \lambda = \pi \sqrt{2} - \epsilon$ for $\epsilon$ small,  $\Mg(\Sigma,p)_{\leq \lambda }$ is still diffeomorphic to $\Nc(\Sigma,p)$. Let  $\varphi_\epsilon\colon \Nc(\Sigma,p) \rightarrow \Mg(\Sigma,p)_{\leq \pi \sqrt{2} - \epsilon }$ be a diffeomorphism with support contained in a neighborhood of $R$, and $\omega_\epsilon$  the  symplectic form of $\Mg(\Sigma,p)_{\leq \pi \sqrt{2} - \epsilon }$, then $\omega = \varphi_\epsilon ^* \omega_\epsilon$ is a non-monotone symplectic form on $\Nc(\Sigma,p)$.

Finally, $\tilde{J}$ is a ``reference'' almost complex structure on $\Nc(\Sigma,p)$, compatible with  $\omega$ and such that $R$ is a complex  hypersurface.

Recall  also the following result concerning the structure of the degeneracy locus $R$ of $\tilde{\omega}$, which will be useful in order to control bubbling phenomenas: 

\begin{prop} (\cite[Prop. 3.7]{MW})\label{degener} The hypersurface $R$ admits a 2-sphere fibration such that the kernel of $\tilde{\omega}$ corresponds to the tangent space of the fibers. Furthermore, the intersection number of a fiber with  $R$ in $\Nc(\Sigma,p)$ is -2.

\end{prop} 

It follows in particular that the pseudo-holomorphic  curves of zero area for $\tilde{\omega}$ are branched covers of fibers of this fibration.

To sum up:

\begin{prop}[\cite{MW}] The 5-tuple ($\Nc(\Sigma,p)$, $\omega$, $\tilde{\omega}$,$R$,$\tilde{J}$) satisfies the hypotheses of  Definition~\ref{defcat}: it is an object  of $\Symp$.
\end{prop}

\subsection{Lagrangian correspondences from three-cobordisms}\label{sec:lagcob}

 As in Definition~\ref{chemincob}, Let $(W,\pi_{\Sigma_0},\pi_{\Sigma_1},p,c)$ be a cobordism with vertical boundary, and $C\subset int(W)$ a closed 1-dimensional submanifold, whose homology class in $H_1(W,\Z{2})$ is $c$. It will follow from property~\ref{reformul} that the next construction will only depend on $c$, rather than $C$.

We start by defining a  correspondence $L(W,\pi_{\Sigma_0},\pi_{\Sigma_1},p, C)$ between the sets $\Mg (\Sigma_0, p_0)$  and $\Mg (\Sigma_1, p_1)$, which will give rise,  when the cobordism is elementary, to a Lagrangian correspondence between the symplectic cuttings. We will denote it $L^c(\pi_{\Sigma_0},\pi_{\Sigma_1},p, C)$, and it will satisfy the hypotheses of the category  $\Symp$.

\begin{defi}
\begin{itemize}

\item[$(i)$] (Moduli space associated to a cobordism with vertical boundary $(W, \pi_{\Sigma_0},\pi_{\Sigma_1},p, C)$). Define the following space of connections, where $s$ stands for the $\rr/\zz$-coordinate of  $p(\rr/\zz \times [0,1])$, and  $\mu$ is an arbitrary meridian of $C$:

\begin{align*}&\mathscr{A}_F^\mathfrak{g}(W,C) =  \\ 
& \left\lbrace A \in \Omega^{1} (W \setminus C) \otimes \mathfrak{su(2)}\ |\ F_A = 0, \mathrm{Hol}_{\mu} A = -I, A_{|\nu p(\rr/\zz \times [0,1])} = \theta ds\right\rbrace.
\end{align*}

This space is acted on by the following gauge group:

\[\Gc (W\setminus C)  = \left\lbrace u \colon W \setminus C \rightarrow SU(2)\ |\ u_{|\nu p(\rr/\zz \times [0,1])} = I\right\rbrace,\]

We then define the quotient \[ \Mg (W, \pi_{\Sigma_0},\pi_{\Sigma_1},p, C) =  \mathscr{A}_F^\mathfrak{g}(W,C) /\Gc (W\setminus C ).\]

\item[$(ii)$] (Correspondence associated to a cobordism with vertical boundary.) Let  

\[ L(W,\pi_{\Sigma_0},\pi_{\Sigma_1},p, C) \subset \Mg (\Sigma_0, p_0)^- \times \Mg (\Sigma_1, p_1) \]
be the correspondence consisting of the pairs of connections that extend flatly  to $W\setminus C$, with  holonomy $-I$ around $C$:

\[L(W,\pi_{\Sigma_0},\pi_{\Sigma_1},p, C) = \{([A_{|\Sigma_0} ], [A_{|\Sigma_1} ])\ |\ A \in \Mg(W, C, *)\}\]

\end{itemize}
\end{defi}

\begin{remark}This definition depends on the parametrization $p$ of the vertical part of $\partial W$, this dependence will be described in example~\ref{exemreparam}. 

\end{remark}

\subsubsection{Homology classes as twisted SO(3)-bundles}\label{sec:so3bundles}

We now explain how one can alternatively define the Lagrangian correspondences by using a twisted $SO(3)$-bundle rather than removing the submanifold $C$. It will follow  that these depend on $C$ only through its class $c$ in $H_1(W; \Z{2})$, which is dual to the second Stiefel-Whitney class of the bundle. This fact could also have been proven directly by observing that the moduli space doesn't change when one removes a crossing of $C$ and replace it by two parallel arcs, nevertheless this point of view highlights the analogy of our exact sequence with Floer's initial exact triangle for instanton homology.

Let $P$ be an $SO(3)$-principal bundle over $W$ defined by gluing the trivial bundles over $ W\setminus \nu C$ and $ \nu C$ along the boundary $\partial \nu C$ by a  transition function  $f\colon \partial \nu C \rightarrow SO(3)$ such that the image of every meridian of $C$ is non-trivial in $\pi_1(SO(3))$: 

\[ P =  SO(3) \times (W\setminus \nu C) \cup_f SO(3) \times  \nu C .\]

We will denote $\tau\colon P_{| W\setminus \nu C}\rightarrow ( W\setminus \nu C ) \times SO(3)$ the trivialization of $P$ on $W\setminus \nu C$.

\begin{lemma} The second Stiefel-Whitney  class of $w_2(P)\in H^2(W; \Z{2})$ is Poincare-dual  to the image of $c$ in $H_1(W, \partial W; \Z{2})$.
\end{lemma}

\begin{proof}

We start by recalling the construction of $w_2(P)$ in \v{C}ech homology: fix $\lbrace U_i\rbrace_i$ an acyclic cover of $W$ compatible with  $C$ in the  following sense: in the neighborhood of $C$, the cover is modeled on 4 open subsets $V_0, V_1, V_2, V_3$ as in figure~\ref{recouvrement}. Each connected component  of $\nu C$ is covered by 3 (or more) open subsets of type $V_0$, and $\partial \nu C$ is covered by $V_1, V_2, V_3$.

\begin{figure}[!h]
    \centering
    \def\svgwidth{.65\textwidth}
    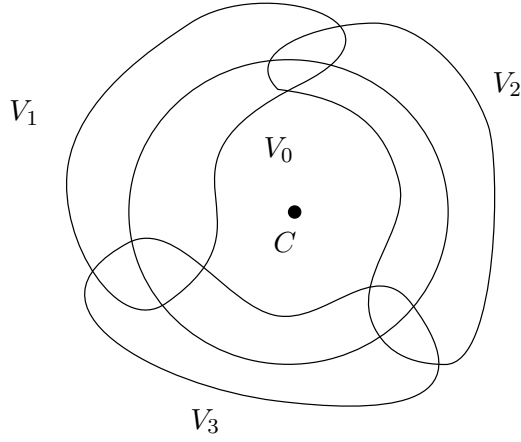
      \caption{The covering in the neighborhood of $C$.}
      \label{recouvrement}
\end{figure}

The bundle $P$  is given by transition  functions of  $\alpha_{ij}\colon U_i \cap U_j \rightarrow SO(3)$. These functions satisfies $\alpha_{ij} \alpha_{ji} = I$ and $\alpha_{ij} \alpha_{jk} \alpha_{ki} = I$. Let  \[\widetilde{\alpha}_{ij}\colon U_i \cap U_j \rightarrow SU(2)\] be lifts of $\alpha_{ij}$ to $SU(2)$, the second former relation becomes \[\widetilde{\alpha}_{ij} \widetilde{\alpha}_{jk} \widetilde{\alpha}_{ki} \in \Z{2}\] and allows one to define a  2-cocycle in \v{C}ech homology:
\[ (c_{ijk} \colon U_i \cap U_j \cap U_k \rightarrow \Z{2})_{ijk} \in \check{\mathrm{C}}^2(W,\Z{2}), \]
and its class in $\check{\mathrm{H}}^2(W,\Z{2})$ is then $w_2(P)$, by definition.

By  construction of the bundle $P$, the transition functions can be chosen, if one still denotes by $f$ an extension of $f$ to a neighborhood of $\partial \nu C$, as $\alpha_{ij}(x) = f(x) $ if $U_i = V_0$ and $U_j \in \lbrace V_1, V_2, V_3 \rbrace$, and $\alpha_{ij} = I$ otherwise. 

By hypothesis, the transition function  $f$ doesn't lift to a function from $\partial \nu  C$ to $SU(2)$. However, it is possible to chose lifts $\widetilde{f}_j \colon V_j\rightarrow SU(2)$ for each $j = 1,2, 3$. One can assume that 
$\widetilde{f}_1 = \widetilde{f}_2 $ on $V_1 \cap V_2 $, $\widetilde{f}_2 = \widetilde{f}_3 $ on $V_2 \cap V_3 $, and $\widetilde{f}_3 = - \widetilde{f}_1 $ on $V_3 \cap V_1 $. Take  then  $\widetilde{\alpha}_{ij}(x) = \widetilde{f}_j(x) $ if $U_i = V_0$ and $U_j \in \lbrace V_1, V_2, V_3 \rbrace$, and $\alpha_{ij} = I$ otherwise. The cocycle $(c_{ijk}) _{ijk}$ takes  then the following values:

\[ c_{ijk} = \begin{cases} -I\text{ if } \lbrace U_i, U_j, U_k\rbrace = \lbrace V_0, V_1, V_3 \rbrace\\ I\text{ otherwise,} \end{cases}  \]
where $V_0$, $V_1$ and $V_3$ are  open subsets of the former  type.

Take now a cycle \[F = \sum_{ \lbrace i,j,k\rbrace\in I_F}{U_{i}\cap U_j \cap U_k } \in \check{\mathrm{C}}_2(W,\Z{2}), \] then \[\langle w_2(P), [F]\rangle = \sum_{\lbrace i,j,k\rbrace\in I_F}{c_{ijk}} = [C] . [F].\]

\end{proof}

Let $\A(W,P)$ be the space of flat connections on  $P$ of the form $\theta ds$ in the neighborhood of $\partial ^{vert} W$, where one has  identified connections with $\mathfrak{su(2)}$-valued 1-forms via the trivialization $\tau$, and $s$ is the circular parameter  of $\partial ^{vert} W$. This space is acted on by the group $\G^0(W,P)$ of gauge  transformations  which are trivial in the neighborhood of $\partial ^{vert} W$ and homotopic to the identity (i.e. the   connected component of the identity of the group of gauge transformations  trivial in the neighborhood of $\partial ^{vert} W$). Let \[\Mg(W,P) =  \A(W,P)/\G^0(W,P)\] be the corresponding orbit space. The trivialization $\tau$ allows one to define a map \[  \Mg(W,P) \rightarrow  \N(\Sigma_0) \times \N(\Sigma_1)\] by restriction to the boundaries and pullback to $SU(2)\times ( \Sigma_0 \sqcup \Sigma_1 )$. We denote $L(W,P) \subset \N(\Sigma_0) \times \N(\Sigma_1)$ its image.

\begin{remark}\label{dependiso} The moduli space $\Mg(W,P)$ only depends on the isomorphism type of $P$, i.e. the class $c$, and the correspondence $L(W,P)$ only depends on the restriction of $\tau$ to $\partial W$. It follows that $L(W,P)$ only depends on  $C$  via $c$.
\end{remark}

\begin{prop} \label{reformul}

The moduli space $\Mg(W,C)$ is canonically identified with $\Mg(W,P)$. It follows that $L(W,P) = L(W,C)$, and from remark~\ref{dependiso} $L(W,C)$ only depends on the class $c$.

\end{prop}

\begin{proof}

In order to prove this, we will construct two maps which will be inverses from one another:
\[ \begin{cases}
\Phi_1\colon \Mg(W,C) \rightarrow \Mg(W,P) \\ 
\Phi_2\colon \Mg(W,P) \rightarrow \Mg(W,C). 
\end{cases} \]

\textbf{1. The map $\Phi_1$.} Let $[A] \in \Mg(W,C)$ and take $A \in [A]$  a representative. The connection $A$ descends to a flat connection $\widehat{A}$ on $SO(3) \times (W\setminus\nu C)$, and if $\mu$ is a meridian of $C$, $\mathrm{Hol}_\mu \widehat{A} = I$. This fact, together with the choice of a parametrization $p\colon C\times \rr/\zz \rightarrow \partial \nu C$ allows one to define a transition  function   $f\colon \partial \nu C\rightarrow SO(3)$ by \[f(c,s) = \mathrm{Hol}_{\lbrace c\rbrace\times [0,s]}\widehat{A},\] with  $[0,s]\subset \rr/\zz$ an oriented arc from 0 to $s$.

This transition function allows one to glue the flat bundle  \[(SO(3) \times (W\setminus\nu C), \widehat{A})\] with  the horizontal bundle  $(SO(3) \times \nu C, A_{horiz})$. Let $(Q, A_Q)$ be the resulting flat  bundle. The  function $f$ satisfies the same  hypothesis as the one used to define  the bundle $P$, thus  the bundles $Q$ and $P$ are isomorphic. Let $\varphi\colon Q\rightarrow P$ be an isomorphism, such that $\tau \circ\varphi$ is the identity on $SO(3) \times (W\setminus\nu C)$.

Define finally $\Phi_1([A]) = [\varphi_* A_Q]\in \Mg(W,P)$. This class is independent on the choices made, modulo an element of $\G^0(W,P)$.

\textbf{2. The map $\Phi_2$.} Let  $[A] \in \Mg(W,P)$, and $A\in [A]$ a representative. The push-forward $\tau_*  A_{|W\setminus \nu C}$ defines a connection on $SO(3)\times W\setminus \nu C$, call $\widetilde{A}$ the connection on $SU(2)\times ( W\setminus \nu C)$ pulled-back by the quotient map. This connection satisfies $\mathrm{Hol}_\mu \widetilde{A} = -I$ for every meridian $\mu$ of $C$, indeed in the trivialization over $\nu C$, the loop $\gamma\colon s\mapsto \mathrm{Hol}_{[0,s]} A$ is nullhomotopic in $SO(3)$, as $\mu$ bounds a disc. It follows that the loop $\widetilde{\gamma} \colon s\mapsto \mathrm{Hol}_{[0,s]} A$ defined in the  trivialization over $W\setminus \nu C$ is not nullhomotopic, as $\gamma$ and $\widetilde{\gamma}$ differs by the transition  function $f$. The connection $\widetilde{A}$ defines consequently an element $\Phi_2([A])$ of  $ \Mg(W,C)$, independent on the choices modulo the action of $\Gc(W,C)$.

These two maps are inverses from one another by construction, and thus identify $\Mg(W,C)$ with $\Mg(W,P)$.
\end{proof}

\subsubsection{Correspondences associated to an elementary cobordism}
The moduli space $\Mg (W, \pi_{\Sigma_0},\pi_{\Sigma_1},p, c)$ associated to an arbitrary vertical cobordism might not be smooth, and the map induced by the inclusion \[\Mg (W, \pi_{\Sigma_0},\pi_{\Sigma_1},p, c) \rightarrow \Mg (\Sigma_0)  \times \Mg (\Sigma_1)\] might not be an embedding, therefore the correspondence $L(W, \pi_{\Sigma_0} , \pi_{\Sigma_1} ,p ,  c)$ might not be a  Lagrangian submanifold. We will see that these problems don't appear for elementary  cobordisms. We now describe the correspondences associated to such cobordisms, and then prove that they actually are  Lagrangian submanifolds in Proposition~\ref{correspcobelem}.

\begin{exam}[Trivial cobordism]\label{exemcobtriv}
Let $(\Sigma,p)$ be a surface with para\-metrized boundary, and $W$ the cobordism with vertical boundary $\Sigma \times [0,1]$, endowed with the embeddings $\pi_i(x) = (x,i)$ and $p(s,t) = (p(s) , t)$.

\begin{itemize}
\item If $C = \emptyset$, $L(W,C)$ is the diagonal $\Delta_{\Mg(\Sigma,x,*)}$.
\item If $C \neq \emptyset$, $L(\Sigma \times [0,1] , C)$ is the graph of the diffeomorphism whose representation-theoretic expression is given by: 
\[\begin{cases}
A_i \mapsto (-1)^{\alpha_i . C'} A_i\\
B_i \mapsto (-1)^{\beta_i . C'} B_i,
\end{cases}\]
 where $\pi_1 (\Sigma , p(0)) = \langle\alpha_1 , \cdots , \beta_h\rangle$, $C'$ is the projection of $C$ on $\Sigma$ and $\alpha_i . C'$, $\beta_i . C'$ denotes the intersection numbers in $\Sigma$ modulo 2.
\end{itemize}

In particular, if $a_i = [\alpha_i]\in H_1(\Sigma \times [0,1], \Z{2})$ and $b_i = [\beta_i]$, then $L(\Sigma \times [0,1] , a_i, *)$ corresponds to the diffeomorphism sending $B_i$ to $-B_i$ and preserving the other holonomies, and $L(\Sigma \times [0,1] , b_i, *)$ corresponds to the diffeomorphism sending $A_i$ to $-A_i$ and preserving the other holonomies.

\end{exam}

\begin{proof} Take for the link $C$  a simple curve contained in the surface $\Sigma\times \lbrace \frac{1}{2}\rbrace$, so that the complement $W\setminus C$  retracts to the union of $\Sigma\times \lbrace 0\rbrace$ and a torus envelopping $C$ and touching $\Sigma\times \lbrace 0\rbrace$ in $C\times \lbrace 0\rbrace$. From the Seifert-Van Kampen Theorem, $\pi_1(W\setminus C, *) \simeq (\zz \lambda \oplus \zz \mu ) * F_{2g-1}$, where $\lambda$  and $\mu$ denotes a longitude and a meridian of $C$. Therefore, the representations of $\pi_1(W\setminus C, *) $ sending $\mu$ to $-I$ are in one-to-one correspondence with representations of $(\zz \lambda  ) * F_{2g-1} \simeq \pi_1(\Sigma, *)$, as $-I$ is in the center of $SU(2)$. It follows that $\Mg(W,C) \simeq \Mg(\Sigma \times \lbrace 0\rbrace,p)$.

Let us now look at the map $\Mg(W,C) \rightarrow \Mg(\Sigma \times \lbrace 1\rbrace,p)$ induced by the inclusion. If $\gamma$ is a based loop in $\Sigma$, the square $\gamma\times [0,1]$ meets $C$ $\gamma \cdot C'$ times,  therefore the holonomy of a connection $A$ around its boundary is $(-1)^{\gamma \cdot C'}$. Besides, this holonomy is also equal to $\mathrm{Hol}_{\gamma \times \lbrace 1\rbrace} A (\mathrm{Hol}_{\gamma \times \lbrace 0\rbrace}A)^{-1}$.

\end{proof}

\begin{exam}[Reparametrization of the vertical cylinder]\label{exemreparam}
Suppose that $W$ and the embeddings $\pi_0$, $\pi_1$ are as in the previous example, but $p(s,t) = (p(s) +\psi (t) , t)$, for a function $\psi\colon [0,1] \rightarrow \rr$. Then the correspondence $L(W, \pi_{\Sigma_0} , \pi_{\Sigma_1} , C, p)$ is the graph of the following diffeomorphism:
\[ (\theta, A_1, B_1, \cdots) \mapsto (\theta, Ad_{e^{\alpha \theta}}A_1, Ad_{e^{\alpha \theta}}B_1, \cdots),\]
with $\alpha = \psi(1) - \psi(0)$. (This corresponds to rotating the boundary of $\Sigma$ with an angle $\alpha$.) 

\end{exam}

\begin{proof} When reproducing the same reasoning as in the previous example, the holonomies along the two remaining sides of the square $\gamma\times [0,1]$ are $e^{\alpha \theta}$ and $e^{-\alpha \theta}$ respectively.
\end{proof}

\begin{exam}[Diffeomorphism of a surface]\label{exemdiffeo}

Let $\varphi$ be a diffeomorphism of $(\Sigma,p)$ equal to the identity near the boundary, $W = \Sigma \times  [0, 1]$, $\pi_0 =  id_\Sigma \times  \lbrace 0\rbrace$, $\pi_1 = \varphi  \times  \lbrace 1\rbrace$, and $p'(s,t) = (p(s),t)$. If $\pi_1 (\Sigma' , *) = \langle\alpha_1 , \cdots , \beta_h\rangle$ is the free group with $2h$ generators, let  $ u_i (\alpha_1 , \cdots , \beta_h )$ be the word in $\alpha_1 , \cdots , \beta_h$ corresponding to $\varphi_* \alpha_i$, and $v_i (\alpha_1 , \cdots , \beta_h )$ the word corresponding to $\varphi_* \beta_i$. Then  $L(W, \pi_{\Sigma_0} , \pi_{\Sigma_1} ,  p', 0)$ is the graph of the diffeomorphism:

\[ (\theta, A_1, B_1, \cdots) \mapsto (\theta, u_1(A_1, B_1, \cdots), v_1(A_1, B_1, \cdots), \cdots).\]

In particular, the Dehn twist along a curve freely  homotopic to $\beta_1$ is the graph of the diffeomorphism:

\[ (\theta, A_1, B_1, \cdots) \mapsto (\theta, A_1 B_1,  B_1,  \cdots).\]

\end{exam}

\begin{proof}This follows from example~\ref{exemcobtriv} and the formula giving  the holonomy along a product of loops.

\end{proof}

The  following example illustrates the necessity of considering the category  $\Cob$ rather than a category  of cobordisms without vertical boundaries. Indeed, if one closes the following cobordism by gluing a tube along the vertical boundary, one obtains a trivial cobordism  from the torus to itself, identical with the one one would have obtained from the cobordism of example~\ref{exemcobtriv}, but their associated Lagrangian correspondences aren't the same.

\begin{exam}[A change of ``base path'']\label{exemchgtchemin}
Let  $\Sigma$ be the 2-torus with a small disc $D$ removed, $*$ a base point on its  boundary,  $\alpha_1$ and $\beta_1$ two simple curves  forming a  basis of its fundamental group, and $W = (T^2\times[0,1]) \setminus S$, where $S$ is a tubular neighborhood  of the path $(\alpha_1(t),t)$, (with its vertical boundary  parametrized without spinning) see figure~\ref{exemchgtcheminfig}. $L(W,p, 0)$ is the graph of:

\[ (\theta, A_1, B_1) \mapsto (\theta, A_1 , A_1^{-1} B_1 A_1).\]

\end{exam}

\begin{figure}[!h]
    \centering
    \def\svgwidth{.65\textwidth}
    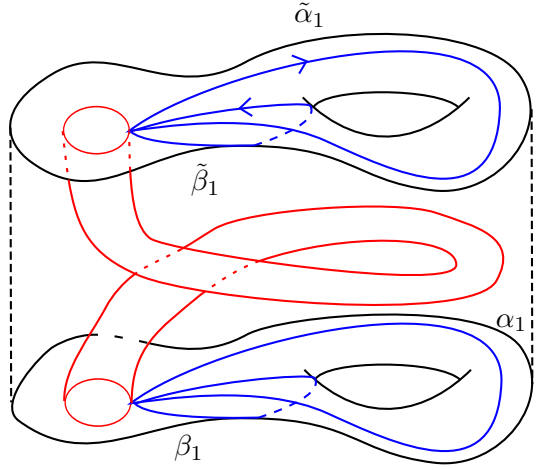
      \caption{A change of ``base path''.}
      \label{exemchgtcheminfig}
\end{figure}

\begin{proof} Identify $\alpha_1$ and $\beta_1$  with the corresponding curves in $\Sigma \times\lbrace 0\rbrace$, and denote $\tilde{\alpha}_1$ and $\tilde{\beta}_1$ the corresponding curves in $\Sigma \times\lbrace 1\rbrace$. Denote $*$ and $\tilde{*}$ the corresponding base points  and $\gamma$ the vertical arc  going from $*$ to $\tilde{*}$. The claim follows from the fact that  $\tilde{\alpha}_1$ (resp.$\tilde{\beta}_1$ ) is homotopic to $\gamma \alpha_1 \gamma^{-1}$ (resp. $\gamma \alpha_1^{-1} \beta_1 \alpha_1 \gamma^{-1}$).

\end{proof}

\begin{exam}[2-handle attachment]\label{exemhandle}

Let $s\subset int(\Sigma)$ be a simple curve freely homotopic to $\beta_{1}$, and $W\colon \Sigma \rightarrow S$ the cobordism corresponding to the attachment of a 2-handle along $s$,  then $\pi_1(S) = \langle \alpha_2, \beta_2, \cdots \rangle$, and

\[ L(W, p, 0) = \lbrace (\theta , A_1 , I , A_2, B_2,  \cdots A_h , B_h ), (\theta , A_2 , B_2 , \cdots A_h , B_h )  \rbrace, \]
where $A_1 \in SU(2)$ and $(\theta , A_2 , B_2 , \cdots A_h , B_h ) \in \Mg(S)$.

\end{exam}

\begin{proof} The cobordism $W$ retracts to the wedge of $S$ and the circle $\alpha_1$ corresponding to the co-sphere of the handle. It follows that $\Mg(W) \simeq \Mg(S) \times SU(2)$.

Furthermore, under this identification, the map $\Mg(W) \rightarrow \Mg(S)$ is the projection on the first factor, and $\Mg(W) \rightarrow \Mg(\Sigma)$ sends the pair $(A_1, [A] )$ to the connection such that $\mathrm{Hol}_{\beta_{1}} = I$, $\mathrm{Hol}_{\alpha_{1}} = A_1$, and  whose other  holonomies are the same as those of $A$.

\end{proof}

\begin{prop}\label{correspcobelem}If there exists, as in Definition~\ref{chemincob}, a  Morse function $f$ on $W$, constant on the  boundaries $\Sigma_0$ and $\Sigma_1$, and with  at most one critical point, then in restriction to the symplectic part  of the moduli spaces, $L(W, \pi_{\Sigma_0} , \pi_{\Sigma_1} , p, c)$ is a   Lagrangian correspondence.
\end{prop}

\begin{proof}

An elementary cobordism  can be either a trivial cobordism, a 1-handle or a 2-handle attachment. According to examples~\ref{exemcobtriv} and~\ref{exemhandle}, in each case, $L(W, \pi_\Sigma , \pi_S , C, p)$ is smooth and of maximal dimension (in fact, 1-handle and 2-handle attachments are symmetric).

We prove that these correspondences are isotropic for the symplectic form. Let $[A]\in \Mg (W, C, p)$, take a representative $A\in [A]$ of the form $\eta_0 ds$ in the neighborhood of $C$, where $s\in \rr/\zz$ is the parameter of a meridian, and $\eta_0\in \mathfrak{su(2)}$ a fixed element  such that $\exp(\eta_0) = -I$.

 Let  $\alpha$, $\beta$ be representatives of two tangent vectors in $T_{[A]}\Mg (W, C, p)$, namely $\mathfrak{su(2)}$-valued 1-forms   satisfying $d_A\alpha = d_A\beta = 0$, and of the form $\theta ds$ in the neighborhood of $\partial^{vert} W$. Since every flat connection near $A$ can be written, up to a gauge transform, in the form $\eta_0 ds$ in the neighborhood of $C$, one can assume furthermore that $\alpha$ and $\beta$ vanish in the neighborhood $C$. In particular  $\alpha$ and $\beta$ can be extended flatly to $W$.

If we denote $\tilde{A} \in L(W, \pi_\Sigma , \pi_S , C, p)$ a representative of the image of $[A]$ by the embedding

\[ \Mg (W, C, p) \rightarrow \Mg (\Sigma_0, p_0) \times \Mg (\Sigma_1, p_1), \]
and $\tilde{\alpha},\tilde{\beta} \in T_{\tilde{A}}L(W, \pi_\Sigma , \pi_S , C, p)$ the corresponding tangent vectors,

\[ \omega_{\tilde{A}}(\tilde{\alpha},\tilde{\beta}  ) = \int_{\Sigma_1}{\langle \alpha \wedge \beta \rangle } -  \int_{\Sigma_0}{\langle \alpha \wedge \beta \rangle }. \]

According to Stokes formula: 

\begin{align*}
 0 &= \int_{W}{d \langle \alpha \wedge \beta \rangle }  \\ 
 &= \int_{\Sigma_1}{\langle \alpha \wedge \beta \rangle } -  \int_{\Sigma_0}{\langle \alpha \wedge \beta \rangle } +  \int_{\partial^{vert} W}{\langle \alpha \wedge \beta \rangle } \\
\end{align*}
And the last term vanishes since $\alpha$ and $\beta$ are proportional to $ds$ on the vertical part.
\end{proof}

It follows from Proposition~\ref{correspcobelem} that the diffeomorphisms appearing in the previous examples are symplectomorphisms. Only the last kind of  correspondences (example~\ref{exemhandle}) doesn't come from a symplectomorphism, but from a fibered coisotropic submanifold. The following statement, which can be found in \cite[Example 6.3]{MW} for a spherically fibered coisotropic submanifold, gives an useful criterion for cutting Lagrangian correspondences, and applies to all the previous examples (indeed, for a symplectomorphism it suffices to consider $C = M_0$ and  $\varphi = \pi$).

\begin{remark}To be completely rigorous, the 2-form can degenerate, however ``Hamiltonian action'' continues to make sense as long as  the equation ``$\iota _{X_\xi} \omega =  d\langle H,\xi\rangle$'' holds. The following statement still holds in this case.
\end{remark}

\begin{prop}
\label{decoupcoiso}Let $M_0$ be a symplectic manifold endowed with a $U(1)$-Hamiltonian action  with moment $\varphi_0\colon M_0\rightarrow \rr$, together with a coisotropic  submanifold  $C\subset M_0$ admitting a fibration $\pi\colon C\rightarrow M_1$ over a symplectic manifold $M_1$ such that the image of $C$,  $L =(\iota \times \pi) (C) \subset M_0^- \times M_1$ is a Lagrangian correspondence.

Let $\lambda\in \rr$ be a regular value of $\varphi_0$  such that the action of $U(1)$ on $\varphi_0^{-1}(\lambda)$ is free. One can then take Lerman's symplectic cutting $M_{0,\leq \lambda} = M_{0,< \lambda}\cup R_0$.

Assume furthermore that $C$ is $U(1)$-equivariant, and intersects $\varphi_0^{-1}(\lambda)$ transversely. The  $U(1)$-action descends to a Hamiltonian action  with moment $\varphi_1 \colon M_1 \rightarrow \rr$,  for which $\lambda$ is a regular  value. We denote $M_{1,\leq \lambda} = M_{1,< \lambda}\cup R_1$ Lerman's symplectic cutting.

Then, the closure $L^c$ of $L\cap  (M_{0,< \lambda}^- \times M_{1,< \lambda} )$ in $M_{0,\leq \lambda}^- \times M_{1,\leq \lambda}$ defines  a  $(R_0, R_1)$-compatible Lagrangian correspondence.

If furthermore $M_{0,\leq \lambda}$ and $M_{1,\leq \lambda}$ are objects of $\Symp$, and if \[L\cap  (M_{0,< \lambda}^- \times M_{1,< \lambda} )\] is simply connected and spin, then $L^c$ is a morphism of  $\Symp$: every nonconstant pseudo-holomorphic disc $(u_0,u_1)\colon (D^2,\partial D^2) \rightarrow (M_{0,\leq \lambda}^- \times M_{1,\leq \lambda}, L^c)$ of zero area has an intersection number with  $(R_0, R_1)$ strictly smaller than -2.

\end{prop}

\begin{proof}

Denote $\Phi_i \colon M_i \times \cc \rightarrow \rr$ the moments of the $U(1)$-action, defined by \[\Phi_i(m,z) = \varphi_i (m) + \frac{1}{2}|z|^2 - \lambda,\] which will give rise to the cuttings $M_{i,\leq \lambda} = \Phi_i^{-1}(0) / U(1)$. Denote also 
$Q_i = \varphi_i^{-1}(\lambda) $, so that $R_i = Q_i /U(1)$. Finally, denote \[\widetilde{L} = (L\times \cc^2 ) \cap ( \Phi_0^{-1}(0) \times \Phi_1^{-1}(0)) \subset M_0 \times \cc \times M_1 \times \cc .\]

We will show that $ \widetilde{L} / U(1)^2 \subset M_{0, \leq \lambda} \times M_{1, \leq \lambda}$ is a smooth Lagrangian correspondence, and is compatible with  the hypersurfaces. This correspondence contains $L\cap  (M_{0,< \lambda}^- \times M_{1,< \lambda} )$ as an   open  dense subset, it will follow that $L^c = \widetilde{L} / U(1)^2$.

First, $\Phi_0^{-1}(0)$ and $ \Phi_1^{-1}(0)$ are smooth, provided $\lambda$ is a regular value of $\varphi_0$ and $\varphi_1$. The intersection $(L\times \cc^2 ) \cap ( \Phi_0^{-1}(0) \times \Phi_1^{-1}(0))$ is transverse in $M_0 \times \cc \times M_1 \times \cc$, indeed $(\lbrace 0\rbrace \times \cc^2 ) \cap ( \Phi_0^{-1}(0) \times \Phi_1^{-1}(0)) = \lbrace 0\rbrace$.  Finally, the action of $U(1)^2$ on $\tilde{L}$ is free, provided the action of $U(1)$ is free on $\cc\setminus 0$ and on $\varphi_0^{-1}(\lambda)$, by assumption. It follows that $\tilde{L}/U(1)^2$ is smooth.

Let us prove now the compatibility with the hypersurfaces. 
First, let $(m_0,m_1) \in L$. Since $\varphi_0 (m_0) = \varphi_1 (m_1)$, it follows that
\[L\cap (M_0 \times Q_1 ) = L\cap (Q_0 \times M_1) = L\cap (Q_0 \times Q_1),\] 
and 
\[\widetilde{L}\cap (M_0 \times Q_1 \times \cc^2) = L\cap (Q_0 \times M_1\times \cc^2) = L\cap (Q_0 \times Q_1\times \cc^2),\] 
and then: 
\[L^c\cap (M_{0, \leq \lambda} \times R_1  ) = L\cap (R_0 \times M_{1, \leq \lambda}) = L\cap (R_0 \times R_1).\]

Let us check now that these intersections are transverse. One should prove that
 $\forall x \in L^c\cap (R_0 \times R_1 ),$ \[T_x L^c\cap T_x (M_0 \times R_1 ) =T_x  L^c\cap T_x (R_0 \times M_1) = T_x ( L\cap (R_0 \times R_1) ).\]
Let $x \in L\cap (Q_0 \times Q_1 )$ and $(v_0,v_1) \in T_x  L$. It follows from $v_1 = d\pi_{x_0}.v_0$ that:
\[ T_x L\cap T_x (M_0 \times Q_1 ) =T_x  L\cap T_x (Q_0 \times M_1) = T_x ( L\cap (Q_0 \times Q_1) ),\]
and the claim follows.

Finally, let us prove the property about the zero-area pseudo-holomorphic discs. Let $(u_0,u_1)\colon D^2\to \times M_{0, \leq \lambda}^- \times M_{1, \leq \lambda}$ be such a disc, with boundary in $L^c$. Then $\pi\circ u_0$ and $u_1$ coincide on the boundary of $D^2$, and  $\pi\circ u_0 \cup u_1\colon D^2\cup_{\partial D^2} D^2 \rightarrow M_{1, \leq \lambda}$ defines a zero-area nonconstant pseudo-holomor\-phic sphere of $M_{1, \leq \lambda}$, which intersects $R_0$ in a  positive  multiple of -2, since by assumption $M_{1, \leq \lambda}$ is an object of $\Symp$.

\end{proof}
Notice that if $(W,p,c)$ is an elementary cobordism,  \[ L(W,p,c) \cap \left( \N (\Sigma_0,p_0)^- \times \N (\Sigma_1,p_1) \right)\] may be identified either with  $\N (\Sigma_i,p_i) \times SU(2)$, with  $i=0$ or 1, or with $\N (\Sigma_0,p_0)$. In both cases, this is an open subset of a product of copies of $SU(2)$, therefore its  second Stiefel-Whitney  class vanishes.

Define now the Lagrangian correspondences between the cut spaces:

\begin{defi}
If $(W,p,c)$ is an elementary cobordism with vertical boun\-dary  from $(\Sigma_0,p_0)$ to $(\Sigma_1,p_1)$, then the correspondence $L(W,p,c)$ satisfies the hypotheses of Proposition~\ref{decoupcoiso}. One can then define \[L^c(W,p,c)\subset \Nc(\Sigma_0,p_0)^- \times \Nc(\Sigma_1,p_1)\] as the closure  of 
$ L(W,p,c) \cap \left( \N (\Sigma_0,p_0)^- \times \N (\Sigma_1,p_1) \right)$,
which is a morphism of $\Symp$ according to Proposition~\ref{decoupcoiso}.
\end{defi}

\subsection{Cerf moves invariance}\label{sec:cerfinvar}
The following fact holds true for every cobordism with vertical boundary, elementary or not:

\begin{prop}[Composition formula]\label{compocob} Let  $\Sigma,S,T$ be three surfaces with parametrized boundaries, and  $(W_1, c_1)$, $(W_2, c_2 )$ two cobordisms with vertical boundary, going respectively from $\Sigma$ to $S$, and from $S$ to $T$. Then,
\[L(W_1 \cup_S W_2, c_1 + c_2) = L(W_1, c_1) \circ  L(W_2, c_2 ).\]
\end{prop}

\begin{proof}The inclusion of $L(W_1 \cup_S W_2, c_1 + c_2)$ in the  composition is obvious. The reverse inclusion comes from the fact that, if $C_1$ and $C_2$ are submanifolds representing the classes $c_1$ and $c_2$,  two flat connections on  $W_1\setminus C_1$ and $W_2\setminus C_2$ which  coincide on $S$ can be glued together to a flat connection on $W_1\setminus C_1 \cup W_2\setminus C_2 $.
\end{proof} 

\begin{remark}This geometric composition is not embedded in general.
\end{remark}

\begin{theo}\label{hyptechniques} The following functor from $\Cobelem$ to $\Symp$  factors to a functor from $\Cob$ to $\Symp$.
\[ \left\lbrace \begin{aligned}
(\Sigma,p) & \mapsto \Nc(\Sigma,p) \\
(W,f,p,c) & \mapsto L^c(W,f,p,c).
\end{aligned}
\right. \]
We will denote $\underline{L}(W,p,c)$ the image of a cobordism by this functor.
\end{theo}

\begin{proof}It suffices to check that the functor satisfies the assumptions of Proposition~\ref{factorisation}. The assumptions $(i)$ and $(ii)$ are clearly  satisfied, and  $(iii)$ follows from \cite[Lemma 6.11]{MW}. It remains to check the hypotheses $(iv)$ and $(v)$.

Concerning assumption $(iv)$: let $(\Sigma_0,p_0)$  be a surface with parametrized boundary  of genus $g\geq 2$, and $s_1,\ s_2$ two disjoint, non-separating, attaching circles in $\Sigma_0$. Let \[\alpha_1, \cdots , \alpha_g, \beta_1, \cdots ,\beta_g\] be a  system of generators of $\pi_1(\Sigma_0, p_0(0))$ such that $\partial \Sigma_0  $ is the product of the commutators of the  $\alpha_i$ and $\beta_i$, and such that $s_i$ is freely homotopic to $\alpha_i$ ($i=1,\ 2$).

Let  $W_1$ be the cobordism between $\Sigma_0$ and $\Sigma_1$ corresponding  to the attachment of a 2-handle along $s_1$ ($\Sigma_1$ has genus $g-1$), $W_2$ the cobordism between $\Sigma_1$ and $\Sigma_2$ corresponding to the attachment of a 2-handle along $s_2$ ($\Sigma_2$ has genus $g-2$).

Let $W_1'$ be the cobordism between $\Sigma_0$ and $\widetilde{\Sigma_1}$ corresponding to the attachment of a 2-handle along $s_2$ ($\widetilde{\Sigma_1}$ has genus $g-1$), $W_2'$ the cobordism between $\widetilde{\Sigma_1}$ and $\Sigma_2$ corresponding to the attachment of a 2-handle along $s_1$.

\[\xymatrix{ & \Sigma_1 \ar[rd]^{W_2} &   \\  \Sigma_0 \ar[ru]^{W_1} \ar[rd]_{W_1 '}    &  & \Sigma_2  \\  & \widetilde{\Sigma_1} \ar[ru]_{W_2'} &  }\]

Let $C_1 = \lbrace A_1 = I \rbrace\subset \N(\Sigma_0 )$ be the coisotropic submanifold corresponding to $W_1$, and $C_2 = \lbrace A_2 = I \rbrace\subset \N(\Sigma_0 )$ the coisotropic submanifold corresponding to $W_1'$.

The submanifolds $C_1$ and $C_2$ intersect transversely in $\N(\Sigma_0 )$, the compositions $L(W_1) \circ L(W_2)$ and $L(W_1') \circ L(W_2')$ are then embedded, and  coincide since they both correspond  to the coisotropic submanifold $C_1 \cap C_2 \subset \N(\Sigma_0 )$, which is fibered over $\N(\Sigma_2 )$, and simply connected since diffeomorphic to $SU(2)^2\times \N(\Sigma_2 )$. Then, according to Proposition~\ref{decoupcoiso}, its closure in $\Nc(\Sigma_0 )$ defines a morphism of $\Symp$.

Following the same reasoning as in the proof of \cite[Lemma 6.11]{MW}, we check the quilted cylinders assumption for the composition $L(W_1) \circ L(W_2)$  (the assertion concerning $L(W_1') \circ L(W_2')$ is similar). Let us show that every quilted cylinder  intersects the triplet $(R_0 , R_1 , R_2)$ in a positive multiple  of $-2$. Let $ u = (u_0,u_1,u_2)$ be an index zero pseudo-holomorphic quilt as  in figure~\ref{cylinder}, such that $u_i$ takes values in $\Nc(\Sigma_i)$, and with seam conditions given by $L^c(W_1)$ , $L^c(W_2)$ and  $L^c(W_1) \circ L^c(W_2)$.

By monotonicity, the area of the discs $u_i$ for the monotone 2-forms $\tilde{\omega_i}$ is zero, and thus $u_i$ is contained in a fiber of the degeneracy locus $R_i$. Indeed, recall  that $R_i$ admits a sphere fibration whose vertical bundle exactly corresponds  to the kernel of $\tilde{\omega_i}$, see  Proposition~\ref{degener}. Furthermore, the fiber containing $u_0$ meet the submanifold  $ C_1 \cap C_2$ and is then entirely included in this last one. Thus, it projects to a fiber of $R_2$. The same applies for $u_1$: it is contained in $\lbrace A_2 = I \rbrace \subset \Nc(\Sigma_1 )$ and projects to a fiber of $R_2$.
 
Therefore, $u_2$ and the images of $u_0$ and $u_1$ by the projections on $\Nc(\Sigma_2 )$ glue together to a pseudo-holomorphic sphere of $R_2$, this sphere intersects $R_2$ in a positive multiple of -2, and this intersection number is precisely $u.(R_0 , R_1, R_2 )$.

We now check the hypothesis $(v)$: note that if $c_i + c_{i+1} = d_i + d_{i+1}$, then according to Proposition~\ref{compocob}, $L(W_i,c_i) \circ L(W_{i+1},c_{i+1})$ and $L(W_i,d_i) \circ L(W_{i+1}, d_{i+1})$ coincide with   $L(W_i \cup W_{i+1}, c_i + c_{i+1}) $. Finally, the correspondence associated to a trivial cobordism  $(\Sigma\times [0,1],c)$ is the graph of a symplectomorphism, hence its left/right composition   with  every other  correspondence satisfies the hypotheses of Theorem~\ref{compogeom}. 

\end{proof}

\subsection{Definition of twisted symplectic instanton homology}
\label{sec:defhsi}

Let  $Y$ be a closed oriented  3-manifold, $c\in H_1(Y; \Z{2})$ and $z\in Y$. Denote $W$ the manifold with boundary obtained from $Y$ by performing a real oriented blow-up at $z$, namely $W = (Y\setminus z)\cup S^2$, and $p\colon \rr/\zz \times [0,1] \rightarrow S^2$ an oriented embedding. $(W,p,c)$ is a morphism in the category  $\Cob$ from the disc to itself. The set  of generalised intersection points  $\I(\underline{L}(W,p,c))$ is contained in the product of the zero levels of the moment maps ``$\theta_i = 0$'', hence in $\mathrm{int}\left\lbrace \omega_i = \tilde{\omega}_i \right\rbrace$. One can then consider their quilted Floer homology. It follows from Theorem~\ref{hyptechniques}:
\begin{cor}The $\Z{8}$-relatively graded  abelian group $HF(\underline{L}(W,p,c))$, up to isomorphism, only depends on the topological type  of $Y$, the point $z$, and the class $c$. We denote it $HSI(Y,c,z)$.
\end{cor}
\arnaque

\begin{remark}[Naturality] We expect that the groups $HSI(Y,c,z)$ are natural, i.e. canonically defined as groups, and not only up to isomorphism, see \cite[Section~7.1]{these}. If so, one could alternatively see $\bigcup_z HSI(Y,c,z)$ as a bundle over $Y$, (and in particular over a Heegaard splitting as in \cite[Parag. 5.3]{MW}). We will sometimes denote $HSI(Y,c)$ instead of $HSI(Y,c,z)$.
\end{remark}

\begin{remark}\label{functorainfini}The functor from $\Cob$ to $\Symp$ we constructed in this section allowed us to define HSI homology by taking the quilted Floer homology. Nevertheless, such a functor might contain more information and should give rise to other kind of invariants, which algebraic form might be more sophisticated. Namely, in \cite{WWfunctoriality} Wehrheim and Woodward  associate  to a Lagrangian correspondence  $L\subset M_0^- \times M_1$ a functor between two categories  $Don^\#(M_0)$ and $Don^\#(M_1)$ called ``extended Donaldson categories''. One can hope that their construction give rise to   invariants for  3-manifolds with boundaries endowed with isotopy classes of paths   between their boundaries, similar to those appearing in Fukaya's recent work  \cite{Fukayaboundary}. Such invariants would motivate the   construction and study of the corresponding categories for the moduli spaces $\Nc(\Sigma)$. Similar functors between their derived  Fukaya categories should  also exist.
\end{remark}

\section{First properties}\label{sec:firstprop}

\subsection{Computation of HSI(Y,c) from a Heegaard splitting}\label{sec:compusplitting}

Let $Y = H_0 \cup_{\Sigma} H_1$ be a given genus $g$ Heegaard splitting  of $Y$,   $z\in \Sigma$ a base point, and $c \in H_1(Y;\Z{2})$ a homology class, which can be decomposed as the sum  of two classes $c = c_0 + c_1$, with  $c_0 \in H_1(H_0;\Z{2})$ and $c_1 \in H_1(H_1;\Z{2})$.

\begin{remark}The maps $H_1(H_i;\Z{2}) \rightarrow H_1(Y;\Z{2})$ induced by the inclusions are both surjective, one can always assume either that  $c_0 = 0$ or $c_1 = 0$.
\end{remark}

Denote by $W$, $\Sigma'$, $ H_0'$ and $H_1'$ respectively the  blow-ups of $Y$, $\Sigma$, $ H_0$ and $H_1$ at $z$, so that  $ W = H_0' \cup_{\Sigma'} H_1'$.

Given a parametrization $p\colon \rr/\zz \times [0,1] \rightarrow \partial W$ such that $p(\rr/\zz \times \frac{1}{2}) = \partial \Sigma'$, we denote $p_0$ (resp. $p_1$) the restriction of $p$ to $\rr/\zz \times [0,\frac{1}{2}]$ (resp. $\rr/\zz \times [\frac{1}{2},1]$). Hence, in the category  $\Cob$, \[(H_0, p_0, c_0)\in Hom(D^2, \Sigma')\text{, and } (H_1, p_1, c_1)\in Hom(\Sigma', D^2).\] Let  $f_0, f_1$ be Morse functions on $H_0$ and $H_1$ respectively, adapted  to the parametrizations $p_0$ and $p_1$ (so that they are  vertical), and having each one exactly $g$ critical points  (of index 1 for $f_0$ and index 2 for $f_1$). They  decompose $H_0$ and $H_1$ into $g$ elementary cobordisms: $H_0 = H_0^1 \odot H_0^2 \odot \cdots \odot H_0^g$, $H_1 = H_1^1 \odot H_1^2 \odot \cdots \odot H_1^g$.

\begin{lemma}
For all $i$ from  $2$ to $g$, the composition $L(H_0^1 \cup  \cdots \cup H_0^{i-1}  ) \circ L(H_0^i)$ is embedded, satisfies the assumptions of Theorem~\ref{compogeom}, and corresponds to $L(H_0^1 \cup  \cdots \cup H_0^{i}  )$. 
\end{lemma}

\begin{proof}Let $\alpha_1,\cdots \alpha_i, \beta_1,\cdots \beta_i$ be a generating system of the fundamental group of the genus $i$    boundary component of $H_0^i$ such that $H_0^i$ corresponds to the attachment of a 2-handle along $\beta_i$, and such that the curves $\alpha_1,\cdots \alpha_{i-1}$, $\beta_1,\cdots \beta_{i-1}$ induces  a generating system for the genus $i-1$    boundary component. Under the following representation-theoretic descriptions of the moduli spaces  

\begin{align*}\N(\Sigma_0^{i-1}) &= \lbrace  (A_1,B_1, \cdots,  A_{i-1}, B_{i-1})~|~ [A_1,B_1] \cdots  [A_{i-1}, B_{i-1}] \neq -I \rbrace  \\
 \N(\Sigma_0^{i}) &= \lbrace  (A_1,B_1, \cdots,  A_{i}, B_{i})~|~ [A_1,B_1] \cdots  [A_{i}, B_{i}] \neq -I \rbrace ,
\end{align*}
the correspondences are given by:
\begin{align*}
 & L(H_0^1 \cup  \cdots \cup H_0^{i-1}  ) =\lbrace (A_1, \epsilon_1 I, A_2, \epsilon_2 I, \cdots ) \rbrace, \text{where } \epsilon_i = \pm 1   \\
 & L(H_0^{i}  ) = \lbrace (A_1,B_1, \cdots,  A_{i-1}, B_{i-1}),(A_1,B_1, \cdots,  A_{i-1}, B_{i-1}, A_i, \epsilon_i I)\rbrace .
\end{align*}
The intersection $ ( L(H_0^1 \cup  \cdots \cup H_0^{i-1}  ) \times \N(\Sigma_0^{i}) )\cap L(H_0^{i}  )$ is hence transverse in $ \N(\Sigma_0^{i-1}) \times \N(\Sigma_0^{i}) $, and corresponds to 
\[ \left\lbrace (A_1, \epsilon_1 I, \cdots,  A_{i-1},  \epsilon_{i-1} I),(A_1,\epsilon_1 I, \cdots,  A_{i-1}, \epsilon_{i-1} I, A_i, \epsilon_i I)\right\rbrace \simeq SU(2)^i .\] 
The projection to $\N(\Sigma_0^{i})$ induces  an embedding to $ L(H_0^1 \cup  \cdots \cup H_0^{i}  )$, which is simply connected, and compatible (because disjoint) with  the hypersurface $R_i$.

Furthermore, the zero area pseudo-holomorphic disc  assumption  is automatically satisfied because $ L(H_0^1 \cup  \cdots \cup H_0^{i}  )$ is disjoint from $R_i$, and the one concerning cylinders can be checked as in the proof of Lemma~\ref{hyptechniques}: provided one of the three patches is sent to a point, one can remove it,  and the quilted cylinder  corresponds then to a quilted disc as in figure~\ref{disque}, with boundary conditions in $  L(H_0^1 \cup  \cdots \cup H_0^{i-1}  )$ and $  L(H_0^1 \cup  \cdots \cup H_0^{i}  )$,  and seam conditions in $L(H_0^{i}  )$. Such a quilted disc  projects to a zero area disc of $\N(\Sigma_0^{i-1})$  with boundary in $  L(H_0^1 \cup  \cdots \cup H_0^{i-1}  )$, which cannot exist because this last  Lagrangian is disjoint from the degeneracy locus $R_{i-1}$ of the 2-form $\tilde{\omega}_{i-1}$.

\begin{figure}[!h]
    \centering
    \def\svgwidth{.65\textwidth}
    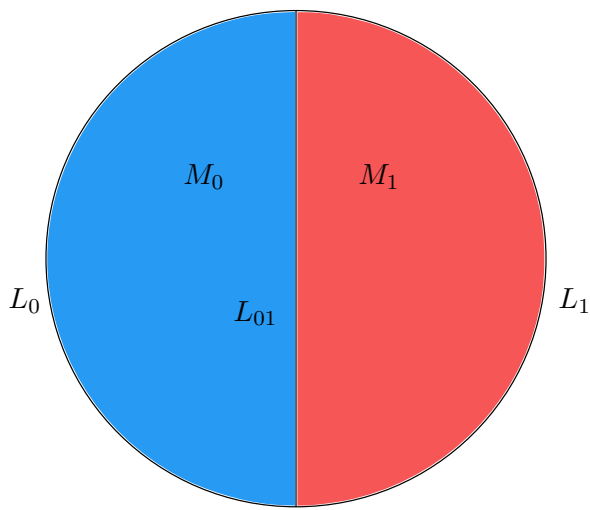
      \caption{A quilted disc.}
      \label{disque}
\end{figure}

\end{proof}

Hence, the generalized Lagrangian correspondence $\underline{L}(H_0, p_0,c_0)$ is equivalent in $\Symp$ to the Lagrangian $L(H_0, p_0,c_0)$, and the same holds for $H_1$. As a consequence of Theorem~\ref{compogeom}, we then get:
\begin{prop}\label{calculscindement} Under these conditions, $HSI(Y,c,z) \simeq HF(L_0,L_1; R)$, where $L_i = L(H_i,c_i,p_i) \subset \Nc(\Sigma,p)$. In particular, for $c=0$,  $HSI(Y,0,z)$ corresponds to the group $HSI(Y,z)$ defined by Manolescu and Woodward.
\end{prop}
\arnaque

\subsection{Orientation reversal}\label{sec:orreversal}

Let $\underline{L}\in Hom_{\Symp}(pt,pt)$ be a generalized Lagrangian correspondence. One can define the cohomology $HF^*(\underline{L})$, i.e. the homology of the dual complex  $CF^*(\underline{L}) = \mathrm{Hom}_\zz ( CF_*(\underline{L}), \zz)$ of $CF_*(\underline{L})$, equipped with the dual differential.

Recall that one denotes $\underline{L}^T$ the generalised Lagrangian correspondence obtained by reversing the arrows. The following fact is a quilted generalization of the duality $HF^*(L_0,L_1)\simeq HF_*(L_1,L_0)$.
\begin{prop}
$HF_*(\underline{L}^T) \simeq HF^*(\underline{L})$.
\end{prop}
\begin{proof}
Assume that $\underline{L}$ has transversal self-intersection, so that  $\I(\underline{L})$ and $  \I(\underline{L}^T)$ are finite sets, which are canonically identified. From this identification, the complexes $CF^*(\underline{L})$ and $CF_*(\underline{L}^T)$ are identified as $\zz$-modules, denote respectively  $\partial$ and $\partial^T$ their differentials.

Let $\underline{x},\underline{y}\in \I(\underline{L})$ be two  generalised intersection points. A quilted trajectory  $\underline{u}$ with seam conditions in  $\underline{L}$ and going from $\underline{x}$ to $\underline{y}$ can be seen, by applying the holomorphic involution $(s,t)\mapsto(-s, 1-t)$ to the domain, as a quilted trajectory $\underline{u}^T$ with seam conditions in  $\underline{L}^T$ and going from  $\underline{y}$ to $\underline{x}$, hence contributing to $\partial^T$.

Consequently, the moduli spaces involved for defining $\partial$ and $\partial^T$ are in one-to-one correspondence, it remains to compare their orientations. Recall these are defined after specifying a \emph{relative spin structure} on $\underline{L}$ (which in our case exists and is essentially unique), an \emph{end datum} at each strip-like end, to which are associated Fredholm operators $D_{\underline{x}}$ and $D_{\underline{y}}$, and an orientation for each of these operators, we refer to \cite{WWorient} for more detail and the corresponding definitions of these notions.

When passing from $\underline{u}$ to $\underline{u}^T$ the sign change is given by the orientation of the isomorphism 

\[ det(D_{\underline{x}}) \otimes det(D_{\underline{u}}) \otimes det(D_{\underline{y}}) \to det(D_{\underline{y}}) \otimes det(D_{\underline{u}}) \otimes det(D_{\underline{x}})  ,\]
which is $(-1)^{Ind (D_{\underline{x}})  Ind (D_{\underline{y}}) + Ind (D_{\underline{x}} ) Ind (D_{\underline{u}}) + Ind( D_{\underline{u}})  Ind (D_{\underline{y}})}$. Furthermore, the index of $\underline{u}$ is 1, and $Ind (D_{\underline{x}}) + Ind (D_{\underline{u}}) + Ind (D_{\underline{y}}) = 0$, since the total Lagrangian boundary condition is required to be spin. Hence the sign change is $-1$, which means that $\partial^T$ corresponds to 
 $-\partial$ under the former identification, and the homology groups are isomorphic.
\end{proof}

If $Y,c$ is a 3-manifold endowed with a homology class, $z$ a base point, denote $W$ the blow up, and $\overline{W}$ the blow-up with reversed orientation. Then, one has $\underline{L}(\overline{W},c) = \underline{L}(W,p,c)^T$. Then,  denoting $HSI_*$ what we denoted  $HSI$ so far, and $HSI^*$ the cohomology, one has:

\begin{prop}$HSI_*(\overline{Y},c,z) \simeq HSI^*(Y,c,z)$
\end{prop}
\arnaque

\subsection{Connected sum}\label{sec:connectedsum}

In order to prove Theorem~\ref{sommecnx}, recall the Künneth formula for quilted Floer homology, see \cite[Theorem 5.2.6]{WWqfc}  for the monotone non-relative setting, whose proof straightforwardly generalises to the setting of $\Symp$:

\begin{prop}(Künneth formula, \cite[Theorem 5.2.6]{WWqfc})
\label{kunnethfloer}
 Let  $\underline{L}$ and $\underline{L'}$ be two generalized Lagrangian correspondences from $pt$ to $pt$, then \[HF (\underline{L},  \underline{L'} ) \simeq HF (\underline{L}) \otimes HF (\underline{L'})\oplus \mathrm{Tor}(HF (\underline{L}) , HF (\underline{L'}))[-1],\] where $\mathrm{Tor}$ stands for  the $\mathrm{Tor}$ functor, and $[-1]$ means a $-1$ shift in degrees.
\end{prop}
\arnaque

\begin{proof}[Proof of Theorem~\ref{sommecnx}] Let $\underline{L}$ and $\underline{L'}$ be the generalized Lagrangian correspondence associated to $(Y, c)$ and $(Y' , c' )$, which are both  morphisms of $\Symp$ from $pt$ to $pt$. Then $\underline{L}, \underline{L'}$ is a generalized Lagrangian correspondence associated to $(Y, c)\# (Y' , c' )$. The claim follows from   Proposition~\ref{kunnethfloer}. 
\end{proof}

\subsection{Euler characteristic}\label{sec:eulercar}
As the groups $HSI(Y,c,z)$ are relatively $\Z{8}$-graded, their Euler  characteristic $\chi(HSI(Y,c,z))$ is defined only up to a sign.

\begin{prop}\label{eulercara} If $b_1(Y) = 0$, $\left|  \chi(HSI(Y,c,z)) \right| =  \mathrm{Card} H_1(Y;\zz)$, otherwise $\chi(HSI(Y,c,z)) = 0$.

\end{prop}

\begin{proof}The case when $c=0$ has been established by Manolescu and Woodward, \cite[Parag. 7.1]{MW}: for $c_0 =  c_1 = 0$, the Euler characteristic is given by the intersection number $[L(H_0,c_0 )] . [L(H_1,c_1 )]$  of the two Lagrangians inside the moduli space of the splitting, and this number is computed in \cite[Prop. III.1.1, (a),(b)]{cassonpup}. If $c\neq 0$, the intersection number remains unchanged, indeed $L(H_i,c_i)$ can be sent to $L(H_i, 0)$ through a (non-Hamiltonian) isotopy  of $SU(2)^{2h}$ in the following way: for some fixed  presentation of the fundamental group, $L(H_i,c_i)$ is defined by equations  
\[ \left\lbrace (A_1,B_1,\cdots)\ |\ A_1 = \epsilon_1 I, A_2 = \epsilon_2 I, \cdots \right\rbrace, \]
where $\epsilon_i = \pm 1$. It then suffices to take a path in $SU(2)$ going from $-I$ to  $I$ to set all the $\epsilon_i$ equal to $+1$.

\end{proof}

\subsection{Manifolds of Heegaard genus 1}\label{sec:genusone}

Manolescu and Woodward computed the HSI homology groups for manifolds of Heegaard genus  1 when the class $c$ is zero, see \cite[Prop. 7.2, 7.3]{MW}. Their computations can be extended to all classes.

\begin{prop}\label{genreun}\begin{itemize}

\item[$(i)$] For $Y=S^2 \times S^1$ and $c\in H_1(Y; \Z{2})$,

\[ HSI(Y,c) = \begin{cases} \zz[0] \oplus \zz[3]\text{ if  }c=0, \\ \lbrace 0\rbrace \text{ otherwise.} \end{cases} \]

\item[$(ii)$] $HSI(L(p,q),c)$ has rank $p$ for every class $c$. Furthermore, all non-zero classes have the same parity in degree.
\end{itemize}
\end{prop}

\begin{proof}
$(i)$: For $c=0$, $HSI(S^2 \times S^1)$ has been computed by Manolescu and Woodward. For $c\neq 0$, with  $\Sigma$ the genus 1 splitting   and $A$, $B$ the holonomies along a basis of the fundamental group of $\Sigma$ whose first curve bounds a disc in both solid tori, the two Lagrangians $\lbrace A=I\rbrace$ and $\lbrace A= -I\rbrace$ are disjoint.

Concerning $L(p,q)$, one can take a Heegaard splitting  and a coordinate system such that the two Lagrangians are defined by: $L_0 = \lbrace B = I \rbrace$ and $L_1 = \lbrace A^p B^{-q}  = \pm I\rbrace$. They intersect cleanly in a union of copies of $S^2$ and, depending on the parity of $p$, one or two points. One can displace one of the  two Lagrangians by a Hamiltonian isotopy in a way that  the intersection becomes transverse, and each copy of $S^2$ gives rise to two points. There are now $p$ intersection points,  besides  one knows that this number corresponds to the intersection number of the two Lagrangians. Hence the chain complex doesn't have nontrivial summands of consecutive degree, and its differential is trivial.

\end{proof}

\section{Dehn surgery}\label{sec:surgery}

In this section we prove the surgery exact sequence, Theorem~\ref{trianglechir}. Recall its setting: $Y$ is an oriented compact 3-manifold with boundary, whose boundary is a 2-torus, $c$ is a class in $H_1(Y;\Z{2})$, $\alpha$, $\beta$ and $\gamma$ are three oriented simple closed curves in $\partial Y$ such that $\alpha . \beta = \beta . \gamma = \gamma .\alpha = -1$, $Y_\alpha$, $Y_\beta$ and $Y_\gamma$ denote the Dehn fillings of $Y$ along these curves. For $\delta\in \lbrace\alpha, \beta, \gamma\rbrace$, let $c_\delta \in H_1(Y_\delta;\Z{2})$ denote the push-forward of $c$ by the inclusions, and $k_\delta \in H_1(Y_\delta;\Z{2})$ the class corresponding to the core of the solid torus. The aim of Theorem~\ref{trianglechir} is to prove a long exact sequence:
\[ \cdots\rightarrow HSI(Y_\alpha ,c_\alpha+ k_\alpha) \rightarrow HSI(Y_\beta,c_\beta) \rightarrow HSI(Y_\gamma,c_\gamma)\rightarrow \cdots . \]

\begin{remark}By cyclic symmetry of the three curves, the modification $k_\alpha$ can also be put on $Y_\beta$ or $Y_\gamma$. It is also possible to prove a more symmetric sequence:
\[ \cdots\rightarrow HSI(Y_\alpha ,c_\alpha+ k_\alpha) \rightarrow HSI(Y_\beta,c_\beta+ k_\beta) \rightarrow HSI(Y_\gamma,c_\gamma+ k_\gamma)\rightarrow \cdots.\]
 Indeed, let $d$ be the class of the curve $\alpha$ in $H_1(Y;\Z{2})$. Its induced classes on $Y_\alpha$, $Y_\beta$ and $Y_\gamma$ are respectively 0, $k_\beta$ and $k_\gamma$. The exact sequence of the theorem applied with $c+d$ instead of $c$ gives the announced exact sequence. 
\end{remark}

In order to prove this theorem, we will see that a Dehn twist of the punctured torus $T'$ along a non-separating simple closed curve induces a symplectomorphism of the moduli space $\Nc ({T} ')$. This symplectomorphism can be expressed as a Hamiltonian flow   outside the Lagrangian sphere  corresponding to the  connections whose holonomy along the curve $\gamma$ is $-I$. Whereas this is not a priori a generalized Dehn twist, it is nevertheless  possible to build such a twist that permits to obtain the exact sequence, by applying an analog of Seidel's exact sequence (Theorem~\ref{quilttri} ) for quilted Floer homology.

\subsection{Generalized Dehn twists and quilted Floer homology}\sectionmark{generalized Dehn twists}\label{sec:twistexactseq}

All the symplectic manifolds, Lagrangians, and Lagrangian correspondences  appearing by now will satisfy, unless explicitly stated, the assumptions of the category  $\Symp$. Let  $M_0$, $M_1$, ..., $M_k$ be objects of $\Symp$,  

\[\underline{L} = \left(  \xymatrix{   M_0 \ar[r]^{L_{01}} &  M_1 \ar[r]^{L_{12}} &  M_2 \ar[r]^{L_{23}} & \cdots \ar[r]^{L_{(k-1)k}} & M_k \ar[r]^{L_{k}} & pt} \right), \]
a generalized Lagrangian correspondence, $S \subset M_0$ a Lagrangian sphere  disjoint from the hypersurface $R_0$, and $\tau_S\in Symp(M_0)$ a generalized Dehn twist along $S$, as defined in section~\ref{sec:generalizedtwists} (or \cite[Section 1.2]{Seidel}). The aim  of this section is to prove the   following theorem:

\begin{theo}
\label{quilttri} Let  $L_0\subset M_0$ be a Lagrangian submanifold, $S$ and $\underline{L}$ as before. Assume further that $\dim S >2$. There exists a long exact sequence:

\[ \ldots \rightarrow HF(\tau_S L_0, \underline{L}) \rightarrow HF(L_0, \underline{L}) \rightarrow HF(L_0,S^T, S, \underline{L})\rightarrow \cdots \]

\end{theo}

\begin{remark}
\label{remtriww} 
This theorem has been established by Wehrheim and Woodward in the monotone setting \cite[Theorem 1.3]{WWtriangle} for the case of fibered Dehn twists. Our proof in the setting of the category  $\Symp$ follows the same lines, the main additional point to check is that there is no bubbling on the divisors any time a reasoning involving 1-dimensional moduli spaces appears.
\end{remark}
\begin{remark}The assumption $\dim S >2$ involved in this theorem ensures the monotonicity of a Lefschetz fibration. A similar statement probably holds, however we will limit ourselves to this case, since $\dim S =3$ would suffice to prove Theorem~\ref{trianglechir}.
\end{remark}

\subsubsection{Generalized Dehn twists in symplectic manifolds}\label{sec:generalizedtwists}

We briefly review some material concerning generalized Dehn twists, and refer to \cite[Section 1]{Seidel} for more details.

\paragraph{\bf Dehn twist inside $T^*S^n$}
Consider the cotangent bundle  $T = T^*S^n$ endowed with its standard symplectic form  $\omega = \sum_i{dq_i \wedge dp_i}$. If  $S^n$ in endowed with the round metric, $T$ may be identified with \[\left\lbrace (u,v) \in \rr^{n+1} \times \rr^{n+1}~|~ |v| = 1, \langle u.v \rangle = 0  \right\rbrace. \] Denote $T(\lambda) = \lbrace (u,v) \in T ~|~ |u| \leq \lambda \rbrace$, In particular $T(0)$ refers to the zero section.

The function $\mu (u,v) = |u|$ generates a circle action on the complement of the zero section, and its flow at time $t$ is given by:

\[ \sigma_t (u,v) = (cos(t) u -sin(t) |u|v, cos(t) v + sin(t) \frac{u}{|u|}),  \]
and the time $\pi$ flow extends to the zero section by the antipodal map, which we will denote $\mathbb{A}$.

Let $\lambda > 0$, and $R\colon \rr\rightarrow \rr$ a smooth function vanishing for $t\geq \lambda$, and such that $R(-t) = R(t) - t$. Consider the Hamiltonian $H = R\circ \mu $  on $T(\lambda) \setminus T(0)$: its time $2\pi$ flow is given by $\varphi^H_{2\pi} (u,v) =  \sigma_t (u,v)$, with  $t = R'(|u|)$, and extends smoothly to the zero section by the antipodal map. The symplectomorphism obtained $\tau$ is a ``model Dehn twist'', with angle function $R'(\mu(u,v))$.

\begin{defi}\label{concave}A model  Dehn twist will be said to be concave if the function $R$ involved in the  definition  is strictly concave and decreasing, that is, satisfying  $R'(t)\geq 0$ and $R''(t)< 0$ for all $t\geq 0$.
\end{defi}

Seidel proves  the following result in a slightly general case, allowing the angle  functions  to oscillate slightly, in a  ``$\delta$-wobbly'' way, with  $0\leq \delta <\frac{1}{2}$.  The following statement, corresponding to $\delta = 0$, will be enough for our purpose.

\begin{lemma} \label{intertwistbis} (\cite[Lemma 1.9]{Seidel}) Suppose that the twist $\tau$ is concave. Let  $F_0 = T(\lambda)_{y_0}$ and $F_1 = T(\lambda)_{y_1}$ be fibers over two points $y_0,y_1 \in S^n$. Then $\tau(F_0)$ and $F_1$ intersect transversely in a single point $y$. Moreover, this point satisfies \[2\pi R'(y) = d(y_0,y_1),\] where $d$ stands for the standard distance on $S^n$.
\end{lemma}

\paragraph{\bf Dehn twist along a Lagrangian sphere}
Let $S\subset M$ be a Lagrangian sphere,  it admits a Weinstein neighborhood, namely a symplectic embedding  $\iota\colon T(\lambda) \rightarrow M$ for some $\lambda >0$, with  $\iota(T(0)) = S$. Hence, a model Dehn twist of $ T(\lambda)$ defines a symplectomorphism of $M$, denoted by $\tau_S$, with support contained in $\iota (T(\lambda))$. 

A symplectomorphism of $M$ is called a generalized Dehn twist along $S$ if it is Hamiltonian isotopic to such a model  Dehn twist.

\begin{remark}While two model Dehn twists of $T(\lambda)$ always differ from a Hamiltonian isotopy  of $T(\lambda)$, a Dehn twist along $S$ may depend on the parametrization of $S$, see \cite{rizell}.
\end{remark}

\subsubsection{Homology with coefficients in the group ring of $\rr$}\label{sec:coeffs}

The principal ingredient in the proof of Seidel's theorem relies in  the fact that the Floer complexes are $\rr$-graded by the symplectic action, since both  symplectic manifolds and  Lagrangians he considers are exact. The leading order terms of the morphisms involved in the exact sequence   with respect to this filtration induced by this  grading correspond to small energy pseudo-holomorphic curves. It then suffices  suffices to prove that these induce an exact sequence.

When the symplectic manifolds and  Lagrangians aren't exact anymore but only monotone, the symplectic action is only defined  modulo $M = \kappa N$, with  $\kappa$ the monotonicity constant  and $N$ the minimal Maslov number, see section~\ref{sec:smallcontrib}. Wehrheim and Woodward's approach  consist in  encoding this energy in the power of a formal parameter $q$, via the group ring of $\rr$:

\[ \Lambda =  \left\lbrace \left. \sum_{k=1}^{n}{a_k q^{\lambda_k}} ~\right|~ n\geq 1,~a_k \in \zz,~ \lambda_k \in \rr  \right\rbrace .\]

The Floer complex with coefficients in this ring is then the free $\Lambda$-module  $CF (\underline{L};\Lambda):= CF(\underline{L})\otimes_\zz \Lambda$, endowed with the differential $\partial_\Lambda$ defined by:

\[ \partial_\Lambda x_- = \sum_{x_+}{\sum_{\underline{u}\in \mathcal{M}(x_-,x_+)}{o(\underline{u}) q^{A(\underline{u})}} x_+}, \]
where $x_+,x_-\in \I(\underline{L})$ are generalized intersection points, $\mathcal{M}(x_-,x_+)$ denotes the moduli space of index 1 generalized Floer trajectories with zero intersection  with $\underline{R}$ (modulo translation), $o(\underline{u}) = \pm 1$ is the orientation of the point $\underline{u}$ in the moduli space constructed in \cite{WWorient} from of the unique relative spin structure on $\underline{L}$, and $A(\underline{u})$ is symplectic area  for the monotone forms  $\tilde{\omega}_i$.

The homology of $(CF(\underline{L};\Lambda),\partial_\Lambda)$ is then the $\Lambda$-module denoted $HF(\underline{L};\Lambda)$. Generally, this homology may differ from the homology with $\zz$-coefficients, we will however see in section~\ref{sec:prooftriangle} that the monotonicity of $\underline{L}$ ensures that $ CF(\underline{L};\Lambda) \simeq CF(\underline{L};\zz) \otimes_\zz \Lambda$, and $ HF(\underline{L}) \simeq HF(\underline{L};\Lambda) / ( q-1)$.

\subsubsection{Short exact sequence at the chain level}\label{sec:shortexactseq}

The following proposition follows from Lemma~\ref{intertwistbis}:

\begin{prop}\label{inter}
Let $\iota\colon T(\lambda) \rightarrow M_0$ be a  symplectic embedding, $\tau_S$ a concave model Dehn twist   associated to $\iota$. Assume:
\begin{itemize}

\item[i)] that $\I(L_0, \underline{L})$ is disjoint from $\iota\left(  T(\lambda) \right)$, 

\item[ii)] that  $L_0 \cap \iota( T(\lambda))$ is a union of fibers:

\[  \iota^{-1}( L_0) = \bigcup_{y\in \iota^{-1}( L_0\cap S)}{T(\lambda)_y} \subset T(\lambda)\]

\item[iii)] that  $L_{01}$ and $S\times M_1$ intersect transversely in $M_0\times M_1$, and that, denoting $\pi \colon  \iota\left(  T(\lambda) \right) \rightarrow S$ the projection,

\[ L_{01}\cap \left(  \iota\left(  T(\lambda)  \right)  \times M_1 \right) = (\pi\times id_{M_1})^{-1}(L_{01} \cap \left( S\times M_1 \right) ). \]  

\end{itemize}

Then, there exists two natural injections \[i_1\colon \I(\tau_S L_0,S^T, S, \underline{L}) \rightarrow \I(\tau_S L_0, \underline{L})\] and \[i_2\colon \I( L_0, \underline{L}) \rightarrow \I(\tau_S L_0, \underline{L})\] such that

\[ \I(\tau_S L_0, \underline{L}) = i_2 \left( \I(L_0, \underline{L}) \right)  \sqcup  i_1 \left( \I(\tau_S L_0,S^T, S, \underline{L}) \right) . \]

\end{prop}
\begin{proof}

Denote $\nu S = \iota\left(  T(\lambda) \right)$,

\begin{align*}
\I(\tau_S L_0, \underline{L}) = & \I(\tau_S L_0, \underline{L}) \cap \left( M_0\setminus \nu S\right) \times M_1 \times \cdots \times M_k \\
  &\sqcup \I(\tau_S L_0, \underline{L}) \cap \nu S \times M_1 \times \cdots \times M_k .
\end{align*}
 
According to $i)$ and the fact that $\tau_S$ has support contained in $\nu S$, 

\[ \I(\tau_S L_0, \underline{L}) \cap \left( M_0\setminus \nu S\right) \times M_1 \times \cdots \times M_k = \I(L_0, \underline{L}). \]
The map $i_2$ can then be chosen to be the identity. From $ii)$, one has:

\[\I(\tau_S L_0, \underline{L}) \cap \nu S \times M_1 \times \cdots \times M_k = \bigcup_{x_0 \in L_0\cap S} {\I( \tau_S \left( T(\lambda)_{x_0}\right) , \underline{L} ) }. \]

Let $\underline{y} \in \I(S, \underline{L})$, by assumption $T(\lambda)_{y_0} \times  \left\lbrace  y_1 \right\rbrace \subset L_{01}$, and by Lemma~\ref{intertwistbis}, $\tau_S \left( T(\lambda)_{x_0}\right)$ and $T(\lambda)_{y_0}$ intersect in exactly one point $z$. One then defines $i_1$ by taking $i_1(x_0, y_0, y_1, \cdots) = (z, y_1, \cdots)$. This map realizes  a bijection between $\I(\tau_S L_0,S^T, S, \underline{L})$ and \[\bigcup_{x_0 \in L_0\cap S} {\I( \tau_S \left( T(\lambda)_{x_0}\right) , \underline{L}) }. \]
Indeed its inverse map is given by the map $(z, y_1, \cdots) \mapsto (x_0, y_0, y_1, \cdots),
$ where $ x_0 = \pi (z) $ and $ y_0 = \pi ( \tau_S^{-1} (z))$.

\end{proof}

\begin{remark}\label{justifinter}
Up to displacing the Lagrangians by Hamiltonian isotopies and taking $\lambda$ sufficiently small, one can always assume that the hypotheses of  Proposition~\ref{inter} are satisfied. Indeed all the  intersections can be made transverse, and then one can choose the embedding $\iota$ in order to have $ii)$ and $iii)$.

\end{remark}

Hence one has the  direct sum decomposition of the following $\Lambda$-modules: 

\[ CF (\tau_S L_0, \underline{L};\Lambda) = CF (\underline{L};\Lambda) \oplus CF (\tau_S L_0,S^T, S, \underline{L};\Lambda), \]
and a $\Lambda$-modules short exact sequence (and not necessarily chain complexes):
\begin{equation}\label{suitecourte}
0 \rightarrow CF (\tau_S L_0,S^T, S, \underline{L};\Lambda ) \rightarrow CF (\tau_S L_0, \underline{L} ;\Lambda) \rightarrow CF (L_0, \underline{L};\Lambda) \rightarrow   0 .
\end{equation}

\begin{remark} The sphere $S$ being invariant by the twist, one has the following isomorphisms:
 \begin{align*} CF (\tau_S L_0,S^T, S, \underline{L};\Lambda) &\simeq CF (\tau_S L_0,\tau_S S^T, S, \underline{L};\Lambda) \\ 
  &\simeq CF (L_0,S^T, S, \underline{L};\Lambda).  
\end{align*}

\end{remark}

\subsubsection{Quilted Lefschetz fibrations}\label{sec:quiltlefschetz}

The strategy for proving the long exact sequence consists in  approximating the maps of the short exact sequences by chain complexes  morphisms. In order to commute with  the differentials, these morphisms will be constructed by counting pseudo-holomorphic quilts, more precisely pseudo-holomorphic sections  of quilted Lefschetz fibrations. We recall the definitions of these objects, taken from \cite{WWquilts} and adapted to the framework of the category $\Symp$. 

\begin{defi}\label{fiblefschetz}
Let $S$ be a compact Riemann surface, possibly with boun\-dary. A \emph{Lefschetz fibration over $S$}, in the framework of  $\Symp$, consists of a tuple $(E, \pi, \omega, \tilde{\omega}, R, \tilde{J})$, with:

\begin{itemize}
\item $E$  a compact orientable manifold  of dimension $2n +2$,
\item $\pi\colon E\rightarrow S$  a surjective  differentiable map, such that $\partial E = \pi^{-1} (\partial S)$, which is a submersion except at a finite number of critical points $E^{crit}$, disjoint from $\partial E$,
\item  $\tilde{J}$  an almost complex structure on $E$, integrable in a neighborhood of $E^{crit}$, such that the differential of $\pi$ is $\cc$-linear, and that in a neighborhood of each critical point, in  holomorphic charts, $\pi$ can be written:

\[\pi(z_0, \cdots, z_n) = \sum_{i}{z_i^2},\]
\item $\omega$ and $\tilde{\omega}$  two closed 2-forms on  $E$ which are  non-degenerated in the neighborhood of the critical points,

\item $R$  an almost complex  hypersurface for $\tilde{J}$, disjoint from $E^{crit}$, transverse to the fibers of $\pi$, and such that for every regular fiber $F$ of $\pi$, $(F, \omega_{|F}, \tilde{\omega}_{|F}, R \cap F, \tilde{J}_{|F})$ is an object of $\Symp$.
\end{itemize}

 \end{defi}

\begin{remark}
We assume that $\tilde{\omega}$ is monotone only along the fibers, however according to \cite[Prop. 4.6]{WWtriangle}, provided $n\geq 2$, this implies that the form $\tilde{\omega}$ is  monotone on $E$.

\end{remark}

\begin{defi} A \emph{quilted surface with strip-like ends} consists of:
\begin{enumerate}

\item a compact quilted   surface   $\underline{S}$,

\item a finite set of incomings and outgoing marked points,  \[\mathcal{E} = \mathcal{E}_- \sqcup \mathcal{E}_+ \subset \partial \underline{S}.\]

\item  strip-like ends associated  to each marked point $e \in \mathcal{E}$, namely quilted holomorphic maps
\[ \epsilon_e \colon  \begin{cases}
 [0, + \infty ) \times [0,N_e] \rightarrow \underline{S}  &\text{if $e \in \mathcal{E}_+$ is an outgoing end}\\
 ( -\infty ,0] \times [0,N_e] \rightarrow \underline{S} &\text{if $e \in \mathcal{E}_-$ is an incoming end}
\end{cases}  \]
having for limit $e$ in $\pm \infty$, and whose image's closures is a neighborhood of $e$ in $\underline{S}$. If $N_e$ represents the number of patches $S_1, S_2, \cdots$, $S_{N_e}$ touching $e$, $[0, \pm \infty ) \times [0,N_e]$ can be seen as a quilted surface  with  $N_e$ parallel strips of width $1$ seamed altogether. The map $ \epsilon_e$ corresponds to $N_e$ maps:
\[ \epsilon_{k,e} \colon [0, \pm \infty ) \times [k-1,k] \rightarrow S_k. \]
\end{enumerate}

\end{defi}

\begin{defi} \label{quiltfiblefschetz}
 Let $\underline{S}$ be a  quilted surface with strip-like ends, a \emph{quilted Lefschetz fibration over $\underline{S}$, with  seam  and boundary conditions}, consists of:
\begin{enumerate}

\item For each patch $S_k$, a Lefschetz fibration  $\pi_k \colon E_k \rightarrow S_k$  as in Definition~\ref{fiblefschetz}.

\item A  set of   Lagrangian seam  and boundary conditions, denoted $\underline{F}$, consisting of:

\begin{itemize}

\item[$(a)$] for a seam $\sigma = \lbrace I_{k_0,b_0},I_{k_1,b_1} \rbrace \in \mathcal{S}$, a submanifold 
\[ F_{\sigma} \subset {E_{k_0}}_{|I_{k_0,b_0}} \times_{|I_{k_0,b_0}} \varphi_{\sigma}^* {E_{k_1}}_{|I_{k_1,b_1}}, \] which is isotropic for the forms $\tilde{\omega_i}$, transverse to the fibers, and such that the intersection with  every  fiber is a  Lagrangian correspondence satisfying  the assumptions of $\Symp$. Recall that $\varphi_\sigma \colon I_{k,b} \rightarrow I_{k',b'}$ refers to  the real analytic   diffeomorphism which identifies the seams.

\item[$(b)$] for a boundary $I_{k,b}\notin \bigcup_{\sigma\in \mathcal{S}}\sigma$, a submanifold $F_{k,b} \subset {E_k} |_{I_{k,b}}$, transverse to the fibers, and such that its intersection with  every fiber is a Lagrangian submanifold satisfying the hypotheses of $\Symp$.

\end{itemize}
\item trivializations over the ends $\epsilon_{k,e}$:
\[{\epsilon_{k,e}}  ^*{(E_k)}\simeq (E_k)_e \times [0, \pm \infty ) \times [k-1,k],  \]
such that the seam  and boundary conditions are   constant in these identifications: \[F_{\sigma} \simeq (F_{\sigma})_e \times [0,\pm \infty ) \times \lbrace k\rbrace\text{,  and }F_{k,b} \simeq (F_{k,b})_e \times [0,\pm \infty ) \times \lbrace k\rbrace.\]

\end{enumerate}

\end{defi}

\paragraph{\bf Relative invariant  associated to a quilted Lefschetz fibration}  
Let $\underline{\pi}\colon (\underline{E},\underline{F}) \rightarrow \underline{S}$ be a quilted Lefschetz fibration as before, and $\underline{J} = ( J_k)_k$ a family of almost complex structures on $\underline{E}$, which coincide with  the reference almost complex structures   $\widetilde{\underline{J}}$ in a neighborhood of the hypersurfaces $\underline{R}$, and such that the projections are pseudo-holomorphic, and compatible with  the 2-forms $\omega_k$ along the fibers,

If $u\colon (\underline{S},\mathcal{S})\rightarrow (\underline{E},\underline{F})$ is a  pseudo-holomorphic section, its associated linearized Cauchy-Riemann operator is defined by:
\[
 D_u \colon \begin{cases}  \Omega^0(u^*T^{vert}\underline{E}, u^*T^{vert}\underline{F} ) \rightarrow &  \Omega^{0,1}(u^*T^{vert}\underline{E}) \\

  \xi \mapsto & \frac{d}{dt}_{t=0} \Pi_{t\xi}^{-1}\overline{\partial}_J \exp_u (t \xi) ,

\end{cases}
\]
where $\Omega^0(u^*T^{vert}\underline{E}, u^*T^{vert}\underline{F} )$  denotes the space of quilted sections  of the fibration $u^*T^{vert}\underline{E}$ with  values in $u^*T^{vert}\underline{F}$ over the seams (for suitable Sobolev norms),  $\Omega^{0,1}(u^*T^{vert}\underline{E})$ stands for the  $(0,1)$-forms with values in this fibration, $\overline{\partial}_J u = \frac{1}{2} (du + J(u)\circ du \circ j)$ is the Cauchy-Riemann operator  associated to $J$, and $\Pi_{t\xi} \colon T_{u(x)}{M}\rightarrow T_{\exp_{u(x)} (t\xi)}{M}$ refers to a parallel transport.

As soon as the end conditions are transverse, $D_u$ is a Fredholm operator, see \cite[Lemma 3.5]{WWquilts}, and is surjective for generic almost complex structures, see \cite[Theorem 4.11]{WWtriangle}.

For such almost complex structures, the moduli space  of $\underline{J}$-holomorphic sections $  \underline{s}\colon  \underline{S} \rightarrow \underline{E}$, with seam/boundary conditions given by $\underline{F}$, with zero intersection  with  the family of hypersurfaces $\underline{R}$, and having for  limits \[\underline{x} \in \prod_{e\in \mathcal{E}_-(\underline{S})} \I(L^{(k_{e,0},b_{e,0})}  , \cdots , L^{(k_{e,l(e)},b_{e,l(e)})}), \]
and  
\[\underline{y} \in \prod_{e\in \mathcal{E}_+(\underline{S})} \I(L^{(k_{e,0},b_{e,0})}  , \cdots , L^{(k_{e,l(e)},b_{e,l(e)})})\] at the   corresponding ends is the union of smooth manifolds  $ \mathcal{M}(\underline{E},\underline{F},\underline{J},\underline{x},\underline{y})_k$ of dimensions $k\geq 0$. Their dimension corresponds to  the index of the operator $D_u$. This index generalizes the  Maslov index and can be computed from topological data, see \cite{WWorient}.

In this setup, $ \mathcal{M}(\underline{E},\underline{F},\underline{J},\underline{x},\underline{y})_0$ is a compact manifold  of dimension 0, which allows one to define a map:

\begin{align*}
C\Phi_{\underline{E},\underline{F}} \colon & \bigotimes_{e\in \mathcal{E}_-(\underline{S})} CF(L^{(k_{e,0},b_{e,0})}  , \cdots , L^{(k_{e,l(e)},b_{e,l(e)})})\\
&\rightarrow \bigotimes_{e\in \mathcal{E}_+(\underline{S})} CF(L^{(k_{e,0},b_{e,0})} , \cdots , L^{(k_{e,l(e)},b_{e,l(e)})} )
\end{align*}
by the following formula:

\[ C\Phi_{\underline{E},\underline{F}}( \otimes_{e\in \mathcal{E}_-(\underline{S})}(x_e^0, \cdots x_e^{l(e)}) ) = \sum_{\otimes y_e^i} \sum_{\underline{s}\in \mathcal{M}(\underline{E},\underline{F},\underline{x},\underline{y})}{o(\underline{s}) q^{A(\underline{s})} \otimes_e (y_e^i)_i}, \]

For generic almost complex structures, this map commutes  with  the differential. To prove this fact, one applies the following standard argument in Floer theory: one observes  that the coefficients of  $\partial C\Phi_{\underline{E},\underline{F}} - C\Phi_{\underline{E},\underline{F}} \partial$ corresponds to  the cardinal of the boundary of a compact 1-dimensional manifold.

\begin{lemma} There  exists a comeagre subset of almost complex structures on $\underline{E}$ for which the Gromov compactification of 
 $ \mathcal{M}(\underline{E},\underline{F},\underline{J},\underline{x},\underline{y})_1$ is a  compact one-dimensional manifold with boundary, and its  boundary is identified with:

\begin{align*}
 \partial \mathcal{M}(\underline{E},\underline{F},\underline{J},\underline{x},\underline{y})_1  = &\bigcup_{\underline{x}'}{  \widetilde{\mathcal{M}}(\underline{x},\underline{x}')_1 \times  \mathcal{M}(\underline{E},\underline{F},\underline{J},\underline{x}',\underline{y})_0  } \\
 \cup& \bigcup_{\underline{y}'}{     \mathcal{M}(\underline{E},\underline{F},\underline{J},\underline{x},\underline{y}')_0 \times \widetilde{\mathcal{M}}(\underline{y}',\underline{y})_1}  ,
\end{align*}
where $\underline{x}'$ (resp. $\underline{y}'$)  runs over the generating set of the source (resp. target) chain complex,  $\widetilde{\mathcal{M}}(\underline{x},\underline{x}')_1 $ and $\widetilde{\mathcal{M}}(\underline{y}',\underline{y})_1$ denote the  quotients by $\rr$ of the  spaces of index 1  quilted Floer   trajectories, i.e. the coefficients of the differentials of the corresponding complexes.

\end{lemma}

\begin{proof} The fact that the right hand side is contained in the left hand side is a standard gluing result, see \cite[Theorem 3.9]{WWquilts}. The fact that there is no other kind of degeneracies comes from Gromov compactness and the following Lemma~\ref{nobubbling}:
\end{proof}
\begin{lemma}\label{nobubbling}There exists a comeagre subset of almost complex structures on $\underline{E}$ for which no bubbling appears in the moduli spaces of sections of index  smaller or equal to 1.

\end{lemma}

The argument is analogous to the one appearing in the  proof of \cite[Prop. 2.10]{MW}. We recall  it here in this setup. It is based on the following lemma, which ensures that every   pseudo-holomorphic section appearing  in these spaces  intersect the hypersurface transversely.

\begin{lemma}(see \cite[Lemma 2.3]{MW})\label{codimtangence}
There exists a  comeagre subset of regular almost complex structures on $\underline{E}$ for which the moduli spaces of  pseudo-holomorphic sections are smooth, and the subspaces  consisting of sections meeting  $\underline{R}$ at points of order of tangency  $k$  are  contained in the finite union of codimension $2k$ submanifolds.
\end{lemma}
\begin{proof} This is an analog of \cite[Proposition~6.9]{CieliebakMohnke}, applied to each patch of the quilt. The proposition is stated for surfaces without boundary, yet the proof adapts to our framework: the subset of the universal moduli space  $\lbrace (\underline{u},\underline{J})  ~ | ~ \overline{\partial}_{\underline{J}} \underline{u} = 0   \rbrace $ consisting of pairs such that $u_i$ admits an order $k$ tangency point with $R_i$ is a Banach submanifold of  codimension $2k$. The claim then follows from Sard-Smale  theorem applied  to the projection $(u,J)\mapsto J$ defined on this space.

\end{proof}

\begin{proof}[Proof of Lemma~\ref{nobubbling}]

According to  Lemma~\ref{codimtangence}, for generic almost-complex  structures, every curve in the 0 and 1-dimensional moduli spaces intersect the hypersurfaces transversely.

A compactness theorem  analog to the one concerning non-quilted pseudo-holomorphic curves is still valid, see \cite[Theorem 3.9]{WWquilts}. Let $\underline{s}_\infty$ be a limit of quilted sections: it consists a priori of a nodal quilted map \[\underline{s}_\infty = \underline{u}_\infty \cup \bigcup_{k}{\underline{b}_k} \cup \bigcup_{l}{d_l}, \] with a principal component $\underline{u}_\infty$ (which might be broken),  some bubbles $(\underline{b}_k)_k$, possibly quilted if attached along a seam, attached to it, and discs $(d_l)_l$ attached to the boundaries, as depicted in  figure~\ref{bubbling}).
\begin{figure}[!h]
    \centering
    \def\svgwidth{\textwidth}
    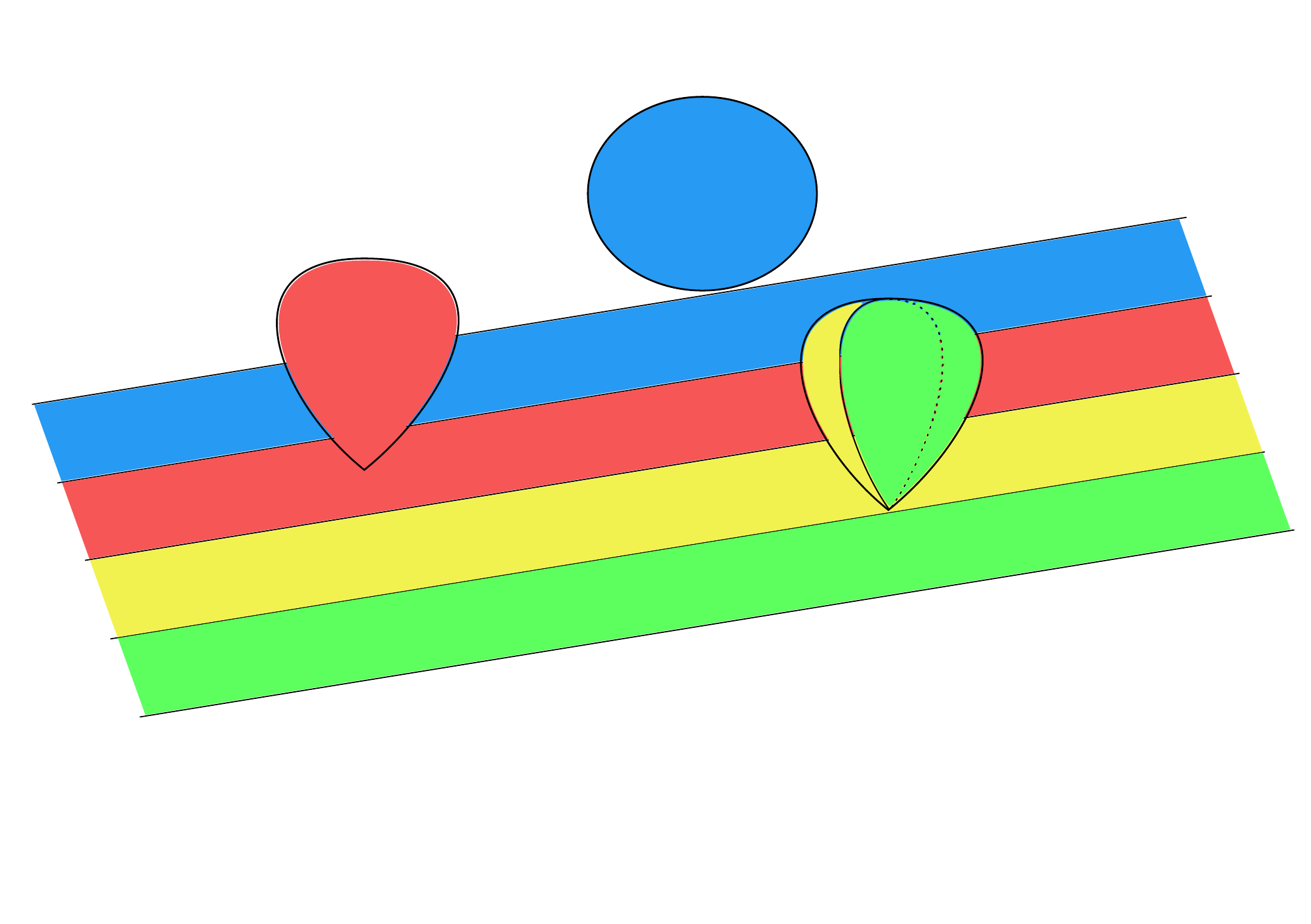
      \caption{A section with  bubbling.}
      \label{bubbling}
\end{figure}

Each disc and  bubble, which comes from a zoom in the neighborhood of a point in the base, is necessarily  contained in  a fiber of $\underline{E}$. Hence, according to \cite[Lemma 2.9]{MW}, every  disc and  bubble has non-negative index. This index is hence zero,  otherwise it would be greater  than 4 (which  divides the minimal Maslov number), which is impossible since the initial configuration  is of index smaller than 2. The area of these discs and bubbles for the monotone forms $\tilde{\omega_i}$  are thus zero: they are consequently contained in the  hypersurfaces $\underline{R}$. Since the Lagrangian submanifolds associated to a boundary are disjoint from the hypersurfaces, the nodal map contains no discs $d_l$.

Hence, the only kind of possible bubbling would be spheres,  quilted or not, contained in the hypersurfaces $\underline{R}$. Each sphere would have an intersection number with  the hypersurfaces smaller than -2, by definition of the category  $\Symp$.  On the other hand the total intersection number  $\underline{s}_\infty . \underline{R}$ is zero, but we have \[\underline{s}_\infty . \underline{R} =  \underline{u}_\infty . \underline{R} +\sum_{k}{\underline{b}_k . \underline{R}} = 0,\]
which implies that $\underline{u}_\infty . \underline{R}$ is greater than twice the number of bubbles. It follows that $\underline{u}_\infty$ intersects $\underline{R}$ transversely at  points to which no bubbles are  attached, which is impossible for  a limit of curves that do not intersect $\underline{R}$.
\end{proof}

One can now show that the map $C\Phi_{\underline{E},\underline{F}}$ commutes with  the differentials of the  complexes, and induces a morphism $\Phi_{\underline{E},\underline{F}} $ at the level of homology  groups, which are independent on the regular almost complex structures $\underline{J}$, and invariant under Hamiltonian isotopies. The proof of these two facts is a standard argument, similar to the one given in section~\ref{sec:homotopy}, and consisting in joining two almost complex structures by a path and  considering a one-dimensional parametrized moduli space, which can be compactified in a  manifold with boundary, and provides a homotopy between the two corresponding chain maps.

\paragraph{\bf Lefschetz fibration associated to a generalized Dehn twist}
A Lefschetz fibration is endowed with its canonical symplectic connection  \cite[Formula (2.1.5.)]{Seidel} on the complement of the critical set, 
\[T^hE = (\mathrm{Ker}~D_e\pi)^{\omega}.\]
One can then define the monodromy along a path of the base avoiding the critical values.

As noticed by Arnold in \cite{Arnold}, the monodromy of a  Lefschetz fibration around a critical value  is a generalized Dehn twist. Conversely, if $\tau _S$ is a model Dehn twist  along a Lagrangian sphere  $S\subset M$ (disjoint from the hypersurface $R$), there exists  a Lefschetz fibration $E_S$, called the standard fibration  associated to $\tau _S$, over the disc, with  a single  critical point over $0$, whose fiber over $1$ is $M$, and  monodromy around $0$ corresponds to this twist, see for example  \cite[Lemma 1.10, Prop. 1.11]{Seidel}. If $M$ is monotone, $E_S$ is also monotone as long as  $S$ has dimension greater than 2, according to \cite[Prop. 4.9]{WWtriangle}. We refer  to  \cite[Lemma 1.10]{Seidel} for the construction of this fibration.

Recall  the two following definitions, taken from \cite{Seidel}:

\begin{defi}An almost complex structure $J$ on $E$ is called horizontal if it preserves the decomposition $TE = T^vE \oplus T^hE$ on the complement of the critical set.
\end{defi}

\begin{defi}A quilted Lefschetz fibration is said to have positive curvature if for all horizontal tangent vector $v$, $\omega(v, J v) \geq 0$. 
\end{defi}

These ensures the following proposition:

\begin{prop}\label{areaposit}Let $(\underline{E},\underline{F})$ be a quilted Lefschetz fibration with positive curvature, $\underline{J}$ a family of horizontal almost complex structures, and $\underline{u}$ a $\underline{J}$-holomorphic section. Then $\underline{u}$ has positive area: $\sum_i\int  u_i^* \omega_i \geq 0$.
\end{prop}

\begin{proof}
Let $v+h \in T_x E_i = T_x E_i^v \oplus T_x E_i^h$ be a tangent vector  to the total space. $\omega_i (v+h, J_i (v+h)) \geq 0$, Indeed it is the sum of the 4 following terms:

$\omega_i (v, J_i v) \geq 0$, since $\omega_i$ is symplectic in restriction to the fibers, and $J_i$ is compatible with $\omega_i$.

$\omega_i (h, J_i h) \geq 0$, since the fibration has positive curvature.

$\omega_i (v, J_i h)  = \omega_i (h, J_i v) = 0$, since $J_i$ is horizontal, and by definition  $T_xE_i^h$ is the orthogonal of $T_xE_i^v$ for $\omega_i$.

It follows that the bilinear form  $\omega_i(., J_i .)$ is positive, hence the claim.
\end{proof}

\begin{remark}The standard Lefschetz fibrations  $E_S$ associated to model Dehn twists  have positive  curvature, according to \cite[Lemma 1.12, (iii)]{Seidel}.
\end{remark}

\paragraph{\bf Composition  of relative invariants} 
 Let  $\underline{\pi}_1 \colon \underline{E}_1 \rightarrow \underline{S}_1$ and $\underline{\pi}_2 \colon \underline{E}_2 \rightarrow \underline{S}_2$ be quilted Lefschetz fibrations   as in Definition~\ref{quiltfiblefschetz}, with boundary and seam conditions  respectively $\underline{F}_1$ and $\underline{F}_2$. Suppose there exists a bijection between the incoming ends $\mathcal{E}_{2,-}$ of $\underline{S}_2$ and the outgoing ends $\mathcal{E}_{1,+}$ of $\underline{S}_1$ such that $\underline{\pi}_1$ and $\underline{\pi}_2$ coincide on each end, i.e.  the number  of patches, the symplectic manifolds and the correspondences associated to the seams correspond.

Let $\rho >0$, denote by $\underline{S}_1 \cup_\rho \underline{S}_2$ the quilted surface  obtained by  gluing the patches \[ [0, \rho ] \times [k-1,k] \subset [0, + \infty ) \times [k-1,k]\] and \[ [- \rho,0]  \times [k-1,k] \subset  (- \infty,0] \times [k-1,k],\] and $\underline{E}_1 \cup_\rho \underline{E}_2$ the glued quilted fibration. 

The following proposition is the analog of \cite[Theorem 4.18]{WWtriangle}, its proof is identical. 
\begin{prop}\label{compoinvrelat} For $\rho$ sufficiently large, there exists a comeagre subset of product almost complex structures  for which the  spaces of index 0 and 1   pseudo-holomorphic sections  are smooth  and may be identified with the fibered  products:
 \begin{align*} \mathcal{M}(\underline{E}_1 \cup_\rho \underline{E}_2, \underline{F}_1 \cup_\rho \underline{F}_2)_0  & \simeq \mathcal{M}(\underline{E}_1, \underline{F}_1)_0 \times_{ev_1, ev_2} \mathcal{M}(\underline{E}_2, \underline{F}_2)_0, \\
\mathcal{M}(\underline{E}_1 \cup_\rho \underline{E}_2, \underline{F}_1 \cup_\rho \underline{F}_2)_1 & \simeq \mathcal{M}(\underline{E}_1, \underline{F}_1)_0 \times_{ev_1, ev_2} \mathcal{M}(\underline{E}_2, \underline{F}_2)_1 \\ &\cup \mathcal{M}(\underline{E}_1, \underline{F}_1)_1 \times_{ev_1, ev_2} \mathcal{M}(\underline{E}_2, \underline{F}_2)_0 ,
 \end{align*}
where $ev_i\colon \mathcal{M}(\underline{E}_i, \underline{F}_i)\rightarrow \I (\mathcal{E}_{1,+})$ is the map sending a section to its  limits at the incoming (resp. outgoing) ends for $\underline{E}_2$ (resp.  $\underline{E}_1$).  

It follows that $C\Phi_{\underline{E}_1 } \circ C\Phi_{\underline{E}_2 }= C\Phi_{\underline{E}_1 \cup_\rho \underline{E}_2}$

\end{prop}
\arnaque 

\subsubsection{Construction of the maps}\label{sec:mapstriangle}

in order to  construct the two chain complex morphisms
\begin{align*}
C\Phi_1 & \colon CF (\tau_S L_0,S^T, S, \underline{L};\Lambda) \rightarrow CF (\tau_S L_0, \underline{L};\Lambda) \\
C\Phi_2 & \colon CF (\tau_S L_0, \underline{L};\Lambda) \rightarrow CF (L_0, \underline{L};\Lambda) 
\end{align*}
which will  approximate the maps of the short exact sequence~\ref{suitecourte} induced by the inclusions of the intersection points  from Proposition~\ref{inter}, we apply the previous construction to the  two quilted Lefschetz fibrations  described here.

\paragraph{\bf Definition of $C\Phi_1$:} 
The map $C\Phi_1$ is defined as being the relative invariant  associated to the quilted Lefschetz fibration $(\underline{E}_1, \underline{F}_1) \rightarrow \underline{S}_1$ described in figure~\ref{quiltedpants}: the quilted surface  $\underline{S}_1$ consists of $k$  parallel strips  $[0,1]\times \rr$ seamed altogether, and a pair of pants seamed to the others along one of its boundaries, see figure~\ref{quiltedpants}. The fibration $\underline{E}_1$ is trivial on each patch, its various fibers  $M_0$, ... ,  $M_k$ are specified in the figure. The Lagrangian conditions $\underline{F}_1$ are  constant in these  trivializations and correspond to $\underline{L}$ on the  parallel strips, $S$ on the boundary component joining the  two incoming ends, and $\tau_S L_0$ on the last  boundary of the pair of pants.

\begin{figure}[!h]
    \centering
    \def\svgwidth{\textwidth}
    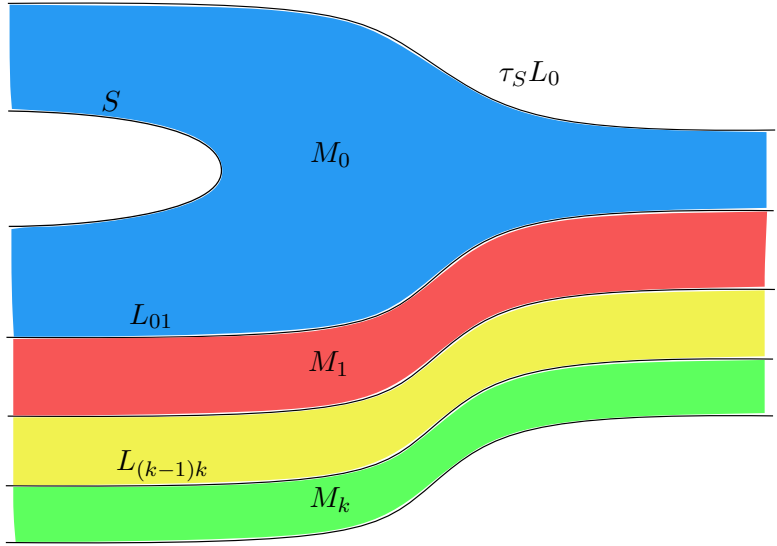
      \caption{Quilted surface  defining $C\Phi_1$.}
      \label{quiltedpants}
\end{figure}

Denote by \[\Phi_1\colon HF (\tau_S L_0,S^T, S, \underline{L};\Lambda) \rightarrow HF (\tau_S L_0, \underline{L};\Lambda) \] the map induced by   $C\Phi_1$ in homology.

\paragraph{\bf Definition of $C\Phi_2$:}
The map $C\Phi_2$ is defined as being the relative invariant  associated to the quilted Lefschetz fibration $(\underline{E}_2, \underline{F}_2) \rightarrow \underline{S}_2$ described in figure~\ref{quiltedlefschetz}: $\underline{S}_2$ consists of $k + 1$ parallel strips  seamed altogether. The restriction of $\underline{E}_2$ over the first strip corresponds to $E_S$, the standard fibration  associated to $S$, and is trivial over the other strips, with fibers specified in the figure. The Lagrangian conditions  $\underline{F}_2$ are constant in the trivializations of all but the first patches, and correspond to $\underline{L}$. On the patch corresponding to $M_0$, as in \cite{Seidel}, we have drawn a dashed line connecting the critical value with a boundary point, and we trivialize the  fibration on the complement of this path. In this trivialization, the Lagrangian conditions on the two sides of the path differ from the monodromy of this fibration, namely the twist $\tau_S$.

\begin{figure}[!h]
    \centering
    \def\svgwidth{\textwidth}
    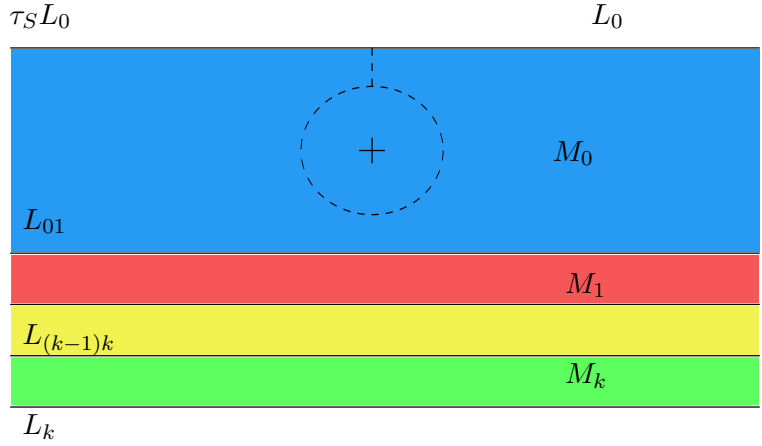
      \caption{Quilted Lefschetz fibration  defining $C\Phi_2$.}
      \label{quiltedlefschetz} 
\end{figure}

We denote by $\Phi_2\colon HF (\tau_S L_0, \underline{L};\Lambda) \rightarrow HF (L_0,\underline{L};\Lambda) $ the induced map in homology.

\subsubsection{The composition is homotopic to zero}\label{sec:homotopy}

According to Proposition~\ref{compoinvrelat}, the composition $C\Phi_{2} \circ C\Phi_{1}$ corresponds to the relative invariant  associated to the glued fibration $\underline{S}_1 \cup_\rho \underline{S}_2$, for a sufficiently large gluing parameter $\rho$. By deforming the base surface, we will show that $C\Phi_{2} \circ C\Phi_{1}$ is homotopic to the composition $C\Phi_{4} \circ C\Phi_{3}$ of two relative invariants, then we will see that  the morphism $C\Phi_{3}$ is homotopic to 0.

\begin{figure}[!h]
    \centering
    \def\svgwidth{\textwidth}
    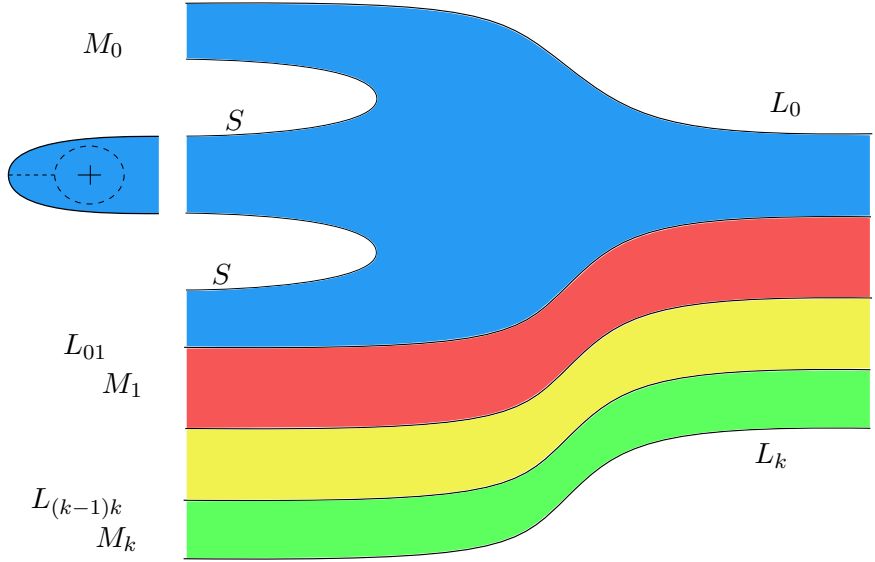
      \caption{Quilted fibration  defining $C\Phi_4$ and $C\Phi_3$.}
      \label{PHI_4_PHI_3}
\end{figure}

 Let  $\underline{E}_3 \rightarrow \underline{S}_3$ and $\underline{E}_4 \rightarrow \underline{S}_4$ be as in figure~\ref{PHI_4_PHI_3}, and $\rho'$ a sufficiently large gluing parameter so that, by Proposition~\ref{compoinvrelat}, $C\Phi_{4} \circ C\Phi_{3} = C\Phi_{\underline{E}_3 \cup_{\rho'} \underline{E}_4 } $. 

The fibrations $ \underline{E}_1 \cup_{\rho} \underline{E}_2$ and $ \underline{E}_3 \cup_{\rho'} \underline{E}_4$ are diffeomorphic. As a smooth manifold, denote $\underline{E}$ their common total space, $\underline{S}$ their common base and $\underline{\pi}$ the common projection.

We  describe  a one parameter family  of almost complex structures on this  fibration, $(\underline{E}_t, \underline{S}_t)_{t\in [0,1]}$, that interpolates from $ \underline{E}_1 \cup_{\rho} \underline{E}_2$ to $ \underline{E}_3 \cup_{\rho'} \underline{E}_4$:

Let $(\underline{j}_t)_{t\in [0,1]}$ be a one parameter family of complex structures on $ \underline{S}_1 \cup_\rho \underline{S}_2 \simeq \underline{S}_3 \cup_{\rho'} \underline{S}_4 $ such that $\underline{j}_0$ corresponds to the complex structure of  $ \underline{S}_1 \cup_\rho \underline{S}_2$ and $\underline{j}_1$ corresponds to the one from $\underline{S}_3 \cup_{\rho'} \underline{S}_4$. 

Let $(\underline{J}_t)_{0\leq t \leq 1}$ be a one parameter family of complex structures on the total space  $\underline{E}_t$ such that $\underline{J}_0$ corresponds to the almost complex structure of $ \underline{E}_1 \cup_{\rho} \underline{E}_2$, $\underline{J}_1$ corresponds to the almost complex structure of $ \underline{E}_3 \cup_{\rho'} \underline{E}_4$, and such that for all $t$, the projection $\pi$ is $(\underline{J}_t,\underline{j}_t)$-holomorphic.

The standard following reasoning, see for example \cite[Th.~3.1.6]{McDuSal}, enables one to prove that such a generic deformation induces a  homotopy between the maps $C\Phi_{2} \circ C\Phi_{1}$ and $C\Phi_{4} \circ C\Phi_{3}$. One considers the following parametrized moduli space: for $k = -1$ or $0$, let $\mathcal{M}^k_{param} = \bigcup_{t}{ \lbrace t\rbrace \times \mathcal{M}_t^k  }$, where $\mathcal{M}_t^k $ stands for the union over all $\underline{x}\in \I(L_0,S,S^T, \underline{L})$, $ \underline{y}\in \I(L_0,\underline{L})$, of the moduli spaces  $\mathcal{M}_t(\underline{x},\underline{y})_k $ of pseudo-holomorphic sections  of $\underline{E}_t$ of index $k$,  having for  limits $\underline{x}$ and $\underline{y}$ at the ends. This space corresponds to the vanishing locus of a  section of a Banach bundle, whose linearization near a solution is a Fredholm operator: the linearized parametrized Cauchy-Riemann  operator, see \cite[Def. 3.1.6]{McDuSal}). For a generic choice of families $\underline{j}_t$ and $\underline{J}_t$, it is surjective. In these conditions,  $\mathcal{M}^k_{param}$ is a manifold with boundary of dimension $k+1$.

Hence $\mathcal{M}^{-1}_{param}$ has dimension zero and provides a map \[h\colon CF(L_0,S,S^T, \underline{L};\Lambda) \rightarrow CF(L_0,\underline{L};\Lambda)\] 
defined by:
\[ h ( \underline{x} ) = \sum_{\underline{y}}{\sum_{\underline{u}\in \mathcal{M}^{-1}_{param}(\underline{x},\underline{y})}{o(\underline{u}) q^{A(\underline{u})}} \underline{y}}, \]
and $\mathcal{M}^{0}_{param}$ has dimension 1, and furnishes a cobordism which can be used for proving that $h$ is a homotopy. Indeed it can be compactified in a compact manifold with boundary, whose  boundary may be identified with the disjoint union:
\[\overline{\mathcal{M}}_0 \sqcup \mathcal{M}_1 \bigsqcup_{\underline{x}'}{ \widetilde{\mathcal{M}}(\underline{x},\underline{x}') \times \mathcal{M}_{par}^0(\underline{x}',\underline{y}) } \bigsqcup_{\underline{y}'}{\mathcal{M}_{par}^0(\underline{x},\underline{y}') \times \widetilde{\mathcal{M}}(\underline{y}',\underline{y})}, \]
where $\widetilde{\mathcal{M}}$ stands for the moduli space of pseudoholomorphic strips  involved in the differentials,   and $\underline{x}'$, $\underline{y}'$ run over the generating   sets of the source and target Floer complexes. 

\begin{proof}[Proof of the compactification of $\mathcal{M}^{0}_{param}$]No bubbling can occur on the Lagrangians and on the hypersurfaces, for the same reasons as in the proof of  Lemma~\ref{nobubbling}.
\end{proof}

It follows that \[ -C\Phi_{\underline{E}_1 \cup_{\rho} \underline{E}_2}+ C\Phi_{\underline{E}_3 \cup_{\rho'} \underline{E}_4} +  \partial h + h \partial = 0,\] which proves that $C\Phi_{2} \circ C\Phi_{1}$ and $C\Phi_{4} \circ C\Phi_{3}$ are homotopic.

It remains to show that $C\Phi_{3}$ is homotopic to $0$. This follows from \cite[Cor. 4.23]{WWtriangle}: on the one hand, for $r>0$ sufficiently small, the standard fibration  over the disc of radius $r$ doesn't admit index zero pseudo-holomorphic sections, since there   exists a family  of sections of index $c-1$, with  $c$ the dimension of the sphere $S$, which  is strictly greater than 2, and whose area  tends to $0$ when $r\rightarrow 0$. By monotonicity, every other section of smaller index has  negative area and cannot be  pseudo-holomorphic, since the fibration has positive curvature. Hence the sections over the disc of some fixed radius are cobordant to the empty set,  a cobordism being given by a parametrized  moduli space   $\bigcup_{r\in [r_0, 1]}{ \lbrace r\rbrace \times \mathcal{M}_r  }$ union of moduli spaces  corresponding to index zero sections  of the standard fibration over the disc of radius $r$, and $r_0$ sufficiently small in order   to have $\mathcal{M}_{r_0} = \emptyset$. 

\subsubsection{Small energy contributions}\label{sec:smallcontrib}
The aim of this section (Proposition~\ref{basseenergy}) is to describe the low degree (in $q$) part of the maps $C\Phi_1$ and $C\Phi_2$ when the Dehn twist is ``sufficiently thin''. An analogous statement in  Wehrheim and Woodward's framework is \cite[Theorem 5.5]{WWtriangle}. In our case, we prove Proposition~\ref{basseenergy} by adapting the original proof of Seidel (\cite[Parag. 3.2-3.3]{Seidel}).

\paragraph{\bf Preliminaries}
While Wehrheim and Woodward's proof  involves analytic arguments such as the mean value inequality, Seidel's proof in the exact case is based on a priori area computations with action functionals $a_{L_0,L_1}$ associated to  pairs of Lagrangians $(L_0,L_1)$. For quilted Floer homology, the analog of these functionals is the quilted action functional, \cite[Parag. 5.1]{WWqfc}. Recall  its definition:

\begin{defi}[Quilted action]

Let $\widetilde{\underline{L}}\colon pt\rightarrow M_0 \rightarrow \cdots \rightarrow pt$ be a generalized Lagrangian correspondence satisfying the assumptions of  Definition~\ref{defcat}.

\begin{itemize}
\item[$(i)$] Denote \[\mathcal{P}(\widetilde{\underline{L}}) = \lbrace \underline{\alpha} = (\alpha_i \colon [0,1] \rightarrow M_i\setminus R_i)_i\ |\ (\alpha_i(1),\alpha_{i+1}(0))\in L_{i,i+1}  \rbrace .\] The intersection points  $\I(\widetilde{\underline{L}})$ are identified with the constant paths.
Notice that $\mathcal{P}(\widetilde{\underline{L}})$ is arc-connected since the Lagrangian  correspondences are connected, and the manifolds $M_i$ are simply connected.

\item[$(ii)$] The \emph{symplectic action} is the functional $a_{\widetilde{\underline{L}}}\colon \mathcal{P}(\widetilde{\underline{L}})\rightarrow \rr/M\zz$, where $M = \kappa N $\footnote{N is the minimal Maslov number, $\kappa = \frac{1}{4}$ is the monotonicity constant} is the minimal area of a  sphere with positive area, defined as follows. 

Fix a  base path $\underline{\alpha}^{bas}$ in $\mathcal{P}(\widetilde{\underline{L}})$. If $\underline{\alpha}\in \mathcal{P}(\widetilde{\underline{L}})$, pick a path $\underline{\alpha}_t$ joining $\underline{\alpha}^{bas}$ and $\underline{\alpha}$ in $\mathcal{P}(\widetilde{\underline{L}})$, which can be seen as a quilted surface  

\[ \widetilde{\underline{\alpha}} = (\widetilde{\alpha_i}\colon [0,1]\times[0,1]\rightarrow M_i\setminus R_i). \]

Define then \[a_{\widetilde{\underline{L}}}(\underline{\alpha}) = \sum_{i}{\int_{[0,1]^2}{\widetilde{\alpha_i}^{*}\tilde{\omega}_i}}, \] where $\tilde{\omega}_i$ denotes the monotone  form  of $M_i$.

\end{itemize}
\end{defi}

From the  monotonicity of $M_i\setminus R_i$ and  simply-connectedness of the $L_{i,i+1}$, this quantity is well-defined modulo $M\zz$ (if one picks another path $\underline{\beta}_t$ joining $\underline{\alpha}^{bas}$ and $\underline{\alpha}$ in $\mathcal{P}(\widetilde{\underline{L}})$, then after capping each corresponding seams by discs in the Lagrangian correspondences, one can see that the two areas differ by a sum of areas of spheres inside each symplectic manifold). The action functional is then well-defined, up to a constant, depending on the   choice of the base path. 

Hence, if $\underline{u}$ is a  quilted strip joining $\underline{x},\underline{y}\in \I(\widetilde{\underline{L}})$, its  symplectic area modulo $M$ is given by the difference:
\[ A(\underline{u}) =  a_{\widetilde{\underline{L}}}(\underline{y}) - a_{\widetilde{\underline{L}}}(\underline{x}).\]

In our framework, define $a_{L_0,\underline{L}}$, $a_{S,\underline{L}}$ and $a_{L_0,S}$ so that, if $\tilde{x}_0\in L_0 \cap S$, $\underline{x}\in \I(S,\underline{L})$ and $\underline{y}\in \I(L_0, \underline{L})$, the quantity \[ \chi (\tilde{x}_0, \underline{x}, \underline{y}) = a_{L_0,\underline{L}}(\underline{y}) -  a_{L_0,S}(\tilde{x}_0) - a_{S,\underline{L}}(\underline{x}) \] coincide modulo $M$ with  the area of a quilted triangle whose seam conditions are specified in  figure~\ref{trianglechi}. This holds true in the following case: choose  base paths  for $a_{L_0,S}$ and $a_{S,\underline{L}}$ with the endpoint of the first coinciding with the beginning of the second, then take the concatenation as a base path for $a_{L_0,\underline{L}}$.

\begin{figure}[!h]
    \centering
    \def\svgwidth{\textwidth}
    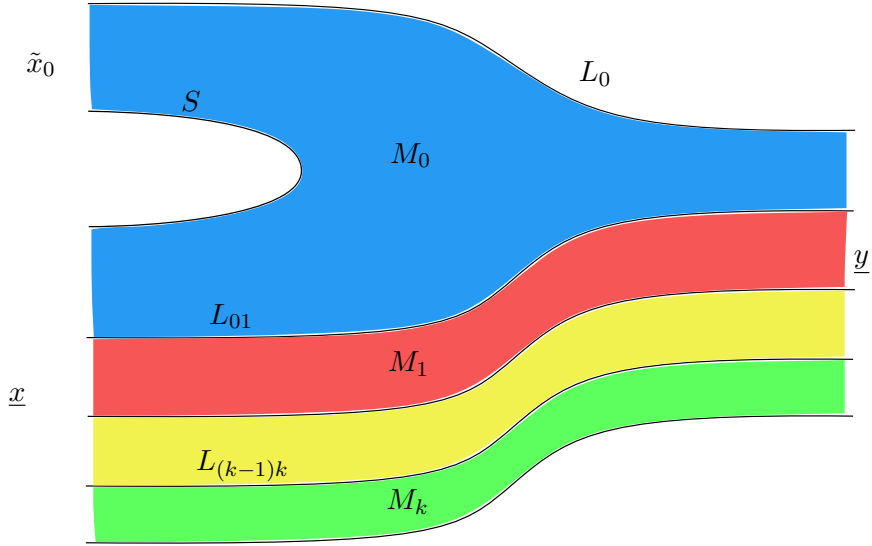
      \caption{Triangle of area $ \chi (\tilde{x}_0, \underline{x}, \underline{y})$ modulo $M$.}
      \label{trianglechi}
\end{figure}

Define now  $a_{\tau_S L_0,S}$ and $a_{\tau_S L_0,\underline{L}}$ so that they coincide with  $a_{ L_0,S}$ and $a_{ L_0,\underline{L}}$ for paths whose $M_0$  component  is outside $\iota(T(\lambda))$. In this way, if $\tilde{x}_0\in \tau_S L_0 \cap S$, $\underline{x}\in \I(S,\underline{L})$ and $\underline{y}\in \I(\tau_S L_0,\underline{L})$, the quantity

\[ \chi_{\tau_S} (\tilde{x}_0, \underline{x}, \underline{y}) = a_{\tau_S L_0,\underline{L}}(\underline{y}) -  a_{\tau_S L_0,S}(\tilde{x}_0) - a_{S,\underline{L}}(\underline{x}) \] 
represents the area of a quilted triangle  as in figure~\ref{quiltedpants} defining the map $C\Phi_1$. We want to  express this quantity from the function $R$ and the data before the twist.

\begin{prop}\label{areatriangles} Assume:
\begin{itemize}
\item[$(i)$] That the  hypotheses of Proposition~\ref{inter} are satisfied, in order to have: 

\[ \I(\tau_S L_0,\underline{L}) = i_2(I( L_0,\underline{L})) \cup i_1( (L_0 \cap S)\times \I(S,\underline{L})).\]

\item[$(ii)$] That $L_{01}$ is a product in a neighborhood of each intersection point of $\I(S,\underline{L})$, namely: 

\[\forall \underline{x}\in \I(S,\underline{L}), \exists U_0, U_1 \colon U_0\times U_1 \cap L_{01} = T(\lambda)_{x_0} \times L_1(\underline{x}),\]
with $\underline{x} = (x_0, x_1, \cdots )$, $U_0$ (resp. $U_1$)  a neighborhood of $x_0$ in $M_0$ (resp. of $x_1$ in $M_1$), and $L_1(\underline{x})\subset U_1$ a Lagrangian (depending on $\underline{x}$).
\end{itemize}
Then, 
\begin{enumerate}

\item If $\tilde{x}_0\in \tau_S L_0 \cap S$, $\underline{x}\in \I(S,\underline{L})$ and $y_0$ denotes  the $M_0$ coordinate of $i_1(\tilde{x}_0, \underline{x})$, 

\[\chi_{\tau_S} (\tilde{x}_0, \underline{x}, i_2(\tilde{x}_0, \underline{x}) ) = K_{\tau_S}(y_0) - 2\pi R(0) \ \ (\text{mod }M),\]
with $K_{\tau_S}(y_0) = 2\pi ( R'(\mu(y_0)) \mu(y_0) - R(\mu(y_0)))$ the function associated to the  twist as in \cite{Seidel}, and as in sections~\ref{sec:generalizedtwists}, $\mu$ and $R$ refer respectively to the norm of a  covector and the function used for  defining the twist (primitive of the angle function). 

Moreover, $K_{\tau_S}(y_0) - 2\pi R(0)$ is exactly the area of an index zero triangle.

\item If $\tilde{x}_0\in \tau_S L_0 \cap S$, $\underline{x}\in \I(S,\underline{L})$ and $\underline{y}\in \I(L_0,\underline{L}) = i_2(\I(L_0,\underline{L}))$,

\[ \chi_{\tau_S} (\tilde{x}_0, \underline{x}, \underline{y}) = \chi (\mathbb{A}(\tilde{x}_0), \underline{x}, \underline{y}) - 2\pi R(0) \ \ (\text{mod }M).\]

\item If $\tilde{x}_0\in \tau_S L_0 \cap S$, $\underline{x}\in \I(S,\underline{L})$ and $\underline{y} = i_2(\tilde{z}_0, \underline{z})\in  i_2(\I(L_0,\underline{L}))$,

\begin{align*}
\chi_{\tau_S} (\tilde{x}_0, \underline{x}, \underline{y}) &= \chi_{\tau_S} (\tilde{z}_0, \underline{z}, i_2(\tilde{z}_0, \underline{z})) + a_{L_0,S}(\mathbb{A}(\tilde{z}_0)) + a_{S,\underline{L}}(\underline{z}) \\ 
 &- ( a_{L_0,S}(\mathbb{A}(\tilde{x}_0)) + a_{S,\underline{L}}(\underline{x}) ) \ \ (\text{mod }M).
\end{align*}

\end{enumerate}

\end{prop}

\begin{proof} As $L_{01}$ is a product in the neighborhood of the points of $\I(S,\underline{L})$, the  part in $M_1, M_2, \cdots, M_k$ of the quilted  triangles  involved in the computation of $\chi_{\tau_S}$ is the same as the one appearing in $\chi$: only the $M_0$ part can change its area, and the computation reduces to Seidel's  one, see formula (3.7) in the proof of \cite[Lemma 3.2]{Seidel}.
\end{proof}

\begin{remark}These formulas are illustrated in figure~\ref{triangle}:  
 $\chi_{\tau_S} (\tilde{x}_0, \underline{x}, \underline{y})$ represents the  area of a quilted triangle  as in figure~\ref{quiltedpants}. It is the sum  of the area of a polygon  independent on the twist (the empty polygon for the purple triangle, a triangle for the green triangle, and a rectangle for the yellow triangle) and a small quantity which depends on the  primitive $R$ of the angle function.
\end{remark}

\begin{figure}[!h]
    \centering
    \def\svgwidth{\textwidth}
    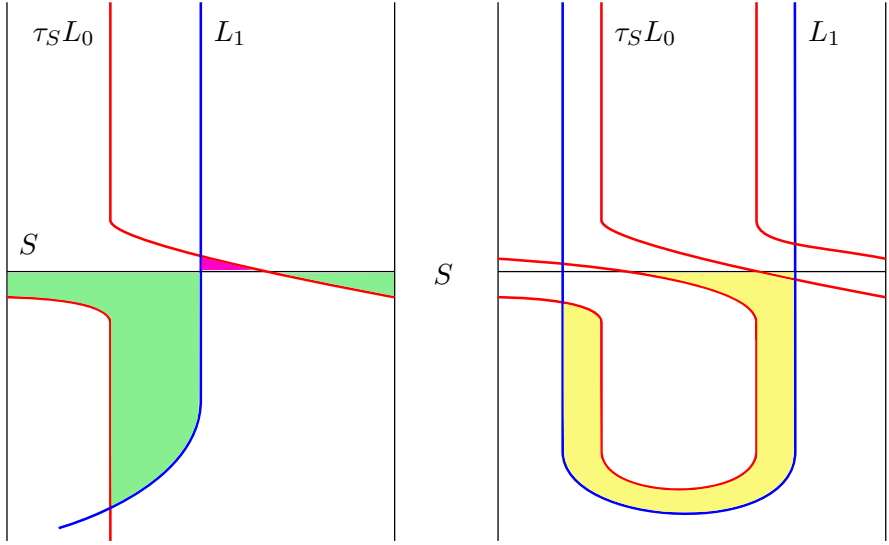
      \caption{Three triangles whose  area is given by $\chi_{\tau_S}$.}
      \label{triangle}
\end{figure}

\begin{prop}[Small energy contributions] \label{basseenergy}
Let $\epsilon >0$  be sufficiently small. Assume:
\begin{itemize}

\item[$(i)$] that the hypotheses of Proposition~\ref{areatriangles} are satisfied,

\item[$(ii)$] \begin{itemize} 

\item[$(a)$] $\forall \underline{x} \neq \underline{y} \in \I(L_0,\underline{L}),\ a_{L_0,\underline{L}}(\underline{x}) - a_{L_0,\underline{L}}(\underline{y})\notin (-3\epsilon,3\epsilon),$ 

\item[$(b)$]$\forall (\tilde{x}_0, \underline{x}) \neq (\tilde{z}_0, \underline{z}) \in (L_0 \cap S) \times \I(S,\underline{L}),$
\[a_{L_0,S}(\tilde{z}_0) + a_{S,\underline{L}}(\underline{z}) - ( a_{L_0,S}(\tilde{x}_0) + a_{S,\underline{L}}(\underline{x}) )\notin (-3\epsilon,3\epsilon),\]
 
\item[$(c)$]$\forall \tilde{x}_0 \in L_0 \cap S,  \underline{x} \in  \I(S,\underline{L}), \underline{y} \in  \I(L_0,\underline{L})$,  \[  a_{L_0,\underline{L}}(\underline{y}) - ( a_{L_0,S}(\tilde{x}_0) + a_{S,\underline{L}}(\underline{x}) )\notin (-5\epsilon,5\epsilon).\] 

\end{itemize}

\item[$(iii)$] $0 \geq 2\pi R(0) > -\epsilon,$ and $\tau_S$ is concave, in the sense of Definition~\ref{concave}.

\end{itemize}

Then, under these hypotheses,
\begin{itemize}

\item[$(a)$] $C\Phi_1 = C\Phi_{1,\leq \epsilon} + C\Phi_{1,\geq 2\epsilon} $, with:
\begin{itemize}

\item[$(i)$] $C\Phi_{1,\leq \epsilon}(\underline{x}) =\pm q^{A(\underline{x})}i_{1}(\underline{x})$, where $A(\underline{x})$ is a number satisfying $ 0 \leq A(\underline{x}) \leq \epsilon ,$

\item[$(ii)$] $C\Phi_{1,\geq 2\epsilon}$ is of order greater than $2\epsilon$.

\end{itemize}

\item[$(b)$] $C\Phi_2 = C\Phi_{2,\leq \epsilon} + C\Phi_{2,\geq 2\epsilon} $, with:
\begin{itemize}
\item[$(i)$] $C\Phi_{2,\leq \epsilon}(i_1(\underline{x})) = 0 $ and $\tilde{C\Phi_2} (i_2(\underline{x})) =  \pm \underline{x} $, 

\item[$(ii)$] $C\Phi_{2,\geq 2\epsilon} $ is of order greater than $2 \epsilon$.

\end{itemize}
\item[$(c)$] The homotopy $h$, and the three differentials are  of order greater than $2\epsilon$.
\end{itemize}

\end{prop}

\begin{remark}\label{justif}

One can always assume that the  hypotheses of  Proposition~\ref{basseenergy} are satisfied. Indeed, it suffices as a fist step to perturb $L_0$, $L_{01}$ and $S$ by Hamiltonian  isotopies  and take $\epsilon$ small enough  in order to  guarantee the inequalities, then slightly perturb  and eventually decrease $\lambda$ to ensure the assumptions of  Proposition~\ref{inter}.

\end{remark}

\begin{proof}
(a) Denote $\mathcal{M}(\tilde{x}_0, \underline{x},\underline{y})_0$ the moduli space of index zero quilted triangles as in  figure~\ref{quiltedpants}, having for limits $\tilde{x}_0$, $\underline{x}$ and $\underline{y}$ at the ends. Suppose $\underline{y} = i_1(\tilde{x}_0, \underline{x})$, and  $ \underline{u} \in \mathcal{M}(\tilde{x}_0, \underline{x},\underline{y})_0$. One has 

\[ \langle C\Phi_1 (\tilde{x}_0, \underline{x}),\underline{y} \rangle = \# \mathcal{M}(\tilde{x}_0, \underline{x},\underline{y}) q^{A(\underline{u})}.\]

On the one hand, \[A(\underline{u}) = \chi_{\tau_S} (\tilde{x}_0, \underline{x}, i_2(\tilde{x}_0, \underline{x}) ) = K_{\tau_S}(y_0) - 2\pi R(0) \in [0, \epsilon),\] by Proposition~\ref{areatriangles}. 

On the other hand, $\# \mathcal{M}(\tilde{x}_0, \underline{x},\underline{y})_0 = \pm 1$. This can be proven by a cobordism argument  similar to the one involved in the proof that $C\Phi_{2} \circ C\Phi_{1}$ and $C\Phi_{4} \circ C\Phi_{3}$ are homotopic.  We recall  briefly this argument, and refer to \cite[Prop. 3.4]{Seidel} for more details. Consider a  family $(f_t\colon S^n\rightarrow S)_{t\in [0,1]}$ of parametrizations of $S$ such that $f_0$ coincide with  the embedding $\iota\colon T(0)\rightarrow M$ and $f_1$ sends the antipode $\mathbb{A}(x_0)$ to the point  $\tilde{x}_0$. This  family allows one to  define a parametrized  moduli space $(\mathcal{M}_t)_{t\in [0,1]}$, where $\mathcal{M}_0 = \mathcal{M}(\tilde{x}_0, \underline{x},\underline{y})_0$ and $\mathcal{M}_t$ denotes the moduli space corresponding to the model Dehn twist induced by $f_t$. For a generic family $f_t$, this space is a 1-dimensional cobordism  between $\mathcal{M}_0$ and $\mathcal{M}_1$. Moreover, this cobordism is compact, according to Lemma~\ref{nobubbling} and since there is not enough energy for positive area breaking to appear, by assumption. In addition,  $\mathcal{M}_1$ consists of a single point, the constant triangle.

Suppose now $\underline{y} \neq i_1(\tilde{x}_0, \underline{x})$. By the hypotheses and Proposition~\ref{areatriangles}, 

\[ \chi_{\tau_S} (\tilde{x}_0, \underline{x},  \underline{y} ) \notin [0, \epsilon) \subset \rr/M\zz . \]

Hence, the area of every pseudo-holomorphic triangle, necessarily positive, is greater than $\epsilon$.

(b) For $C\Phi_2$, Seidel's proof applies similarly, see \cite[Section 3.3]{Seidel}. We recall it briefly: if $\underline{x}\in \I(\tau_S L_0,\underline{L})$ and  $\underline{y}\in \I( L_0,\underline{L})$, 
take an horizontal almost complex structure on $\underline{E}_2$,  which is possible by \cite[Lemma 2.9]{Seidel}). Since the fibration has positive curvature, the only zero area J-holomorphic sections are the constant sections, i.e. intersection points of $\I(L_0, S^T,S, \underline{L} )$. There is exactly one such section when $\underline{y} = i_2(\underline{x})$, and zero if $\underline{y}\neq i_2(\underline{x})$. Every other section has a strictly positive area, according to Proposition~\ref{areaposit}, since the fibration has positive curvature, and this area is given by $a_{L_0,\underline{L}}(\underline{y}) - a_{\tau_S L_0,\underline{L}}(\underline{x}) $. If one denotes $\underline{\tilde{x}}$ the point whose first coordinate is the image of the one  from $\underline{x}$ by the antipodal map, this quantity is given by, according to \cite[Formula (3.2)]{Seidel}:
\[ a_{L_0,\underline{L}}(\underline{y}) - a_{ L_0,\underline{L}}(\underline{\tilde{x}}) - 2\pi R(0) + 2\pi \int_0^{||y||}{(R'(||y||) - R'(t)    )dt},\]
where $y = \iota^{-1}(\tilde{x}_0)\in T(\lambda)$. By assumptions $(ii)$ $(a)$ and $(iii)$, this quantity is greater than $2\epsilon$.

(c) Similarly, an action computation permits to prove the claim for the order of the homotopy: let  $(\tilde{x}_0 ,\underline{x})\in \I(\tau_S L_0,S,S,\underline{L})$ and $\underline{y}\in \I(L_0,\underline{L})$, the area of a  section contributing to the coefficient $(h(\tilde{x}_0 ,\underline{x}), \underline{y})$ is:
\begin{align*}
 & a_{S,\underline{L}}(\underline{y}) + a_{\tau_S L_0,S}(\tilde{y}_0) - a_{S,\underline{L}}(\underline{x}) - a_{\tau_S L_0,S}(\tilde{x}_0)  \\ 
 & = a_{ L_0,S}({y}_0)- a_{ L_0,S}(\hat{x}_0) + a_{S,\underline{L}}(\underline{y})- a_{S,\underline{L}}(\underline{x}),  \end{align*}
denoting $\tilde{y}_0$ the antipode of the first coordinate  $y_0$ of $\underline{y}$, and $\hat{x}_0$ the antipode of $\tilde{x}_0$. This quantity is greater than $2\epsilon$ by assumption $(ii)(b)$.

Finally, the three differentials are of order $\geq 2\epsilon$ by $(ii)(b)$ for the one  from $CF(\tau_S L_0,S,S,\underline{L})$, and $(ii)(a)$ for the ones from $CF(\tau_S L_0,\underline{L})$ and $CF( L_0,\underline{L})$.
\end{proof}

\subsubsection{Proof of the triangle}\label{sec:prooftriangle}
One can now prove Theorem~\ref{quilttri},  following the same  strategy as in \cite[Parag. 5.2.3]{WWtriangle}. Suppose  now that the hypotheses of Proposition~\ref{basseenergy} are satisfied. Introduce the following notations: denote the  three chain complexes with coefficients in $\zz$ by:
\begin{align*}
 A_0 = & CF (\tau_S L_0,S^T, S, \underline{L};\zz) \\ 
 A_1 = & CF (\tau_S L_0, \underline{L};\zz) \\
 A_2 = & CF (L_0, \underline{L};\zz ),
\end{align*}
and denote  $C_i$ the complexes with coefficients in $\Lambda$, $C_i = A_i \otimes_\zz \Lambda$ as $\Lambda$-modules, and endowed with their respective differentials  $\partial_0$, $\partial_1$ and $\partial_2$. The maps $C\Phi_1$, $C\Phi_2$ and the homotopy $h$ between $C\Phi_2 \circ C\Phi_1$ and the zero map constructed in paragraph ~\ref{sec:homotopy} are specified in the following diagram:

\begin{equation*} 
\xymatrix{
 {C_0} \ar[r]^{C\Phi_1} \ar@/_1.5pc/[rr]^{ h }& {C_1} \ar[r]^{C\Phi_2} & {C_2.}
}
\end{equation*}

The mapping cone $\mathrm{Cone}~C\Phi_1$ is the chain complex $C_0\oplus C_1$ whose differential is given  by:\[ \partial_{\mathrm{Cone}~C\Phi_1} =
\left[ \begin{matrix}
 - \partial_0 & 0  \\
 - C\Phi_1 & \partial_1  
 
\end{matrix}\right] .  \] 
From the snake lemma, the short exact sequence on chain complexes induces the long exact sequence:
\[ \cdots \rightarrow H_*(C_0)\rightarrow H_*(C_1)\rightarrow H_*(\mathrm{Cone}~C\Phi_1 ) \rightarrow \cdots ,\]
where the first map is $\Phi_1$. It then suffices  to prove that the morphism of chain complexes \[(h, - C\Phi_2)\colon \mathrm{Cone}~C\Phi_1 \rightarrow C_2\]  induces an isomorphism in homology. This will be the case if and only if its mapping cone is acyclic. As a $\Lambda$-module, $\mathrm{Cone}~(h, - C\Phi_2) = C_0\oplus C_1\oplus C_2 $. In this decomposition, its differential is given by: 
\[  \partial = 
\left[ \begin{matrix}
\partial_0 & 0 & 0 \\
 C\Phi_1 & - \partial_1 & 0 \\
 - h & C\Phi_2 & \partial_2 
 
\end{matrix}\right]. \]
The following Lemma~\ref{lemmaperutz}   permits to prove  the acyclicity of a chain complex over $\Lambda$ from the leading term of its differential. Its conclusion only holds in the $q$-adic completion of $\Lambda $, i.e. the universal Novikov ring:

\[\hat{\Lambda} = \left\lbrace \sum_{k = 0}^{\infty}{ a_k q^{\lambda_k}}: a_k\in \zz,\, \lambda_k \in \rr,\, \lim_{k\to +\infty} \lambda_k = +\infty \right\rbrace\] 

Recall first some terminology for $\rr$-graded modules.

\begin{defi} An \emph{$\rr$-graded module} $A$ is a module endowed with a decomposition $A = \bigoplus_{r\in\rr }{A_r}$. Its \emph{support} is defined by $Supp A =  \lbrace r: A_r\neq 0 \rbrace$. Let $I\subset \rr$, $A$ is said to have \emph{gap $I$}  if $\forall r,r' \in Supp A, r-r' \notin I$.

If $r'\in \rr$, we denote $A[r]$ the \emph{shift} defined by $A[r]_s = A_{r+s}$. One has $Supp A[r] = Supp A -r$.

A linear map $f\colon A\to B$ between two graded modules  is said 
\begin{itemize}
\item to be \emph{of order $I$} if for all $r$, $f(A_r) \subset \bigoplus_{i\in I}{B_{r+i}}$.

\item to have \emph{gap $I$}  if for all $r$, the image $f(A_r)$ has gap $I$.
\end{itemize}
\end{defi}

\begin{lemma}(\cite[Lemma 5.3]{Perutzgysin})\label{lemmaperutz}
 Let  $\epsilon >0$, $(A,d)$ an $\rr$-graded module which has gap $[\epsilon, 2\epsilon)$, endowed with a differential $d$ of order $[0, \epsilon)$. Let $D = A\otimes_\zz \Lambda$, $\hat{D}= A\otimes_\zz \hat{\Lambda}$ be its completion, and $\partial$ a differential on $\hat{D}$ such that:
\begin{itemize}
\item[$(i)$] $\partial$ is $\hat{D}$-linear and continuous.
\item[$(ii)$]  $\partial(A) \subset A\otimes_\zz \hat{\Lambda}_+ $, where $\hat{\Lambda}_+ = \left\lbrace \sum_{k = 0}^{\infty}{ a_k q^{\lambda_k}} \in \hat{\Lambda}:  \lambda_k \in \rr_+ \right\rbrace $.
\item[$(iii)$] $\partial = \partial_{\leq \epsilon} + \partial_{\geq 2 \epsilon}$, where $\partial_{\leq \epsilon} = d\otimes \hat{\Lambda} $ is the differential induced by $d$, and $\partial_{\geq 2 \epsilon}$  of order $[2\epsilon, +\infty)$.
\item[$(iv)$] $(A,d)$ is acyclic.
\end{itemize}
Then,  $(\hat{D},\partial)$ is acyclic.
\end{lemma}

In order to apply this lemma to the double mapping cone, equip $A= A_0 \oplus A_1 \oplus A_2$ with the following grading \footnote{Here  appears a  difference with Seidel's  proof in the exact setup:  the complex cannot be graded by the action, which is only defined  modulo $M$. The conclusion is then a priori weaker: the acyclicity only holds after completion.  We will however see that the monotonicity  hypotheses allows one to obtain  acyclicity over $\zz$.}: 
\begin{itemize}
\item $A_0$ is concentrated in  degree 0.
\item $A_1$, which is isomorphic to $A_0\oplus A_2$, is graded in the following way: its first component is graded so that $C\Phi_{2,\leq \epsilon}$ has degree 0: if $(\tilde{x}_0, \underline{x})\in \I(L_0, \underline{L})$, define $\deg i_2(\tilde{x}_0, \underline{x}) = \chi_{\tau_S} (\tilde{x}_0, \underline{x}, i_2(\tilde{x}_0, \underline{x}) )$. The second component is concentrated in degree 0.
\item $A_2$ is concentrated in degree 0.
\end{itemize}

According to propositon~\ref{basseenergy}, $\supp A\subset [0, \epsilon)$. Moreover, by construction, the differential
\[ d =\begin{pmatrix}
 0 & 0 & 0 \\
 C\Phi_{1,\leq \epsilon} & 0 & 0 \\
 0 & C\Phi_{2,\leq \epsilon} & 0 \end{pmatrix},\]
preserves this  grading.

The module $\hat{D}$ is then  identified  with $C = \mathrm{Cone }(h,-C\Phi_2)$. Its differential $\partial$ defined before satisfies the assumptions $(i)$, $(ii)$ and $(iii)$ of Lemma~\ref{lemmaperutz}, by Proposition~\ref{basseenergy}. Furthermore, $d$ is acyclic. Indeed,  in the decomposition $A_1 = A_0 \oplus A_2$, according to Proposition~\ref{basseenergy},

\[ C\Phi_{1,\leq \epsilon} =  \begin{pmatrix}
  Id_{A_0} \\0
  \end{pmatrix} \text{ and } C\Phi_{2,\leq \epsilon} \begin{pmatrix}
  0 &Id_{A_1}  
  \end{pmatrix},
\]
hence $d = \begin{pmatrix}
0 & 0 & 0 & 0 \\
Id_{A_0} & 0 & 0 & 0 \\
0 & 0 & 0 & 0 \\
0 & 0 & Id_{A_1}& 0 \end{pmatrix}$, which is clearly acyclic. 

We now explain  why the monotonicity of the Lagrangian correspondences ensures the acyclicity over $\zz$, hence the exact sequence of Theorem~\ref{quilttri}. Recall  that if $x,y$ are  generators of a chain  complex $C_i$, the symplectic area  of a strip $u$ going from $x$ to $y$ is given by $ A(u) = \frac{1}{8} I(u) + c(x,y)$\footnote{$\frac{1}{8} = \frac{1}{2} \kappa $, where $\kappa = \frac{1}{4}$ is the monotonicity constant of  the forms $\tilde{\omega}$.}, where $c(x,y)$  is a quantity independent from the strip $u$. These quantites satisfy $ c(x,z) = c(x,y)+ c(y,z)$. Similarly, if now $x$ and $y$ are generators  of different complexes  $C_i$ and $C_j$, there exists similar quantities $c(x,y)$ giving the area of section  of the fibrations defining $C\Phi_1$ and $C\Phi_2$, and these quantities satisfy the same additivity relation, by additivity of the area and of the index.

Recall  also that the  differentials $\partial_0$, $\partial_1$ and $\partial_2$ count index 1  strips, $C\Phi_1$ and $C\Phi_2$   index 0 sections, and $h$ index $-1$ sections. Hence, denoting $d(x,y) = c(x,y) +i(x) -i(y)$, where $i(x)  = 0,1,2$ refers to the subcomplex in which $x$ belongs, the coefficient $(\partial x, y)$ is of the form $m(x,y) q^\frac{d(x,y)}{2}$. 

Fix $x_0$ an arbitrary generator of $C$, and let $f\colon C\to C$ be defined on the  generators by: $f(x) = q^{d(x_0,x)} x$.

An elementary computation shows then that $\partial f = q^\frac{1}{8} f \partial_\zz$, with   $(\partial_\zz x, y) = m(x,y)$. It follows that $H_*(\hat{D},\partial)$ and $H_*(\hat{D},\partial_\zz)$ are isomorphic. Hence, by the universal coefficients theorem, 
\[ H_*(\hat{D},\partial) \simeq H_*(A,\partial_\zz)\otimes_\zz \hat{\Lambda} \oplus \mathrm{Tor}_\zz ( H_*(A,\partial_\zz), \hat{\Lambda} ) [-1].\]
Yet, $H_*(\hat{D},\partial) =0$ by Lemma~\ref{lemmaperutz}, hence $H_*(A,\partial_\zz) =0$ by the classification of finite type  abelian groups, and the fact that $(\Z{n})\otimes_\zz \hat{\Lambda} \neq 0$.

\subsection{Action of a Dehn twist on a surface}\label{sec:twistmodulispace}

The aim of this section is to study the geometric nature  of the transformation induced by a  Dehn twist on a surface $\Sigma$ along a non-separating curve on the moduli spaces  $\N(\Sigma)$. We will see in Proposition~\ref{expressionflot} that this  transformation can be expressed as a Hamiltonian flow on the complement of a coisotropic submanifold, and we will show in Theorem~\ref{twistinduit} that, when $\Sigma$ is a punctured torus, this transformation almost corresponds  to a Dehn twist, except that its support isn't compact in $\N(\Sigma)$. We will see however that a generalized Dehn twist can be built out from this symplectomorphism, which will allow us to prove  Theorem~\ref{trianglechir}.

The  group of  diffeomorphisms of $\Sigma$ which are the identity on the boundary acts in a natural way on $\N (\Sigma )$ by pulling-back. In this paragraph, we show that the symplectomorphism corresponding to a Dehn twist $\tau_K$ along a curve $K \subset  \Sigma$ can be expressed as the time 1 flow of a Hamiltonian  which is smooth outside the coisotropic submanifold $C_- = \{[A]|\mathrm{Hol}_K (A) = -I\}$.

We follow the strategy of \cite[Section 3]{WWtriangle}: we cut the surface  along $K$, introduce an intermediate moduli space  $\N (\Sigma_{cut} )$ associated to the cut suface $\Sigma_{cut}$ (see figure~\ref{sigma_sigmacut}), whose reduction for a natural Hamiltonian $SU(2)$-action yields  the complement $\N(\Sigma) \setminus C_-$ (see paragraph ~\ref{sec:gluingreduction}).

The fact that the Dehn twist $\tau_K$ is isotopic to the identity in $\Sigma_{cut}$ will enable us  to express its  pull-back as a Hamiltonian flow in $\N (\Sigma_{cut} )$ which is invariant for the previous $SU(2)$-action. This flow will thus descend to a Hamiltonian flow in the symplectic  quotient $\N(\Sigma) \setminus C_-$.

\subsubsection{Fundamental groups of $\Sigma$ and $\Sigma_{cut}$}\label{sec:fundamentalgroups}

\begin{figure}[!h]
    \centering
    \def\svgwidth{\textwidth}
    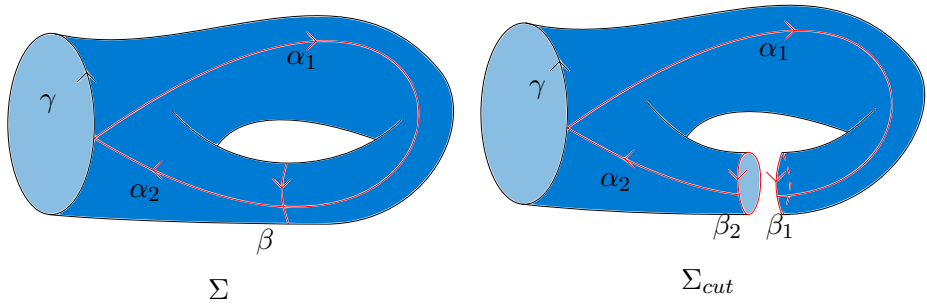
      \caption{The surfaces $\Sigma$ and $\Sigma_{cut}$.}
      \label{sigma_sigmacut}
\end{figure}

Let $p\colon \rr/\zz \rightarrow \partial \Sigma$ be the parame\-trization of the boundary, and $* = p(0)$ the base point. The curve $K$ being non-separating, there  exists a simple curve $\alpha\colon \rr/\zz \rightarrow  \Sigma$ based in $*$ and intersecting transversely $K$ in one point $\alpha (\frac{1}{2} )$. We denote $\alpha_1 = \alpha_{[0,\frac{1}{2}]}$ and $\alpha_2 = \alpha_{[\frac{1}{2}, 1]}$. We also denote $\beta\colon \rr/\zz \rightarrow  \Sigma$ a parametrization of $K$ based  in $\alpha(\frac{1}{2} )$ and oriented so that $\alpha . \beta = +1$, and $\tilde{\beta} = \alpha_2^{-1} \beta \alpha_2 $. The surface $\Sigma \setminus (\alpha \cup \beta)$ has genus $h -1$, Let   $u_2 , v_2 , \cdots , u_h , v_h$ be  generators of its fundamental group such that, denoting $\gamma = [p]$ the boundary curve  in $\Sigma$, one has  $\gamma  = [\alpha , \tilde{\beta} ]\prod_{i=2}^{h}{[u_i , v_i ]}$. The curves  $\alpha , \tilde{\beta}, u_2 , v_2 , \cdots , u_h , v_h$ then form a generating  system of $\pi_1(\Sigma, *)$, and the extended moduli space admits the description:

\[  \N(\Sigma,p) = \left\lbrace  (g, A, \tilde{B}, U_2 , V_2 , \cdots , U_h , V_h )|e^g = [A, \tilde{B}]\prod_{i=2}^{h}{[U_i , V_i ]} \right\rbrace
,\]
where $g\in \mathfrak{su(2)}$ has norm $< \pi \sqrt{2}$, and is the element such that the connection is of the form $g ds$ near the boundary. The elements of $SU(2)$ $A$, $\tilde{B}$, $U_2$, $V_2$, $\cdots$, $U_h$, $V_h$ stand for the  holonomies along the generating  curves $\alpha , \tilde{\beta}, u_2 , v_2 , \cdots , u_h , v_h$.

Let $\Sigma_{cut}$ be the compact surface obtained by cutting $\Sigma$ along $K$. Denote $\beta_1$ and $\beta_2$  parametrizations of the new boundary components, agreeing with  $\beta$ in $\Sigma$, and with $\beta_1$ touching $\alpha_1$ and $\beta_2$ touching $\alpha_2$, see figure~\ref{sigma_sigmacut}. One then associates to $\Sigma_{cut}$ the following moduli space, defined in \cite[Parag. 5.2]{jeffrey} by:
 \[\mathscr{M}^{\mathfrak{g},3}(\Sigma_{cut}) = \mathscr{A}_F^\mathfrak{g}(\Sigma_{cut}) /\Gc (\Sigma_{cut} ),\] 
where the exponent $3$ refers to the number of boundary components  of $\Sigma_{cut}$, $\mathscr{A}_F^\mathfrak{g}(\Sigma_{cut})$ is the space of flat connections on $SU(2)\times\Sigma_{cut}$ of the form $g ds$, $b_1 ds$ and $b_2 ds$ in the  neighborhoods of $\gamma$,   $\beta_1$, $\beta_2$, and $s\in \rr/\zz$ represents the parameter of the boundary. The group  $\Gc (\Sigma_{cut} )$  of gauge  transformations which are trivial in a neighborhood of the boundary acts in a natural way on $\mathscr{A}_F^\mathfrak{g}(\Sigma_{cut})$. 

We will restrict to the open subset $\N(\Sigma_{cut}) \subset  \mathscr{M}^{\mathfrak{g},3}(\Sigma_{cut})$ of connections for which the vectors $g$, $b_1$ and $b_2$ are in the ball of radius $\pi \sqrt{2}$ and center 0. 

This space admits the following description (see \cite[Prop. 5.3]{jeffrey}):

\begin{multline*}\N(\Sigma_{cut} ) \simeq \\ \left\lbrace (g, A_1 , A_2 , b_1 , b_2 , U_2 , V_2 , \cdots )~\left| ~e^g = A_1 e^{b_1} A_1^{-1} A_2^{-1} e^{b_2} A_2 \prod_{i=2}^{h}{[U_i , V_i ]}\right\rbrace \right. , \end{multline*}
where $g, b_1, b_2 \in \mathfrak{su(2)}$ are the values of the connection along the boundaries (elements of the ball of radius $\pi \sqrt{2}$), and $A_1 , A_2 ,  U_2 $ ,$ V_2 , \cdots , U_h , V_h \in SU(2)$ the holonomies along the curves $\alpha_1, \alpha_2, u_2 , v_2 , \cdots $, $u_h , v_h$.

This space is endowed with a symplectic form defined as the one for $\N(\Sigma)$:  if $[A]\in\N(\Sigma_{cut}) $ denotes the orbit of a flat connection $A$, and $\eta$, $\xi$ are $\mathfrak{su(2)}$-valued 1-forms  representing tangent vectors of $T_{[A]} \N(\Sigma_{cut})$, that is, proportional to $ds$ on each boundary, and $d_A$-closed, then

\[ \omega_{[A]}([\eta],[\xi]) = \int_{\Sigma'} \langle \eta \wedge\xi \rangle.  \]

\subsubsection{From $\N(\Sigma_{cut} )$ to $\N(\Sigma )$}\label{sec:gluingreduction}

In order to relate $\N(\Sigma_{cut} )$ and $\N(\Sigma )$, we prove and use the fact that these moduli spaces satisfy the creed ``gluing equals reduction''.

The group $SU(2)^3$ acts in a Hamiltonian way on $\N(\Sigma_{cut} )$, indeed $SU(2)^3$ may be identified with the quotient $\G ^{const} (\Sigma_{cut} ) / \Gc (\Sigma_{cut} )$, where $\G ^{const} (\Sigma_{cut} )$ is the group of gauge transformations  constant near the  boundary. The moment of this action is given by:

\[ \Psi = (\Phi_\gamma,\Phi_1,\Phi_2) \colon \N(\Sigma_{cut} ) \rightarrow \mathfrak{su(2)}^3,  \]
where $\Phi_\gamma ([A]) = g$, $\Phi_1([A]) = -b_1$ and $\Phi_2([A]) = b_2$, if $A$ is a flat connection  of the form $g ds$, $b_1 ds$ and $b_2 ds$ in the neighborhoods of $\gamma$, $\beta_1$ and $\beta_2$ (the minus sign in $\Phi_1$ comes from the fact that we oriented $\beta_1$  as $\beta$, and not by the outward-pointing convention). In the representation-theoretic description of $\N(\Sigma_{cut} )$, this action can be expressed as:

\begin{multline*} (G,G_1 , G_2 ).(g, A_1 , A_2 , b_1 , b_2 , U_2 , V_2 , \cdots , U_h , V_h ) = \\ 
(ad_{G} g, G A_1 G_1^{-1} , G_2 A_2 G^{-1} , ad_{G_1} b_1 , ad_{G_2} b_2 ,G U_2 G^{-1},G V_2 G^{-1}, \cdots ).
\end{multline*}

In particular, the action of $SU(2)$ defined by $G.([A]) = (1,G,G).([A])$ is also Hamiltonian, with moment $\Phi  = \Phi_1 + \Phi_2$, and expression:

\begin{multline*}G.(g, A_1 , A_2 , b_1 , b_2 , U_2 , V_2 , \cdots , U_h , V_h ) = \\
(g, A_1 G^{-1} , G A_2 , ad_G b_1 , ad_G b_2 , U_2 , V_2 , \cdots , U_h , V_h ).\end{multline*}

Denote $\N(\Sigma_{cut} )\red SU(2) $ the  symplectic quotient for this action, and define a map $\N(\Sigma_{cut} )\red SU(2) \rightarrow \N(\Sigma)$ in the following  way: if $A$ is a connection on $\Sigma_{cut}$ such that $\Phi([A]) = 0$, then $b_1 = b_2$ and $A$  glue back to a connection on $\Sigma$. This defines a map $ \Phi^{-1} (0) \rightarrow \N(\Sigma)$. If $G\in SU(2)$, and $\varphi \in \G ^{const}(\Sigma_{cut})$ is a gauge transformation corresponding to $(1,G,G)$, $\varphi$ coincide on $\beta_1$ and $\beta_2$, and glue to a gauge transformation  of $\G^0(\Sigma)$. It then follows that  $A$ and $\varphi . A$ define the same element in $\N(\Sigma)$, i.e. the previous map descends to a map from the quotient  $\N(\Sigma_{cut} )\red SU(2) $ to $\N(\Sigma)$.

\begin{prop}\label{actioncut} In the holonomy description  of $\N(\Sigma_{cut} )$ and $\N(\Sigma)$, this map corresponds to:

\begin{multline*}[(g, A_1 , A_2 , b_1 , b_2 , U_2 , V_2 , \cdots , U_h , V_h )] \mapsto \\
(g, A = A_1 A_2 , \tilde{B} = A_2^{-1} e^{b_1} A_2 , U_2 , V_2 , \cdots , U_h , V_h ).\end{multline*}

Furthermore, this map realizes a symplectomorphism on its image \[\N(\Sigma' ) \setminus C_-,\] where $C_- = \lbrace\tilde{B} = -I\rbrace$.
 
\end{prop}

\begin{proof} The description comes from the fact that $\alpha  = \alpha_1 \alpha_2$ and $\tilde{\beta}  = \alpha_2^{-1} \beta_2 \alpha_2$.

The exponential  realizes a diffeomorphism between the ball \[\{b_1 \in \mathfrak{su(2)} \ |\ |b_1 | < \pi \sqrt{2} \}\] and $SU(2) \setminus \{-I\}$, we denote  $\log$ its inverse map. One can easily check that the inverse map is given by:
\begin{multline*}(g, A, \tilde{B}, U_2 , V_2 , \cdots , U_h , V_h ) \mapsto \\
 [(g, A_1 = A, A_2 = I, b_1 = \log(\tilde{B}), b_2 = b_1 , U_2 , V_2 , \cdots , U_h , V_h )]\end{multline*}
This proves that it is a bijection. Finally, this map preserves the symplectic forms, as both are defined in an analogous way, by integrating the forms on  $\Sigma$ and $\Sigma_{cut}$.
\end{proof}

\subsubsection{Description of a Dehn twist in the moduli spaces}\label{sec:descriptwist}

We start by describing the action of a  Dehn twist inside $\N(\Sigma_{cut} )$. For each $t \in [0, 1]$, denote $\tau_t$ the diffeomorphism of $ \Sigma_{cut}$ being the identity outside a neighborhood of the curve $ \beta_1$, and on $\nu \beta_1\simeq \rr/\zz \times  [0, 1]$, $\tau_t (s, x) = (s + t\psi (x), x)$, where $\psi: [0, 1] \rightarrow [0, 1]$ is a smooth function equal to  1 on $[0, \frac{1}{3}]$ and 0 on $[\frac{2}{3}, 1]$.

\begin{figure}[!h]
    \centering
    \def\svgwidth{0.65\textwidth}
    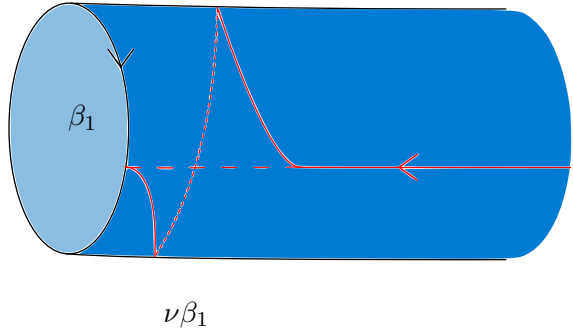
      \caption{The twist $\tau_1$ in the neighborhood of $\beta_1$.}
      \label{twist}
\end{figure}

Only $\tau_0$ and $\tau_1$ glue to diffeomorphisms of $\Sigma$, corresponding respectively to the  identity and a Dehn twist along $\beta$. Denote then the pullback $\varphi_t = \tau_t^*: \N(\Sigma_{cut} ) \rightarrow \N(\Sigma_{cut} )$, defined by $\varphi_t ([A]) = [\tau_t^* A]$. 

\begin{prop}\label{expressionflot}
\begin{itemize}

\item[$(i)$] In the holonomy description of $\N(\Sigma_{cut} )$,  the pullback $\varphi_t$ corresponds to:

\begin{multline*}\varphi_t (g, A_1 , A_2 , b_1 , b_2 , U_2 , V_2 , \cdots , U_h , V_h ) = \\
(g, A_1 e^{tb_1} , A_2 , b_1 , b_2 , U_2 , V_2 , \cdots , U_h , V_h ).\end{multline*}

\item[$(ii)$] For each $t\in [0,1]$, $\varphi_t$ is the time $t$ Hamiltonian flow    of the function $H: \N(\Sigma_{cut} ) \rightarrow \rr$ defined by $H([A]) =  \frac{1}{2}|\Phi_1([A]) |^2$.

\end{itemize}
\end{prop}

In order to  prove  the proposition, recall  the following fact:

\begin{lemma}\label{gradcompo} Let $G$ be a Lie group, $\mathfrak{g}$ its Lie algebra, $(M, \omega , \Phi )$ a  $G$-Hamiltonian manifold ($\Phi: M \rightarrow \mathfrak{g} \simeq \mathfrak{g}^*$) and $f: \mathfrak{g} \rightarrow \rr$ a smooth function, then the symplectic gradient  of $f \circ  \Phi: M \rightarrow \rr$ is given by:
\[\bigtriangledown^\omega (f \circ  \Phi )_m = X_{\bigtriangledown f (\Phi (m))} (m)\]
where $\bigtriangledown f$ is the gradient of $f$ for the scalar product on $\mathfrak{g}$ realizing the isomorphism $\mathfrak{g} \simeq \mathfrak{g}^*$ and, for $\eta \in \mathfrak{g}$, $X_\eta$ stands for the vector field on $M$ corresponding to the infinitesimal action of $G$.
\end{lemma}

\begin{proof}[Proof of the lemma]
By definition, $\bigtriangledown^\omega (f \circ  \Phi )$ is such that, for $m \in M$ and $y \in T_m M$,
\[\omega_m (\bigtriangledown^\omega (f \circ  \Phi )_m , y) = \mathrm{D}_m (f \circ  \Phi ).y\]
hence,
\begin{align*}
\mathrm{D}_m (f \circ  \Phi ).y &= \mathrm{D}_{\Phi  (m)}f \circ \mathrm{D}_m \Phi .y\\
 &= \left\langle \bigtriangledown f (\Phi (m)), \mathrm{D}_m \Phi .y \right\rangle \\
 &= \mathrm{D}_m (f_m).y \\
 &= \omega_m (X_{\bigtriangledown f (\Phi (m))}(m), y),
\end{align*}
where $f_m$ is the function on $M$ defined for $m$ fixed by  \[f_m(m') = \langle \bigtriangledown f (\Phi (m)), \Phi (m') \rangle.\]

\end{proof}

\begin{proof}[Proof of Proposition~\ref{expressionflot}] $(i)$ As  $\tau_t$ corresponds to the identity in the neighborhoods of $\gamma$ and $\beta_2$, and to a  rotation in the  neighborhood of $\beta_1$, the values of $g$, $b_1$ and $b_2$ remain unchanged by $\varphi_t$. Moreover, $\tau_t$ doesn't change the curves $\alpha_2 , u_2 , \cdots , v_h$: the corresponding holonomies also remain unchanged.  Furthermore, it sends $\alpha_1$ to a curve homotopic to $\alpha_1 \cup \beta_1([0, t])$ , hence  \[\mathrm{Hol}_{\alpha_1} (\tau_t^* A) = \mathrm{Hol}_{\alpha_1 \cup \beta_1([0,t])} (A) = \mathrm{Hol}_{\alpha_1} (A) e^{tb_1}.\] 

$(ii)$ Note first that according to the previous point, \[\varphi_t ([A]) = (1,e^{-tb_1} , 1)[A] = (1,e^{t \Phi_1 ([A])} , 1)[A],\] for the action of $SU(2)^3$ previously defined.

Apply  Lemma~\ref{gradcompo} to $M=\N(\Sigma_{cut})$, endowed with  the action of $SU(2)$ with moment $\Phi_1$, with  $f (\xi ) = \frac{1}{2} |\xi |^2$. $\bigtriangledown f (\Phi_1 ([A])) = \Phi_1 ([A]) = - b_1 ([A])$. It follows from the first observation that $\frac{\partial \varphi}{\partial t} |_{t=0} = X_{\Phi_1 ([A])} ([A])$, and according to the lemma, $X_{\Phi_1 ([A])} ([A]) = X_{\bigtriangledown f (\Phi_1 ([A]))} ([A]) = \bigtriangledown^\omega H([A])$. The proof of $(ii)$ now follows  from this, and the fact that  $\varphi_t$ satisfies the flow property  $\varphi_{t+h} = \varphi_t \circ \varphi_h$. 
\end{proof}

Recall the following result:
\begin{prop}(\cite[Prop. 2.15]{WWtriangle})
Let $(M,\omega,\Phi)$ be an $SU(2)$-Hamil\-tonian manifold  such that the moment $\Phi \colon M \rightarrow \mathfrak{su(2)}^*$ takes its values in the ball $\lbrace \xi \in \mathfrak{su(2)}~|~ |\xi |<\pi \sqrt{2}\rbrace$, and such that the stabilizer of the action at each point of $\Phi^{-1}(0)$ is trivial (resp. $U(1)$). Let $\psi \in C^\infty ([0,+\infty))$ be such that $\psi ' (0) = \pi \sqrt{2} $, with  compact support, and such that the  time $1$  Hamiltonian flow of $\psi \circ |\Phi|$ extends smoothly to  $\Phi^{-1}(0)$.

Then  $\Phi^{-1}(0)$  is a codimension 3 (resp. 2) spherically fibered coisotropic  submanifold, and the time $1$ of the flow  of $\psi \circ |\Phi|$ is a fibered Dehn twist  along  $\Phi^{-1}(0)$.
\end{prop}

This proposition applies for the action of $SU(2)$ on $\N(\Sigma_{cut})$ with moment $\Phi_1$. Indeed, on the one hand $Im \Phi_1 \subset \lbrace |\xi |<\pi \sqrt{2}\rbrace$ by definition of $\N(\Sigma_{cut})$. On the other hand, according to the holonomy description, this action is free, and the flow extends. Hence:

\begin{cor} \label{cortwist} Let $R\colon \rr\rightarrow \rr$ be a function that vanishes for $t>\frac{\pi \sqrt{2}}{2}$ and such that $R(-t) = R(t) -2 \pi \sqrt{2} t$. Then the flow of $H = R\circ |\Phi_1|$ at time $1$ extends to a fibered Dehn twist of $\N(\Sigma_{cut})$ along $C_+ = \Phi_1^{-1}(0)$, which is a spherically fibered submanifold.

\end{cor}

Recall  now the following result in order to establish the result for $\N(\Sigma)$.
\begin{theo}(\cite[Theorem 2.10]{WWtriangle})\label{twistred}
 Let  $G$ be a Lie group, $(M,\omega, \Phi)$ a Hamiltonian $G$-manifold such that $0$ is a regular value of the moment $\Phi$. Let $C\subset M$ be a spherically fibered coisotropic submanifold over a base $B$ and stable under the action of $G$. Assume that $C$ intersects $\Phi^{-1}(0)$ transversely, and that, denoting $\Phi_B \colon B\rightarrow \mathfrak{g}$ the moment induced on $B$, the induced action on the base  $\Phi_B^{-1}(0)\subset B$ is free. Let $\tau_C \in Diff(M,\omega)$ be a fibered Dehn twist  along $C$ which is $G$-equivariant. 

Then, the induced symplectomorphism $[\tau_C]\colon M /\!\!/ G \rightarrow M /\!\!/ G $ is a fibered Dehn twist  along $C /\!\!/ G$.
\end{theo}
Consider $M = \N(\Sigma_{cut})$, endowed with the action of moment $\Phi =  \Phi_1 + \Phi_2$. The submanifold $C = \Phi_1^{-1}(0)$ is a spherically fibered coisotropic submanifold over \[B = \lbrace (g, A_2, b_2, \cdots)\rbrace \simeq \N(\Sigma_{cut, cap1}),\] where the surface $\Sigma_{cut, cap1}$ is obtained from  $\Sigma_{cut}$ by gluing a disc on the  boundary component  $\beta_1$. The time 1 of the Hamiltonian flow  of $R\circ |\Phi_1|$, where $R$ is a function as in the previous corollary, is a fibered Dehn twist $\tau_C$. 
One can apply Proposition~\ref{twistred} to this situation. Indeed, $\N(\Sigma_{cut})$ may be identified with the  following open subset of $\mathfrak{su(2)}^2\times SU(2)^{2h}$ consisting of the elements $ (b_1,b_2, A_2, A_2, U_2, V_2, \cdots U_h, V_h)$ satisfying 
\begin{align*}
 &\abs{b_1}<\pi \sqrt{2}, \ \abs{b_2}<\pi \sqrt{2},  \\ & A_1 e^{b_1} A_1^{-1} A_2^{-1} e^{b_2} A_2 \prod_{i=2}^{h}{[U_i , V_i ]} \neq -I .
\end{align*}
Under this identification, $\Phi$ and $\Phi_1$ correspond respectively to the difference of the two first coordinates and to the opposite of the projection onto the first coordinate. The zero vector $0\in \mathfrak{su(2)}$ is then a regular value for $\Phi$, and $C$ intersects $\Phi^{-1}(0)$ transversely along $\lbrace \Phi_1 = \Phi_2 = 0\rbrace$. Furthermore, the action induced on $\Phi_B^{-1}(0)$ is free (the holonomy $A_2$ is affected by left multiplication), and the twist  $\tau_C$ is $SU(2)$-equivariant, indeed $\tau_C$ has the following expression:
\[ \tau_C (g, A_1, A_2, b_1, b_2, \cdots) = (g, A_1 e^{-t b_1} ,  A_2,  b_1,  b_2, \cdots),\]
where $t =  R'(|b_1|)$. If $H\in SU(2)$ and $H.$ denotes the action with moment $\Phi$, one has:
\begin{align*}
\tau_C \left( H.(g, A_1, A_2, b_1, b_2, \cdots)\right)  &=  H.\tau_C(g, A_1, A_2, b_1, b_2, \cdots) \\
 &= (g, A_1 e^{-t b_1} H^{-1}, H A_2, ad_H b_1, ad_H b_2, \cdots).
\end{align*}

Hence it follows from corollary~\ref{cortwist} and Proposition~\ref{twistred}:
\begin{prop}\label{twistsigma} 
Let $R$ be as in  corollary~\ref{cortwist}. The time 1 Hamiltonian flow  of the function $R(|\log ( \tilde{B})|  )$ is a fibered Dehn twist of $\N(\Sigma)$ along $\lbrace \tilde{B} = I\rbrace$.
\end{prop}
Notice that when  $\Sigma$ has genus greater or equal to 2, the submanifold $\lbrace \tilde{B} = I\rbrace$ is not compact in $\N(\Sigma)$: its closure in $\Nc(\Sigma)$ intersects the hypersurface $R$. However,  if  $\Sigma$ has genus 1, it is contained in $\lbrace g = 0\rbrace$. Hence:
\begin{theo}
\label{twistinduit}
 Let $H$ be a solid torus,  $T$ its boundary torus, and $T'$ the surface obtained by removing a small disc. Denote $i\colon T \rightarrow \partial H$ the inclusion, and $L(H)\subset \N(T')$ the corresponding Lagrangian submanifold. Let $\tau_K$ be a Dehn twist along a non-separating curve $K \subset T'$, $i' = i\circ \tau_K$ and $H' = (H, i')$ the cobordism between $\emptyset$ and  $T'$, and $L(H')\subset \N(\Sigma)$. Then there exists a Dehn twist along $S = \lbrace \mathrm{Hol}_K = -I\rbrace$ which sends $L(H')$ to $L(H)$. 
\end{theo}

\begin{remark}The symplectomorphism induced from the twist on the surface isn't a priori a Dehn twist of $\N(T')$ as the Hamiltonian generating it isn't compactly supported, yet we will build a Dehn twist (which will be denoted $tw$) by truncating the Hamiltonian.
\end{remark}

\begin{proof} Recall that we have identified $\N(T')$ with the subset
\[\left\lbrace (g, A, \widetilde{B}) \in \mathfrak{su(2)}\times SU(2)^{2}: e^g = [A,\widetilde{B}] \right\rbrace ,\] 
where $A$ and $\widetilde{B}$ represent the holonomies along the paths $\alpha$ and $\widetilde{\beta}$.  Define three functions \[H^f,H^{tw},H^\tau \colon\N(T') \rightarrow \rr\] by: 
\begin{align*}
H^f  (A, \widetilde{B}) & = \frac{1}{2} |\log(\tilde{B})|^2,\text{ setting }|\log(-I)| = \pi \sqrt{2}, \\
H^{tw}  (A, \tilde{B}) & = \phi(A, \tilde{B}) H^f  (A, \tilde{B}), \\
H^\tau  (A, \tilde{B}) & = R(|\log (- \tilde{B})|  ),
\end{align*}
where $\phi$ is a compactly supported function equal to  1 in a neighborhood of $\lbrace g = 0\rbrace$, 
$R\colon \rr_+\rightarrow \rr$ is zero for $t> \frac{\pi \sqrt{2}}{2}$, and such that $R(t) = \pi^2 - \pi \sqrt{2} t + \frac{1}{2} t^2$ for $t< \frac{\pi \sqrt{2}}{4}$. 

These three functions coincide in the neighborhood of $\lbrace \tilde{B} = -I\rbrace$: this is clear for $H^f$ and $H^{tw}$ since $\lbrace \tilde{B} = -I\rbrace \subset \lbrace g = 0 \rbrace$, and if $-\tilde{B}$ is conjugated to $\begin{pmatrix}
e^{i\alpha}& 0 \\
0 & e^{-i \alpha}
\end{pmatrix}$,  
with $\alpha\in [0,\pi]$, then $\tilde{B}$ is conjugated to $\begin{pmatrix}
e^{i(\pi -\alpha)}& 0 \\
0 & e^{-i (\pi -\alpha)}
\end{pmatrix}$, 
and $\frac{1}{2} |\log ( \tilde{B}) |^2 = (\pi -\alpha)^2 = R(|\log (- \tilde{B})| )$, since $|\log (- \tilde{B})| = \alpha \sqrt{2}$. Hence, $H^{tw} = H^\tau$ in the neighborhood of $\lbrace \tilde{B} = -I\rbrace$.

By proposition ~\ref{expressionflot}, the time 1 flow of $H^f$ is induced by the geometric twist and extends smoothly to $\N(\Sigma)$, 
hence does the flows of $H^{tw}$ and $H^\tau$: denote then $f$, $tw$ and $\tau$ the  flows extended to $\N(\Sigma)$.

On the one hand, the set $ \lbrace g = 0 \rbrace$ is invariant by the flow  of $H^f$ for all time, it follows that $f$ and $tw$  coincide on it, and  $L(H') = tw(L(H))$, since  $L(H') = f(L(H))$ and $L(H)$ is contained in $ \lbrace g = 0 \rbrace$.

On the other hand, by Proposition~\ref{twistsigma}, $\tau$ is the inverse of a Dehn twist  along $\lbrace \tilde{B} = -I\rbrace$. Indeed, denoting $\varphi$ the involution \[(A, \tilde{B}) \mapsto (A, -\tilde{B})\] of $\N(T')$, the map $\varphi \tau \varphi^{-1}$ is a fibered Dehn twist along $\lbrace \tilde{B} = I\rbrace$.

Observe now that  $tw$ can be written as the composition $(tw\circ \tau^{-1})\circ\tau$, with  $\tau$ a Dehn twist along  $\lbrace \tilde{B} = -I\rbrace$, and  $tw\circ \tau^{-1}$  a compactly supported Hamiltonian isotopy. Indeed, outside $\lbrace \tilde{B} = -I\rbrace$,  $tw\circ \tau^{-1}$  is the time 1 flow of the Hamiltonian 

\[H^{comp}(t,x) = H^{tw}(x) - H^\tau( \phi_{tw}^{t} (x)  ),\]
where $\phi_{tw}^{t}$ is the time $t$ flow of $H^{tw}$. Yet, in the neighborhood of $\lbrace \tilde{B} = -I\rbrace$, $\phi_{tw}^{t}$ coincides with the flow of $H^\tau$, hence $H^{comp}(t,x) = H^{tw}(x) - H^\tau (x)$ in the neighborhood of $\lbrace \tilde{B} = -I\rbrace$, and $H^{comp}$ extends smoothly to $\lbrace \tilde{B} = -I\rbrace$.
\end{proof}

\subsubsection{Proof of the surgery exact sequence}\label{sec:proofsurgeryexactseq}

In this paragraph we prove Theorem~\ref{trianglechir}. 

\begin{proof}

 Let  $\alpha$, $\beta$ and $\gamma$ denote three curves in the punctured torus  $T' = \partial Y \setminus \lbrace\mathrm{small~disc}\rbrace$ forming a triad, one has  $\beta^{-1} = \tau_\alpha \gamma$, where $\tau_\alpha$ is a Dehn twist along $\alpha$. Hence, with
\begin{align*}  L_{\alpha} ^{-} &= \lbrace \mathrm{Hol}_\alpha = -I \rbrace, \\
 L_\beta &= \lbrace \mathrm{Hol}_\beta = I \rbrace,  \\
L_\gamma &= \lbrace \mathrm{Hol}_\gamma = I \rbrace
\end{align*} the three Lagrangian spheres  of $\Nc(T')$, it follows from Theorem~\ref{twistinduit} that there  exists a generalized Dehn twist $\tau_S$ of $\Nc(T')$ along $S = L_{\alpha} ^{-}$ which sends $L_\gamma$ to $L_\beta$. Indeed, let $H$  be the solid torus in which  $\beta^{-1}$ bounds a disc, and $i \colon T'\rightarrow \partial H$ the inclusion, one has $i(\beta^{-1}) = \partial D^2$. If $i' = i\circ \tau_\alpha$, one has  $ i'(\tau_\alpha ^{-1} \beta ^{-1}) = i'(\gamma)$.

With $\underline{L} = L(Y,c)$, $S = L_{\alpha} ^{-}$ and $L_0 = L_\beta$,  Theorem~\ref{quilttri} gives an exact sequence:
\[ \ldots \rightarrow HF(\tau_S L_0, \underline{L}) \rightarrow HF(L_0, \underline{L}) \rightarrow HF(L_0,S^T, S, \underline{L})\rightarrow \cdots . \]

It now remains to identify the HSI groups: the Lagrangians $L_\beta$ and $L_\gamma$ being associated to the cobordisms corresponding to a 2-handle attachment  along $\beta$ (resp. $\gamma$) and without homology class, it follows for the two first groups: \[HF(\tau_S L_0, \underline{L}) = HF(L_\beta, \underline{L}) = HSI(Y_\beta,  c_\beta),\] and \[ HF(L_0, \underline{L})= HF( L_\gamma, \underline{L}) = HSI(Y_\gamma,  c_\gamma).\] Finally, $S = L_{\alpha} ^{-}$ corresponds to a two-handle attachment along $\alpha$, with homology class  $k_\alpha$, it follows from the Künneth formula (Theorem~\ref{sommecnx}) and the fact $HF(L_0,S) = HSI(S^3) = \zz$ that: \[ HF(L_0,S^T, S, \underline{L}) = HF(L_0,S) \otimes_\zz HF( S, \underline{L}) = HSI(Y_\alpha, k_\alpha + c_\alpha),\] which completes the proof.

\end{proof}

\subsection{Applications of the exact sequence}\label{sec:applic}

In this section we give some direct applications  of  Theorem~\ref{trianglechir}. These do not require any further properties of the maps in the exact sequence, and follows from an observation due to Ozsv{\'a}th and Szab{\'o}. We start by recalling it, and give some families  of manifolds for which the HSI homology  is minimal. All these  manifolds are L-spaces in Heegaard Floer theory.

\subsubsection{The observation of Ozsv{\'a}th and  Szab{\'o} }\label{sec:OSzobs}The  following fact has been pointed out by Ozsv{\'a}th and Szab{\'o}, see for example \cite[Exercice~1.13]{OSzlectures}. It can be proven directly, or be deduced from the surgery  exact sequence (for $\widehat{HF}$ or $HSI$) by taking the Euler characteristic.

\begin{lemma} Let  $Y_\alpha$, $Y_\beta$ and $Y_\gamma$ be a surgery triad. If one denotes, for a  set $H$, the quantity:
\[ \abs{H} = \begin{cases} \mathrm{Card} H \text{ if $H$ is finite} \\ 0 \text{ otherwise,} \end{cases}\]
one has, up to a permutation of the manifolds, $ \abs{H_1(Y_\alpha;\zz)} = \abs{H_1(Y_\beta;\zz)} + \abs{H_1(Y_\gamma;\zz)} $.
\end{lemma}
\arnaque

Define HSI-minimal manifolds, which are analogs of Heegaard Floer L-spaces:
\begin{defi}A closed oriented 3-manifold $Y$ will be called \emph{HSI-minimal} if for each class $c\in H_1(Y;\Z{2})$,  $HSI(Y,c)$ is a free abelian group  of rank $\abs{H_1(Y;\zz)}$.
\end{defi}

\begin{remark}According to Proposition~\ref{genreun},  $S^2\times S^1$ isn't HSI-minimal, and the lens spaces are.
\end{remark}

It follows then from the surgery  exact sequence (Theorem~\ref{trianglechir})  and from the  formula for the Euler characteristic  (Proposition~\ref{eulercara}):

\begin{prop}\label{triadehsimin}Let $(Y_\alpha, Y_\beta, Y_\gamma)$ be a surgery triad, with  $Y_\beta$ and $Y_\gamma$  HSI-minimal, and $\abs{H_1(Y_\alpha;\zz)} = \abs{H_1(Y_\beta;\zz)} + \abs{H_1(Y_\gamma;\zz)}$. Then $Y_\alpha$ is also HSI-minimal.
\end{prop} 
\begin{proof}Let $c_\alpha \in H_1(Y_\alpha;\Z{2})$, and $c_\beta$, $c_\gamma$ two other classes on $Y_\beta$ and $Y_\gamma$ for which  Theorem~\ref{trianglechir} gives rise to an exact sequence between the three HSI homology groups. Assume by contradiction  that the arrow  between $HSI(Y_\beta,c_\beta)$ and $HSI(Y_\gamma,c_\gamma)$ is nonzero, then one would have \[\mathrm{rk}HSI(Y_\alpha,c_\alpha) < \mathrm{rk}HSI(Y_\beta,c_\beta) + \mathrm{rk}HSI(Y_\gamma,c_\gamma) = \chi ( HSI(Y_\alpha,c_\alpha) ), \]
which is impossible. Hence the  exact sequence is a short exact sequence, and  $HSI(Y_\alpha,c_\alpha)$ is a free abelian group of rank  $\abs{H_1(Y_\alpha;\zz)}$.

\end{proof}

\subsubsection{Some  families of HSI-minimal manifolds}\label{sec:hsiminimals}

We now give  some applications  of the former observation: 
\paragraph{\bf Plumbings}
Let $(G,m)$ denote a weighted  graph: $m$ is a function defined on the set of vertices of the graph $G$, with values in $\zz$. Recall  that one can associate to $(G,m)$ a 4-manifold with boundary obtained by plumbing disc bundles over spheres associated to the vertices, whose Euler number is $m(v)$. Its boundary is a closed oriented 3-manifold  $Y(G,m)$.

\begin{prop}\label{bonplombage}Suppose that $G$ is a  disjoint union of trees, and that, denoting $d(v)$ the number of edges adjascent to a vertex  $v$, the function $m$ satisfies, for every vertex  $v$ of $G$, $m(v)\geq d(v)$, with  at least one vertex for which the   inequality is strict. Then $Y(G,m)$ is HSI-minimal.
\end{prop}

\begin{remark} If  $m(v)= d(v)$ for all vertex of $G$, then one can show after a succession of blow-downs that $Y(G,m) \simeq S^2\times S^1$.
\end{remark}

\begin{proof} The proof is analog to the corresponding one for \cite[Theorem 7.1]{OSsymplectic4}: one proceeds by induction on the number of vertices and on the weights. First, if the graph $G$ consists of a single vertex, then $Y(G,m)$ is a  lens space, and the result follows from Proposition~\ref{genreun}.

We prove now the induction on the number of vertices.  Adding a leaf $v$ with  $m(v) = 1$ corresponds to a blow-up, and doesn't change the topological type  of $Y(G,m)$.

We finally prove  the induction on the weight of a leaf. Let  $(G,m)$ be a  graph satisfying the   hypothesis of the proposition, $v$ a leaf   of $G$, $G'$ the graph  obtained by removing $v$, $m'$ the restriction of  $m$ to $G'$, and $\tilde{m}$ the function agreeing with  $m$ outside $v$, and such that $\tilde{m}(v) = m(v)+1$. Suppose that  $(G,m)$ and $(G',m')$ satisfy the induction hypothesis.

The manifolds $Y(G,\tilde{m})$, $Y(G,m)$ and $Y(G',m')$ form a surgery triad, and $| H_1(Y(G,\tilde{m});\zz)| = | H_1(Y(G,m);\zz)| + | H_1(Y(G',m');\zz)|$, see \cite[Proof of Th. 7.1]{OSsymplectic4}. Hence the induction follows from Proposition~\ref{triadehsimin}.
\end{proof}

\paragraph{\bf Branched double covers of $S^3$ along  quasi-alternating links}
In \cite[Def. 3.1]{OSzdoublecover}, Ozsv{\'a}th and  Szab{\'o}  defined the following class of links, called ``quasi-alternating'': it consists of the smallest class of links satisfying the following:

\begin{enumerate}
\item The trivial knot  is quasi-alternating,

\item   Let $L$ be a link. If there  exists a projection and a crossing  of $L$ such that its two resolutions are  quasi-alternating,  $ \mathrm{det} L_0, \mathrm{det} L_1 \neq 0$ , and $  \mathrm{det} L =\mathrm{det} L_0 + \mathrm{det} L_1 $, then $L$ is also quasi-alternating.
\end{enumerate}
According to  \cite[Lemma 3.2]{OSzdoublecover}, this class contains the links admitting a connected alternating projection. It follows directly from Proposition~\ref{triadehsimin}:

\begin{prop} The branched double covers of  quasi-alternating links are HSI-minimal manifolds.
\end{prop}
\arnaque 

\paragraph{\bf Integral Dehn surgeries along certain knots} Finally, let $K\subset S^3$ be a knot such that, for some integer $n_0>0$, the $n_0$-surgery  $S^3_{n_0}(K)$ is HSI-minimal. From the fact that for every  $n>0$,  $\abs{ H_1(S^3_{n}(K) ,\Z{2}) } = n$, it follows that $S^3_{n}(K)$ is HSI-minimal for every $n>n_0$.

\bibliographystyle{alpha}
\bibliography{biblio}

\address{Laboratoire de Math\'ematiques Jean Leray \\Université de Nantes \\2, rue de la Houssinière - BP 92208 F-44322 Nantes Cedex 3 \\ France}

\email{guillem.cazassus@univ-nantes.fr}

\end{document}